\documentclass[a4paper,11pt,reqno,noindent]{amsart}
\usepackage[centertags]{amsmath}
\usepackage{mathtools} 
\usepackage{amsfonts,amssymb,amsthm,dsfont,cases,amscd,esint,enumerate}
\usepackage[T1]{fontenc}
\usepackage[english]{babel}
\usepackage[applemac]{inputenc}
\usepackage{newlfont,cancel}
\usepackage{color}
\usepackage[body={15cm,21.5cm},centering]{geometry} 
\usepackage{fancyhdr}
\usepackage{enumerate} 

\theoremstyle{plain}

\newtheorem{conj1}{Conjecture}
\newtheorem{lem}{Lemma}[section]
\newtheorem{theor}[lem]{Theorem}
\newtheorem{prop}[lem]{Proposition}

\newtheorem{rem}[lem]{Remark}

\numberwithin{equation}{section}

\newcommand{\e}{\varepsilon}
\newcommand{\R}{\mathbb R}
\newcommand{\Z}{\mathbb Z}

\newcommand{\C}{\mathbb C}

\newcommand{\T}{\mathbb T}
\newcommand{\Pc}{\mathcal P}
\newcommand{\Dc}{\mathcal D}
\newcommand{\Cc}{\mathcal C}
\newcommand{\Lc}{\mathcal L}
\newcommand{\Hc}{\mathcal H}

\newcommand{\Sp}{\mathbb S}
\newcommand{\Id}{\operatorname{Id}}

\newcommand{\sym}{\operatorname{sym}}

\newcommand{\loc}{{\operatorname{loc}}}
\newcommand{\Div}{{\operatorname{div}}}

\newcommand{\3}{|\!|\!|}
\newcommand{\Ld}{\operatorname{L}}
\newcommand{\step}[1]{\noindent \textit{Step} #1.}
\newcommand{\substep}[1]{\noindent \textit{Substep} #1.}

\newcommand{\cvf}{\rightharpoonup}

\usepackage[colorlinks,linkcolor=black,citecolor=black,urlcolor=black]{hyperref}

\title[Dynamics of {point-vortex type} systems near equilibrium]{Dynamics of {point-vortex type} systems near thermal equilibrium: relaxation or not?}

\author[M. Duerinckx]{Mitia Duerinckx}
\address[Mitia Duerinckx]{Universit\'e Libre de Bruxelles, D\'epartement de Math\'ematique, 1050~Brussels, Belgium}
\email{mitia.duerinckx@ulb.be}
\author[P.-E. Jabin]{Pierre-Emmanuel Jabin}
\address[Pierre-Emmanuel Jabin]{Penn State University, Department of Mathematics, State College, PA 16802, USA}
\email{pejabin@psu.edu}

\begin{document}
\selectlanguage{english}

\begin{abstract}
This article is devoted to the long-time dynamics of point-vortex systems near thermal equilibrium and to the possible emergence of collisional relaxation. More precisely, we consider a tagged particle coupled to a large number of background particles that are initially at equilibrium, and we analyze its resulting slow dynamics. On the one hand, in the spirit of the Lenard--Balescu relaxation for plasmas, we establish in a generic setting the outset of the slow thermalization of the tagged particle. On the other hand, we show that a completely different phenomenology is also possible in some degenerate regime: the slow dynamics of the tagged particle then remains conservative and the thermalization no longer holds in a strict sense. We provide the first detailed description of this degenerate regime and of its mixing properties. Note that it is particularly delicate to handle due to statistical closure problems, which manifest themselves as a lack of self-adjointness of the effective Hamiltonian.
\end{abstract}

\maketitle
\setcounter{tocdepth}{1}

\tableofcontents

\section{Introduction}

\subsection{General overview}
This article is devoted to the rigorous analysis of the long-time dynamics of {point-vortex type systems} near equilibrium. More precisely, we consider a tagged particle coupled to a large number of background particles that are initially at thermal equilibrium, and we analyze the resulting slow dynamics of the tagged particle.
According to the physics literature, two types of behavior can occur:
\begin{enumerate}[---]
\item {\it Non-degenerate case: {thermalization.}}\\
Due to the slow correlation with the equilibrium background, the tagged particle is generically expected to thermalize on the slow timescale $t=O(N)$ --- proportional to the total number~$N$ of background particles. {Here, by thermalization, we mean relaxation to equilibrium in the strong sense of an H-theorem.} This was first described in the physics literature by Chavanis~\cite{Chavanis-98,Chavanis-01,Chavanis-12a,Chavanis-12b,Chavanis-23}, under the name ``point-vortex diffusion'', and it can be viewed as the equivalent for point-vortex type systems of the celebrated Lenard--Balescu collisional relaxation for plasmas~\cite{Balescu-60,Lenard-60,Guernsey-62,TGM-64}.
In the present contribution, following the line of a previous joint work of the first author with Saint--Raymond~\cite{DSR-21} on the Lenard--Balescu theory (see also~\cite{D-21a}), and in spite of new difficulties for point-vortex type systems, we establish a partial result in this direction, proving the outset of thermalization on some intermediate timescale.
\smallskip\item {\it Degenerate case: {no thermalization, but some ``resonant relaxation''.}}\\
In some situations, depending on the precise shape of the equilibrium, thermalization can fail. The slow dynamics of the tagged particle is then expected to be of a different nature and to take place instead on the shorter timescale $t=O(N^{1/2})$.
The simplest instance of this degenerate behavior occurs for uniformly distributed point-vortex type systems in a compact space. It seems to have been first discovered in~\cite{Taylor-McNamara-71} in the context of plasmas in a strong external magnetic field. Another application concerns the stellar dynamics in nuclear clusters dominated by a supermassive black hole: the slow motion of orbital planes in that setting is given by a similar degenerate dynamics, which is known as ``vector resonant relaxation''~\cite{Rauch-Tremaine-96,Kocsis-Tremaine-15,Fouvry-BarOr-Chavanis-19,Magnan-Fouvry-Pichon-Chavanis-22}.
Systematic theoretical studies are however lacking in the physics literature: in particular, due to statistical closure problems in link with turbulence~\cite{Krommes-02}, no effective equation is known to describe the degenerate dynamics of the tagged particle, and its properties are poorly understood.
In the present contribution, {at least in the special case of uniform equilibrium,} we clarify this situation by establishing for the first time a well-posed effective description on the critical timescale $t=O(N^{1/2})$ (although with an effective Hamiltonian lacking self-adjointness). {In particular, this confirms that the dynamics remains conservative and that thermalization must fail.
From our effective description, we also deduce a RAGE theorem, which describes the {\it weak} relaxation of the system towards equilibrium for~$t\gg N^{1/2}$, thus giving substance to the notion of ``resonant relaxation'' from the physics literature.}
\end{enumerate}
{As we explain in Section~\ref{sec:formal} below, a formal BBGKY analysis provides a systematic heuristics for this duality of behaviors. More precisely, we show that it
is, in fact, determined by the spectral nature of the linearized mean-field operator at thermal equilibrium: thermalization should be expected when this operator has purely continuous spectrum close to~$0$ orthogonally to its kernel, and resonant relaxation when it has instead eigenvalues accumulating at~$0$.}
In the case of plasmas without external magnetic field, the linearized mean-field operator at thermal equilibrium is always continuous orthogonally to its kernel by linear Landau damping, which is why Lenard--Balescu thermalization is indeed always expected to hold in that case, see e.g.~\cite{DSR-21}. {In contrast, point-vortex type systems have a richer phenomenology and we shall see that thermalization is expected to fail exactly in the case of Gaussian equilibrium.}

\subsection{Formal BBGKY analysis}\label{sec:formal}
In this section, we explain how the above duality of behaviors can be explained by a formal BBGKY analysis, depending on the spectral nature of linearized mean-field operators, and in all cases we formally derive effective equations for the slow dynamics of the tagged particle on the critical timescale. {We emphasize that this formal analysis is new to our knowledge in the degenerate case, as is the effective description of the dynamics that we obtain in that case.

Consider a tagged particle in a point-vortex type system, or more generally in a conservative long-range interacting particle system, at thermal equilibrium.}
The starting point is the BBGKY hierarchy of equations for the tagged particle density~$f_N^1$ coupled to the  multi-particle correlation functions describing fine correlations with background particles.
In terms of the mean-field equilibrium density~$\mu_\beta$, we consider the ratio $g_N^1:=f_N^1/\mu_\beta$, and we denote by~$\{g_N^m\}_{2\le m\le N}$ multi-particle correlations (see Section~\ref{sec:correlations} for precise definitions).
We shall view $g_N^1$ as belonging to the weighted space $\Ld^2(\mu_\beta)$, while for all $m\ge2$ the correlation~$g_N^m$ is defined in
\[\Hc^m\,:=\,\Ld^2(\mu_\beta)\otimes\Ld^2(\mu_\beta)^{\otimes_s (m-1)},\]
where the first factor stands for the tagged particle density and where $\otimes_s$ stands for the symmetrized tensor product for exchangeable background particles.
{In order to investigate the slow correction around mean field for the tagged particle, it is most convenient to consider a framework in which the mean-field evolution is trivial.}
The BBGKY hierarchy then typically takes on the following guise (see Lemma~\ref{lem:eqn-gNm-21} below for a more precise statement),\footnote{We emphasize the minor differences with the actual BBGKY hierarchy derived in Lemma~\ref{lem:eqn-gNm-21}: First, the actual equation for~$g_N^m$ may further involve $g_N^{m-2}$ on top of $g_N^{m\pm1}$. Second, the operators $\{S^{m,\pm}\}_m$ should further depend on $N$ and only satisfy the stated symmetry relation in the limit $N\uparrow\infty$. Finally, the skew-adjointness of $iL^m$ should only hold up to suitably deforming the underlying Hilbert structure. We neglect these issues here as they do not substantially affect the present formal discussion.}
\begin{equation}\label{eq:BBGKY-cor-gen}
\left\{\begin{array}{lll}
\partial_tg_N^1=iS^{1,+}g_N^{2},\\
(\partial_t+iL^m)g_N^m=iS^{m,+}g_N^{m+1}+\tfrac1NiS^{m,-}g_N^{m-1}&:&2\le m\le N,
\end{array}\right.
\end{equation}
where $iL^m:\Hc^m\to\Hc^m$ stands for the (skew-adjoint) $m$-particle linearized mean-field operator and where
operators $S^{m,+}:\Hc^{m+1}\to \Hc^{ m}$ and $S^{m,-}:\Hc^{m-1}\to \Hc^{ m}$ satisfy the symmetry relation
\begin{equation}\label{eq:S-sym-intro}
(S^{m,+})^*=S^{m+1,-},\qquad\text{for all $m\ge1$.}
\end{equation}
As the correlation function $g_N^2$ is expected to be small as $N\gg1$, we are led to a formal timescale separation in~\eqref{eq:BBGKY-cor-gen}: the tagged particle density $g_N^1$ has a slow dynamics, which is coupled to a fast linear subdynamics for correlations $\{g_N^m\}_{2\le m\le N}$.
Heuristically, we then expect that the latter can be relaxed on the slow timescale of the tagged particle. As the fast subdynamics is driven by linearized mean-field operators $\{iL^m\}_m$, its relaxation depends on the spectral properties of the latter. We distinguish between two cases.
\begin{enumerate}[---]
\medskip\item \emph{Non-degenerate case.}\\
Assume that the operators $\{iL^m\}_m$ have purely absolutely continuous spectrum in a neighborhood of $0$.\footnote{As we shall see, the operators $\{iL^m\}_m$ typically have a nontrivial kernel and have purely absolutely continuous spectrum only on the orthogonal complement. Provided that projections on orthogonal complements can be smuggled in the hierarchy~\eqref{eq:BBGKY-cor-gen}, this does not affect the present formal discussion.}
In that case, long-time propagators satisfy the following relaxation property:
\begin{equation}\label{eq:relax-ac}
(\e\partial_t+iL^m)h_\e=r,\quad h_\e|_{t=0}=0\qquad\implies\qquad h_\e\xrightarrow{\e\downarrow0}(0+iL^m)^{-1}r.
\end{equation}
Using this and assuming that correlations vanish initially, a formal analysis of the hierarchy~\eqref{eq:BBGKY-cor-gen} leads to expect $g_N^{m+1}=O(N^{-m})$ uniformly in time for all $m$. The equation for the tagged particle density then yields $\partial_tg_N^1=iS^{1,+}g_N^2=O(N^{-1})$, showing that the natural timescale for its evolution is $t=O(N)$. This leads us to considering the following critically-rescaled quantities,
\[\bar g_N^1(\tau)\,:=\,g_N^1(N\tau),\qquad\bar g_N^{m+1}(\tau)\,:=\,N^{m}g_N^{m+1}(N\tau),\qquad 1\le m< N.\]
In these terms, the rescaled hierarchy gets formally truncated into a closed system,
\[\left\{\begin{array}{lll}
\partial_\tau\bar g_N^1=iS^{1,+}\bar g_N^{2},\\
(\tfrac1N\partial_\tau+iL^2)\bar g_N^2=iS^{2,-}\bar g_N^{1}+O(\tfrac1N).
\end{array}\right.\]
From the relaxation property~\eqref{eq:relax-ac} for $iL^2$, the pair $(\bar g_N^1,\bar g_N^2)$ is then expected to converge to the solution $(\bar g^1,\bar g^2)$ of the limit system
\[\left\{\begin{array}{lll}
\partial_\tau\bar g^1=iS^{1,+}\bar g^{2},\\
\bar g^2=(0+iL^2)^{-1}iS^{2,-}\bar g^{1}.
\end{array}\right.\]
This can be written as a closed equation for $\bar g^1$,
\begin{equation}\label{eq:diff-gen}
\partial_\tau\bar g^1+S^{1,+}(0+iL^2)^{-1}(S^{1,+})^*\bar g^{1}=0,
\end{equation}
which should be thought of as a Fokker--Planck equation: {formally, it yields an H-theorem in form of
\[\partial_\tau\|\bar g^1\|^2_{\Ld^2(\mu_\beta)}\,=\,-2\pi\big\langle (S^{1,+})^*\bar g^1,\delta(L^2)(S^{1,+})^*\bar g^1\big\rangle_{\Ld^2(\mu_\beta)}\,\le\,0,\]
thus describing the thermalization of the tagged particle for $\tau\gg1$ (i.e.,~$t\gg N$).}

\medskip\item \emph{Degenerate case.}\\
Assume that the linearized mean-field operators $\{iL^m\}_m$ are compact: in that case, they have eigenvalues accumulating at $0$ and the resolvent $(0+iL^2)^{-1}$ in~\eqref{eq:diff-gen} no longer makes sense. Long-time propagators then have a completely different scaling and limiting behavior: instead of~\eqref{eq:relax-ac}, we have
\begin{equation}\label{eq:norelax-comp}
\e\partial_th_\e+iL^mh_\e=\e r,\quad h_\e|_{t=0}=0\qquad\implies\qquad \partial_th_\e\xrightarrow{\e\downarrow0}\pi_mr,
\end{equation}
where $\pi_m$ stands for the orthogonal projection onto the kernel of $iL^m$ (we also set \mbox{$\pi_1:=\Id$} for notational convenience below).
Using this, a formal analysis of the hierarchy~\eqref{eq:BBGKY-cor-gen} rather leads to expect $g_N^{m+1}=O(N^{-m/2})$ uniformly in time for all~$m$, so that the natural timescale for the evolution of the tagged particle density is instead~$t=O( N^{1/2})$. This leads us to consider the following critically-rescaled quantities,
\[\bar g_N^1(\tau):=g_N^1(N^\frac12\tau),\qquad\bar g_N^{m+1}(\tau):=N^{\frac m2}g_N^{m+1}(N^\frac12\tau).\]
In these terms, the rescaled hierarchy is no longer truncated into a finite closed system, and we find instead
\[\left\{\begin{array}{lll}
\partial_\tau\bar g_N^1=iS^{1,+}\bar g_N^{2},\\
(\partial_\tau+N^\frac12iL_m)\bar g_N^m=iS^{m,+}\bar g_N^{m+1}+iS^{m,-}\bar g_N^{m-1}&:&2\le m\le N.
\end{array}\right.\]
From~\eqref{eq:norelax-comp}, we deduce that the correlations $\{\bar g_N^m\}_{1\le m\le N}$ should converge to a weak solution~$\{\bar g^m\}_{m\ge1}$ of the following infinite hierarchy,
\begin{equation}\label{eq:compact-lim-hier}
\left\{\begin{array}{lll}
\partial_\tau\bar g^1=i\hat S^{1,+}\bar g^{2},\\
\partial_\tau\bar g^m=i\hat S^{m,+}\bar g^{m+1}+i\hat S^{m,-}\bar g^{m-1}&:&m\ge2,
\end{array}\right.
\end{equation}
where we have set $\hat S^{m,\pm}:=\pi_mS^{m,\pm}\pi_{m\pm1}$ for all $m$. Noting that $(\hat S^{m,+})^*=\hat S^{m+1,-}$, cf.~\eqref{eq:S-sym-intro}, this hierarchy constitutes a (formally) unitary evolution for the limiting tagged particle density~$\bar g^1$ coupled to the infinite collection of limiting correlations $\{\bar g^m\}_{m\ge2}$
on the critical timescale $\tau\sim1$ (i.e., $t\sim N^{1/2}$).
{Viewing correlations with the background as a quantized field, we consider $\bar g:=(\bar g^m)_{m\ge1}$ as an element of
\[\overline{\textstyle\bigoplus_{m\ge1}\Hc^m}=\Ld^2(\mu_\beta)\otimes\mathcal F_+(\Ld^2(\mu_\beta)),\]
where the symmetric Fock space $\mathcal F_+(\Ld^2(\mu_\beta))$ is the state space for background correlations,
and the hierarchy~\eqref{eq:compact-lim-hier} is then understood as a (formally) unitary dynamics,
\[\partial_\tau\bar g=i\hat S\bar g,\]
with effective Hamiltonian
\[\hat S\,:=\,\textstyle\bigoplus_{m\ge1}(\hat S^{m-1,+}+\hat S^{m+1,-})|_{\mathcal H^m}.\]
Given that this effective dynamics is (formally) unitary, hence conservative, thermalization cannot hold in the sense of an H-theorem on the critical timescale $\tau\sim1$ (i.e.,~$t\sim N^{1/2}$).
As we shall see in Section~\ref{sec:unif}, we actually expect the effective Hamiltonian $\hat S$ to lack self-adjointness in link with statistical closure problems, but the well-posedness of the dynamics can nevertheless be proven in some suitable sense.

For the effective dynamics~\eqref{eq:compact-lim-hier}, despite the lack of self-adjointness, we can still prove a version of a RAGE theorem: in the absence of periodic solutions, this describes the relaxation of the system to equilibrium in a weak sense as $\tau\gg1$ (i.e., $t\gg N^{1/2}$).
This weak relaxation can be identified with the so-called ``resonant relaxation'' in the physics literature~\cite{Rauch-Tremaine-96,Kocsis-Tremaine-15,Fouvry-BarOr-Chavanis-19,Magnan-Fouvry-Pichon-Chavanis-22}.}

Some intuition on the system can also be gained by noting that on {subcritical} times~\mbox{$\tau\ll1$} (i.e., $t\ll N^{1/2}$) the infinite hierarchy~\eqref{eq:compact-lim-hier} can be truncated and reduces to a linear wave-type equation for the tagged particle density,
\begin{equation}\label{eq:compact-wave}
\partial_\tau^2\bar g^1+\hat S^{1,+}(\hat S^{1,+})^*\bar g^1=O(\tau^2).
\end{equation}
\end{enumerate}

\smallskip\noindent
{Combining the different cases, we are led to formulate the following general conjecture for the dynamics of the tagged particle. We emphasize that we focus on a setting where the mean-field evolution of the tagged particle is trivial, so that the slow correction due to correlations with background particles is the leading dynamics: this crucially ensures a clean separation of timescales in the BBGKY hierarchy~\eqref{eq:BBGKY-cor-gen} and the above formal analysis does not seem easily extendable otherwise.
}

\begin{samepage}\begin{conj1}\label{conj}
Consider the slow dynamics of a tagged particle in a conservative long-range interacting particle system at thermal equilibrium, {in a setting where the mean-field evolution of the tagged particle density is trivial.}
\begin{enumerate}[(i)]
\item \emph{Non-degenerate case {--- thermalization}:}\\
If linearized mean-field operators at mean-field equilibrium have purely absolutely continuous spectrum close to $0$ (orthogonally to a possibly nontrivial kernel), then thermalization of the tagged particle occurs on the slow timescale $t=O(N)$ and is described by a Fokker--Planck type equation.
\smallskip\item \emph{Degenerate case {--- resonant relaxation}:}\\
If linearized mean-field operators at mean-field equilibrium have eigenvalues accumulating at $0$ (e.g., if these operators are compact), then the slow dynamics of the tagged particle rather takes place on the shorter timescale $t=O(N^{1/2})$ and thermalization fails. More precisely, the dynamics takes the form of a well-posed {conservative} hierarchical evolution for the tagged particle density coupled to the infinite collection of limiting correlation functions, {which implies a weak relaxation to equilibrium for~$t\gg N^{1/2}$.}
\end{enumerate}
\end{conj1}\end{samepage}

\subsection{{Point-vortex type} systems}
To illustrate the above conjecture, we focus on the example of 2D {point-vortex type systems}.
More precisely, in the $2$-dimensional plane~$\R^2$, we consider an interaction force kernel~$K$ and an external force field $F$ that satisfy the incompressibility conditions $\Div(K)=\Div(F)=0$ and the action-reaction condition $K(-x)=-K(x)$. The incompressibility allows to represent $K=-\nabla^\bot W$ and $F=-\nabla^\bot V$ for some potential fields~$V,W$, and we assume that they satisfy the following smoothness conditions,
\[W\in C^\infty_b(\R^2),\qquad V\in C^\infty_\loc(\R^2).\]
We consider the associated point-vortex dynamics
\[\partial_tx_i=F(x_i)+\tfrac1N\sum_{j=1}^NK(x_i-x_j),\qquad x_i|_{t=0}=x_i^\circ,\qquad1\le i\le N.\]
Equivalently, the Liouville equation for the $N$-point density $f_N$ reads
\begin{equation}\label{eq:Liouville-2}
\partial_tf_N+\sum_{i=1}^N\Big(F(x_i)+\tfrac1N\sum_{j=1}^NK(x_i-x_j)\Big)\cdot\nabla_if_N=0,\qquad f_N|_{t=0}=f_N^\circ.
\end{equation}
At inverse temperature $\beta>0$, the Gibbs thermal equilibrium measure for this dynamics is given by
\begin{equation}\label{eq:Gibbs}
M_{N,\beta}(x_1,\ldots,x_N)\,:=\,Z_{N,\beta}^{-1}\,\exp\bigg[-\beta\bigg(\sum_{i=1}^NV(x_i)+\tfrac1{2N}\sum_{i,j=1}^NW(x_i-x_j)\bigg)\bigg],
\end{equation}
where the constant $Z_{N,\beta}>0$ ensures $\int_{(\R^2)^N}M_{N,\beta}=1$.
{At high temperature, or more precisely} if $\beta\|W\|_{\Ld^\infty(\R^2)}$ is small enough, there is a unique associated mean-field invariant measure $\mu_\beta$ defined as the solution of the fixed-point equation
\begin{equation}\label{eq:mubeta}
\mu_\beta\,=\,Z_\beta^{-1}\,\exp\big(-\beta (V+ W\ast\mu_\beta)\big),
\end{equation}
where the constant $Z_{\beta}>0$ ensures $\int_{\R^2}\mu_{\beta}=1$.
The first marginal of the Gibbs measure~$M_{N,\beta}$ is then known to converge precisely to $\mu_\beta$ as $N\uparrow\infty$, see e.g.~\cite{Benarous-Brunaud-90}. Alternatively, recall that~$\mu_\beta$ can be constructed as the unique minimizer of the free energy functional
\[\rho~\mapsto~\int_{\R^2} \rho \log \rho +\beta \int_{\R^2} \big(V+\tfrac12 W\ast \rho\big)\rho.\]
Now consider a tagged particle (labeled `1') in a Gibbs equilibrium background: in other words, we assume that initially the $N$-point density $f_N|_{t=0}=f_N^\circ$ takes on the following guise,
\begin{equation}\label{eq:fN0-2}
f_N^\circ(x_1,\ldots,x_N)\,=\,f^{\circ}(x_1)\,\tilde M_{N,\beta}(x_2,\ldots,x_N),
\end{equation}
for some $f^{\circ}\in\Pc\cap C_c^\infty(\R^2)$,
where $\tilde M_{N,\beta}$ is the restricted Gibbs measure for the background particles,
\[\tilde M_{N,\beta}(x_2,\ldots,x_N)\,:=\,\tilde Z_{N,\beta}^{-1}\exp\bigg[-\beta\bigg(\sum_{i=2}^NV(x_i)+\tfrac1{2N}\sum_{i,j=2}^NW(x_i-x_j)\bigg)\bigg],\]
where the constant $\tilde Z_{N,\beta}>0$ ensures $\int_{(\R^2)^{N-1}}\tilde M_{N,\beta}=1$.
At later times, the tagged particle density is given by the first marginal
\[f_N^1(t,x_1)\,:=\,\int_{(\R^2)^{N-1}}f_N(t,x_1,x_2,\ldots,x_N)\,dx_2\ldots dx_N.\]

{In order to investigate the slow correction around mean field, just as in Conjecture~\ref{conj}, it is important to focus on a framework in which the mean-field evolution of the tagged particle is trivial. For that purpose, we shall focus on the {\it axisymmetric setting}, that is, we assume that the potentials $V,W$ and the initial density $f^{\circ}$ of the tagged particle are radial functions. The mean-field description of the tagged particle is then indeed trivial, $f_N^1(t)\approx f^\circ$ for $t=O(1)$, and the slow correction around mean field becomes the leading nontrivial effect.}
In this axisymmetric setting, it is most natural to focus on the radial density of the tagged particle,
\[\langle f_N^1\rangle(t,r)\,:=\,\fint_{\Sp^1}f_N^1(t,re)\,d\sigma(e),\]
where we use polar coordinates~$x=re$.
For a radial function $h$, we shall use the notation $h'(r)=\partial_rh(r)=\tfrac{x}{|x|}\cdot\nabla h(x)$ for the radial derivative.

{Note that Conjecture~\ref{conj} is formulated for any conservative long-range particle system and is in particular expected to hold for point vortices for 2D Euler, that is, for the system~\eqref{eq:Liouville-2} with the Biot--Savart law $K\sim\nabla^\bot \log |x|$. Unfortunately, the analysis conducted in the sequel only applies to sufficiently smooth interaction kernels, and it does not appear feasible at the moment to extend it anyhow to singular~$K$.}

\subsection{Main results}
{According to Conjecture~\ref{conj}, the behavior of the particle system~\eqref{eq:Liouville-2}--\eqref{eq:fN0-2} depends on the spectral nature of linearized mean-field operators at mean-field equilibrium.
As we shall see, the $m$-particle linearized operator takes the form
\[iL^m\,=\,\sum_{j=1}^m\Id^{\otimes j-1}\otimes iL_\beta\otimes\Id^{\otimes m-j}\,+\,\text{compact perturbations},\]
in terms of the single-particle operator
\[iL_\beta\,:=\,(F+K\ast\mu_\beta)\cdot\nabla\,=\,-\nabla^\bot(V+W\ast\mu_\beta)\cdot\nabla,\]
where $\mu_\beta$ stands for the mean-field equilibrium measure. As $V,W$ are radial, this reads in radial coordinates $iL_\beta=\Omega_\beta(r)\partial_\theta$ with angular velocity $\Omega_\beta(r)=-\frac1r\partial_r(V+W\ast\mu_\beta)$, which obviously has pure point spectrum if and only if the angular velocity is constant, that is, if $(V+W\ast\mu_\beta)(x)\equiv\alpha |x|^2$ for some $\alpha>0$. This condition is equivalent to~$\mu_\beta$ being Gaussian. 
For a non-Gaussian equilibrium measure, the linearized operators therefore do not have pure point spectrum, but one could in principle construct partly degenerate settings with mixed spectrum. This can be avoided under a suitable non-degeneracy condition on the angular velocity, and we focus on the following two cases:}
\begin{enumerate}[---]
\item If the mean-field equilibrium measure $\mu_\beta$ is not Gaussian, under some suitable non-degeneracy assumption, the linearized mean-field operators have purely absolutely continuous spectrum orthogonally to their kernel, cf.~Lemma~\ref{lem:nonGaus}. {By Conjecture~\ref{conj},} we then expect the tagged particle to display thermalization on the slow timescale $t=O(N)$, in line with the generic prediction of point-vortex diffusion.
\smallskip\item If the mean-field equilibrium measure $\mu_\beta$ is precisely Gaussian, that is, if it takes the form $\mu_\beta(x)\propto\exp(-C|x|^2)$ for some $C>0$,
the linearized mean-field operators are compact, cf.~Lemma~\ref{lem:Gaus}. {By Conjecture~\ref{conj}, we then expect thermalization to fail: the tagged particle evolution should take the form of a conservative hierarchical system on the different timescale $t=O(N^{1/2})$, which only describes a weak relaxation to equilibrium.}
\end{enumerate}
Below, we state our main results in both cases separately.

\subsubsection{Non-degenerate (non-Gaussian) setting}
The following result provides a description of the outset of thermalization on relatively short times~$t\ll N$ (for technical reasons, as in~\cite{DSR-21}, we further restrict to~$t\ll N^{1/20}$, which could be slightly improved by optimizing our analysis). Our current analysis does not allow to describe the tagged particle density on the thermalization timescale $t=O(N)$ itself, which is left as an open problem similarly as in~\cite{DSR-21} due to the possibility of uncontrolled echoes.
Yet, {for subcritical times $t\ll N^{1/20}$,} we justify the relevant form of the Fokker--Planck operator describing thermalization, cf.~\eqref{eq:FP-noGauss-0} below.
In contrast with~\cite{DSR-21}, the present situation is substantially more delicate since we do not have a closed formula for the resolvent of linearized mean-field operators. 
The proof is postponed to Section~\ref{sec:non-Gauss}.

\begin{theor}[Non-Gaussian setting]\label{th:main-nonGauss}
Assume that the external potential $V$ further satisfies $\nabla(V'/r)\in C^\infty_b(\R^2)$, {and that $\beta\ll1$ is small enough so that the mean-field equilibrium~$\mu_\beta$ is well-defined, cf.~\eqref{eq:mubeta}.}
Let the angular velocity $\Omega_\beta$ be the smooth radial function given by
\begin{equation}\label{eq:angular-velo-intro}
(\log\mu_\beta)'=\beta r\Omega_\beta.
\end{equation}
Consider the non-Gaussian setting when $\Omega_\beta$ is nowhere constant:
more precisely, we assume for simplicity that $\Omega_\beta$ is monotone and satisfies the following non-degeneracy condition, for some $R\in(0,\infty)$,
\begin{equation}\label{eq:nondegenerate-0}
|\Omega_\beta'(r)|\,\ge\,\tfrac1R(r\wedge1),\qquad|\Omega_\beta''(0)|\,\ge\,\tfrac1R,
\qquad\text{for all $r\ge0$},
\end{equation}
and we also assume that $\beta\ll_R1$ is small enough depending on~$V,W$ and on this constant~$R$.
Then, for any~\mbox{$\sigma\in(0,\frac1{20})$}, the subcritically-rescaled tagged particle density
\[\bar f_N^1(\tau)\,:=\,{f_N^1(N^\sigma\tau)}\]
satisfies in the radial distributional sense on $\R^+\times\R^2$,
\begin{equation}\label{eq:FP-noGauss-0}
{N^{1-\sigma}\partial_\tau\langle\bar f_N^1\rangle}~\xrightarrow{N\uparrow\infty}~\tfrac1r\partial_r \Big(ra_\beta(r)\big(\partial_r-(\log\mu_\beta)'(r) \big)f^\circ\Big),
\end{equation}
for some explicit positive scalar coefficient field $a_\beta$ (see Theorem~\ref{th:main-nonGauss-re}).
\end{theor}

The non-degeneracy assumption~\eqref{eq:nondegenerate-0} implies the radial monotonicity of the angular velocity~$\Omega_\beta$. According to~\cite{Chavanis-Lemou-07}, this is predicted to lead to a kinetic blocking, hence to the validity of (nonlinear) point-vortex diffusion only on the even slower timescale~\mbox{$t=O(N^2)$} for initially chaotic systems.
However, this blocking does not affect the validity of thermalization of a tagged particle on the timescale $t=O(N)$ as studied here.

Instead of the non-degeneracy assumption~\eqref{eq:nondegenerate-0}, we believe that our analysis should essentially hold true more generally whenever~$\Omega_\beta$ is a Morse function, but the analysis would become quite delicate close to critical points (see in particular the needed adaptation of Lemma~\ref{lem:nonGaus} in that case), and we skip it for brevity.

This kind of fast angular scale is also somewhat reminiscent of gyrokinetic approximations, especially in the case of finite Larmor radius as in~\cite{FrenodSonnen}; see also for example~\cite{Bostan}. 

{At low temperature, the mean-field equilibrium density $\mu_\beta$ might not be unique, hence the whole phenomenology of the system becomes completely different. For this reason, we do not expect the above result to hold at low temperature, that is, if $\beta$ is too large.}

\subsubsection{Degenerate (Gaussian) setting}
We turn to the special case when the equilibrium measure $\mu_\beta$ is Gaussian, that is, when potentials $V,W$ satisfy for some $R\in(0,\infty)$,
\begin{equation}\label{eq:Gaussian0}
(V+W\ast\mu_\beta)(x)=\tfrac12R|x|^2,\qquad\mu_\beta(x)=\tfrac{\beta R}{2\pi}e^{-\frac12\beta R|x|^2}.
\end{equation}
Note that for any given interaction potential $W$ we can always construct some external potential $V$ that leads to this special case.
In this setting, the angular velocity $\Omega_\beta$ defined in~\eqref{eq:angular-velo-intro} is constant and linearized mean-field operators reduce to compact operators, cf.~Lemma~\ref{lem:Gaus} below. By Conjecture~\ref{conj}, the thermalization of the tagged particle is thus expected to fail and to be replaced by a nontrivial conservative hierarchical dynamics on the shorter critical timescale $t=O(N^{1/2})$.
In accordance with~\eqref{eq:compact-wave}, we first show that {for subcritical times $t\ll N^{1/2}$ the tagged particle density can be approximated to leading order by the solution of a linear wave equation, which can be viewed as the first iterate of the expected conservative hierarchy and is in agreement with the expected lack of thermalization.}
The proof is postponed to Section~\ref{sec:Gauss}.

\begin{theor}[Gaussian case]\label{th:main-Gauss}
Assume that the interaction potential $W$ is nonzero and belongs to $\Ld^1(\R^2)$, {that $\beta\ll1$ is small enough so that the mean-field equilibrium $\mu_\beta$ is well-defined, cf.~\eqref{eq:mubeta},
and that the latter} is Gaussian in the sense of~\eqref{eq:Gaussian0} for some \mbox{$R\in(0,\infty)$}.
Then, for any~\mbox{$\sigma\in(0,\frac12)$}, the subcritically-rescaled tagged particle density
\[\bar f_N^1(\tau)\,:=\,{f_N^1(N^\sigma\tau)}\]
satisfies in the distributional sense on $\R^+\times\R^2$,
\begin{equation*}
{N^{1-2\sigma}\partial_\tau^2\bar f_N^{1}}
\,\xrightarrow{N\uparrow\infty}\,\Div( A\nabla f^\circ),
\end{equation*}
where the diffusion coefficient field $A$ is explicitly given by
\begin{equation}\label{eq:def-coeff-A-proj}
A(x)\,:=\,2\pi\int_{0}^\infty \Big(\fint_{\Sp^{1}} K(x-re)\,d\sigma(e)\Big)^{\otimes2}\,\mu_\beta(r)\,r\,dr.
\end{equation}
Note that the latter satisfies $A(0)=0$ and $A(Ox)=OA(x)O'$ for all $O\in O(2)$.
\end{theor}

{On the critical timescale $t=O(N^{1/2})$, we do not expect a close equation on $f_N^1$, but a hierarchy of equations coupling $f_N^1$ with all (rescaled) correlations. The wave equation above is just the first iterate and can only be a good approximation for $t\ll N^{1/2}$. The justification of the} expected hierarchical description~\eqref{eq:compact-lim-hier} on the critical timescale is not achievable at the moment by our techniques and is left as an open problem in general.
To get further in this direction, we focus on the extreme case $\beta=0$, which corresponds to the simplest setting of a tagged particle in a uniform equilibrium background \mbox{$\mu_0=$ cst,} say on the torus~$\T^2$. More precisely, instead of~\eqref{eq:Liouville-2}--\eqref{eq:fN0-2}, let us now consider a translation-invariant point-vortex type system on $\T^2$ with vanishing external force $F\equiv0$ and with initial condition $f_N^\circ(x_1,\ldots,x_N)=f^\circ(x_1)$ for some $f^\circ\in\Pc\cap C^\infty_b(\T^2)$.
In this uniform setting, linearized mean-field operators happen to vanish identically, which allows to push our rigorous analysis further and fully describe the slow dynamics of the tagged particle on the critical timescale $t=O(N^{1/2})$.
We refer to Section~\ref{sec:unif} for a more detailed statement and the proof.

\begin{theor}\label{th:main-unif}
Consider the above uniform setting, describing a tagged particle coupled to an initially uniformly distributed background on the torus $\T^2$ (see more precisely~\eqref{eq:Liouville}--\eqref{eq:initial-unif} below).
Then, the critically-rescaled tagged particle density
\[\bar f^1_N(\tau)\,:=\,f_N^1(N^\frac12\tau)\]
converges weakly-* in $\Ld^\infty(\R^+;\Ld^2(\T^2))$ to the unique solution $\bar f^1$ of an effective equation coupled to the set $\bar g=\{\bar g^m\}_{m\ge2}$ of limiting rescaled background correlations: more precisely, on the Hilbert space $\Hc:=\Ld^2(\T^2)\otimes\mathcal F_+(\Ld^2(\T^2))$, where $\mathcal F_+(\Ld^2(\T^2))$ is the bosonic Fock space for background correlations,
we have in the strong sense
\begin{equation}\label{eq:lim-hier-0}
\partial_\tau({\bar f^1},{\bar g})\,=\,iS^*({\bar f^1},{\bar g}),\qquad ({\bar f^1},{\bar g})|_{\tau=0}=({f^\circ},0),
\end{equation}
where $S^*$ is the adjoint of some explicit densely-defined symmetric operator $S$ on $\Hc$ (see Theorem~\ref{th:unif-wave-2}).
Although the limiting Hamiltonian $S^*$ is expected to lack self-adjointness and might not even generate a semigroup, equation~\eqref{eq:lim-hier-0} is well-posed in $C^2_b(\R^+;\Hc)$ (see Proposition~\ref{prop:unique}) and its solution satisfies the following RAGE theorem (see Proposition~\ref{th:RAGE}): denoting by $\{\lambda_k\}_k$ the set of real eigenvalues of~$S^*$, there exists a family of positive contractions $\{Q_k\}_k$ on~$\Ld^2(\T^2)$ such that we can decompose
\[\bar f^1(\tau)\,=\,\sum_ke^{i\tau\lambda_k}Q_kf^\circ+\bar R(\tau),\]
where the series converges in the weak operator topology and where the remainder satisfies for all $h\in\Ld^2(\T^2)$,
\[\lim_{T\uparrow\infty}\frac1T\int_0^T|\langle h,\bar R(\tau)\rangle_{\Ld^2(\T^2)}|^2\,d\tau\,=\,0.\]
{In particular, if the only real eigenvalue of $S^*$ is $0$ and if the kernel reduces to constants, then this RAGE theorem implies the weak relaxation to equilibrium: $\bar f^1(\tau)\cvf1$ weakly in~$\Ld^2(\T^2)$ as $\tau\uparrow\infty$ in Ces\`aro mean.}
\end{theor}

{\subsection{Comparison to Lenard--Balescu thermalization for plasmas}
Let us briefly emphasize the differences of the present work for point-vortex type systems compared to Lenard--Balescu thermalization for plasmas as studied in~\cite{DSR-21}. For plasma type systems, we consider point particles evolving according to Newton's equations with pairwise interactions deriving from a potential, say on the torus~$\T^d$. The corresponding mean-field equilibrium~$\mu_\beta$ is Maxwellian and the single-particle linearized mean-field operator reduces to the transport operator $iL_\beta=v\cdot\nabla_x$, which obviously has absolutely continuous spectrum orthogonally to its kernel. For many-particle linearized operators, compact corrections are added to transport, but linear Landau damping ensures that at Maxwellian equilibrium the spectrum remains absolutely continuous, see e.g.~\cite[Lemma~4.2]{DSR-21}. Therefore, plasma type systems always fit in the non-degenerate setting of Conjecture~\ref{conj}(i) and thermalization is always expected to hold, as was indeed first predicted in 1960 by Balescu and Lenard~\cite{Balescu-60,Lenard-60,Guernsey-62,TGM-64}.
In this plasma setting, we have obtained in~\cite{DSR-21} a result similar to Theorem~\ref{th:main-nonGauss}, justifying the outset of thermalization on subcritical timescales, while the rigorous understanding of the critical timescale was left open as here due to the possibility of uncontrolled echoes.

For point-vortex type systems in the non-degenerate setting, although our approach is broadly the same as in~\cite{DSR-21}, our analysis is made quite more delicate technically due to the lack of a closed formula for the resolvent of linearized mean-field operators: compare indeed Lemma~\ref{lem:nonGaus} below with \cite[Lemma~4.2]{DSR-21}.

As the degenerate case of Conjecture~\ref{conj}(ii) never occurs for plasmas, there is no equivalent of Theorems~\ref{th:main-Gauss} and~\ref{th:main-unif} in~\cite{DSR-21}, and those are indeed based on completely different scaling arguments. However note that, for plasma systems in the presence of a strong external magnetic field, a gyrokinetic approximation would reduce the problem to the same as for point-vortex type systems, thus making degeneracy possible: this was studied in~\cite{Taylor-McNamara-71} in the physics literature and we believe that one could easily adapt Theorems~\ref{th:main-Gauss} and~\ref{th:main-unif} to that case.
}

\section{Rigorous BBGKY analysis}

The starting point of our analysis is the BBGKY hierarchy of equations for correlation functions, combined with rigorous a priori estimates on the latter.

\subsection{BBGKY hierarchy}
We denote by~$\{f_N^m\}_{1\le m\le N}$ the marginals of the $N$-point density $f_N$, that is,
\begin{equation}\label{eq:def-fNm}
f_N^m(t,x_1,\ldots,x_m)\,:=\,\int_{(\R^2)^{N-m}}f_N(t,x_1,\ldots,x_m,x_{m+1},\ldots,x_N)\,dx_{m+1}\ldots dx_N.
\end{equation}
As background particles with labels $2,\ldots,N$ are exchangeable initially, cf.~\eqref{eq:fN0-2}, they remain so over time, hence the marginal $f_N^m$ is symmetric in its last $m-1$ variables.
Upon partial integration, the Liouville equation~\eqref{eq:Liouville-2} yields the following BBGKY hierarchy of equations for marginals,
\begin{multline}\label{eq:BBGKY-2}
\partial_tf_N^m+\sum_{i=1}^m\Big(F(x_i)+\tfrac1N\sum_{j=1}^mK(x_i-x_j)\Big)\cdot\nabla_if_N^m\\
+\tfrac{N-m}N\sum_{i=1}^m\int_{\R^2}K(x_i-x_*)\cdot\nabla_if_N^{m+1}(x_1,\ldots,x_m,x_*)\,dx_*=0.
\end{multline}
We recall that the mean-field approximation is obtained formally by assuming that the tagged particle remains approximately independent of the background particles and that the latter remain approximately at equilibrium. As the first marginal of the Gibbs ensemble converges to~$\mu_\beta$ as $N\uparrow\infty$, this means that we expect to approximate
\begin{equation}\label{eq:fN2-MF}
f_N^2\,\approx\, f_N^1\otimes\mu_\beta.
\end{equation}
Inserting this in the above BBGKY equation for $f_N^1$, we find that $f_N^1$ should stay close to the solution $f^1$ of the linearized mean-field equation
\begin{equation}\label{eq:lin-MFeqn}
\partial_tf^1+(F+K\ast\mu_\beta)\cdot\nabla f^1=0,\qquad f^1|_{t=0}=f^\circ.
\end{equation}
In the present axisymmetric setting, as $f^\circ,\mu_\beta$ are radial and as $F,K$ are orthogonal gradients of radial functions, this mean-field evolution is trivial: $f^1(t)=f^\circ$ for all $t\ge0$.

\subsection{Correlation functions or cumulants}\label{sec:correlations}
For notational convenience, we denote by
\[g_N^1\,:=\,\tfrac1{\mu_\beta}f_N^1\]
the ratio of the tagged particle density by the mean-field equilibrium.
As the linearized mean-field evolution~\eqref{eq:lin-MFeqn} is trivial in the axisymmetric setting,
we aim to characterize the next-order correction, which amounts to the defect in the approximation~\eqref{eq:fN2-MF}.
This brings us to define
\[g_N^2\,:=\,\tfrac1{\mu_\beta^{\otimes2}}(f_N^2-f_N^1\otimes\mu_\beta),\]
which captures the correlation of the tagged particle with a typical background particle. Note that we take the convention to define correlation functions as divided by the mean-field equilibrium.
More generally, we introduce all higher-order correlation functions $\{g_N^m\}_{1\le m\le N}$ for the tagged particle with respect to the mean-field background equilibrium $\mu_\beta^{\otimes N-1}$:
these correlation functions are defined so as to satisfy the following cluster expansions for marginals,
\begin{equation}\label{eq:correl-def-clust-2}
f_N^m(t,x_1,\ldots,x_m)\,=\,\mu_\beta^{\otimes m}(x_1,\ldots,x_m)\sum_{n=1}^m\sum_{\sigma\in P_{n-1}^{m-1}}g_N^n(t,x_1,x_\sigma),\qquad1\le m\le N,
\end{equation}
where $P_{n-1}^{m-1}$ stands for the collection of subsets of $\{2,\ldots,m\}$ with cardinality $n-1$ and where for an index subset $\sigma:=\{i_1,\ldots,i_{n-1}\}$ we have set $x_\sigma:=(x_{i_1},\ldots,x_{i_{n-1}})$.
For all $m$, the correlation function $g_N^m$ is uniquely chosen to be symmetric in its last $m-1$ variables and to satisfy $\int_{\R^2}g_N^m(t,x_1,\ldots,x_m)\,\mu_\beta(x_j)\,dx_j=0$ for all $2\le j\le m$.
More explicitly, the above relations can be inverted and the correlation functions are given by
\begin{equation}\label{eq:correl-def-2}
g_N^m(t,x_1,\ldots,x_m):=\sum_{n=1}^m(-1)^{m-n}\sum_{\sigma\in P_{n-1}^{m-1}}\tfrac{f_N^{n}}{\mu_\beta^{\otimes n}}(t,x_1,x_\sigma).
\end{equation}
For instance,
\begin{equation*}
g_N^3(t,x_1,x_2,x_3)\,=\,\tfrac{f_N^3}{\mu_\beta^{\otimes3}}(t,x_1,x_2,x_3)-\tfrac{f_N^2}{\mu_\beta^{\otimes2}}(t,x_1,x_2)-\tfrac{f_N^2}{\mu_\beta^{\otimes2}}(t,x_1,x_3)+\tfrac{f_N^1}{\mu_\beta}(t,x_1).
\end{equation*}
We may then reformulate the BBGKY hierarchy~\eqref{eq:BBGKY-2} as a hierarchy of equations on correlation functions.
We shall see in Lemma~\ref{lem:eqn-gNm-21-G} that these equations get drastically simplified in the specific case when the mean-field equilibrium is Gaussian.

\begingroup\allowdisplaybreaks
\begin{lem}[BBGKY hierarchy for correlations]\label{lem:eqn-gNm-21}
For all $1\le m\le N$,
\begin{equation}\label{eq:eqn-gNm-21}
\partial_t g_N^m+iL_{N,\beta}^mg_N^m
\,=\,iS_{N,\beta}^{m,+}g_N^{m+1}+\tfrac1N\Big(iS_{N,\beta}^{m,\circ}g_N^{m}+iS_{N,\beta}^{m,-}g_N^{m-1}+iS_{N,\beta}^{m,=}g_N^{m-2}\Big),
\end{equation}
where we have set for notational  convenience $g^r_N=0$ for $r<1$ or $r>N$, and where we have defined the operators
\begin{eqnarray*}
iL_{N,\beta}^mh^m&:=&\sum_{j=1}^m(F+K\ast\mu_\beta)(x_j)\cdot\nabla_{j}h^{m}\\
&+&\tfrac{N-m}N\sum_{j=2}^m(\nabla\log\mu_\beta)(x_j)\cdot\int_{\R^2} K(x_j-x_*)\,h^{m}(x_{[m]\setminus\{j\}},x_*)\,\mu_\beta(x_*)\,dx_*,\\
iS_{N,\beta}^{m,+}h^{m+1}&:=&-\,\tfrac{N-m}N\sum_{j=1}^m\int_{\R^2} K(x_j-x_*)\cdot \nabla_{j;\beta}h^{m+1}(x_{[m]},x_*)\,\mu_\beta(x_*)\,dx_*,\\
iS_{N,\beta}^{m,\circ}h^{m}&:=&-\,\sum_{i,j=1}^m\Big(K(x_i-x_j)-K\ast\mu_\beta(x_i)\Big)\cdot\nabla_{i;\beta}h^m\\
&+&\sum_{i=1}^m\sum_{2\le j\le m\atop i\ne j}\int_{\R^2}K(x_i-x_*)\cdot\nabla_{i;\beta}h^{m}(x_{[m]\setminus\{j\}},x_*)\,\mu_\beta(x_*)\,dx_*,\\
iS_{N,\beta}^{m,-}h^{m-1}&:=&-\,\sum_{i=1}^m\sum_{2\le j\le m\atop i\ne j}\Big(K(x_i-x_j)-(K\ast\mu_\beta)(x_i)\Big)\cdot\nabla_{i;\beta}h^{m-1}(x_{[m]\setminus\{j\}})\\
&-&\sum_{i=2}^m\sum_{j=1}^mK(x_i-x_j)\cdot(\nabla\log\mu_\beta)(x_i)\,h^{m-1}(x_{[m]\setminus\{i\}})\\
&+&\sum_{2\le i,j\le m}^{\ne}\int_{\R^2} K(x_i-x_*)\cdot (\nabla\log\mu_\beta)(x_i)\,h^{m-1}(x_{[m]\setminus\{i,j\}},x_*)\,\mu_\beta(x_*)\,dx_*,\\
iS_{N,\beta}^{m,=}h^{m-2}&:=&-\,\sum_{2\le i,j\le m}^{\ne}K(x_i-x_j)\cdot(\nabla\log\mu_\beta)(x_i)h^{m-2}(x_{[m]\setminus\{i,j\}}),
\end{eqnarray*}
with the short-hand notations $[m]:=\{1,\ldots,m\}$ and $\nabla_{i;\beta}:=\nabla_i+(\nabla\log\mu_\beta)(x_i)$.
\end{lem}

\begin{proof}
By the definition~\eqref{eq:correl-def-2} of correlation functions, the BBGKY equations~\eqref{eq:BBGKY-2} yield
\begin{multline}\label{eq:eqn-gNm-0-2}
\partial_t g_N^m
+F(x_1)\cdot\nabla_{1;\beta}g_N^m
\,=\,-\sum_{j=2}^mF(x_j)\cdot\nabla_{j;\beta}\sum_{n=1}^m(-1)^{m-n}\sum_{\sigma\in P_{n-1}^{m-1}}\mathds1_{j\in\sigma}\tfrac{f_N^n}{\mu_\beta^{\otimes n}}(x_1,x_\sigma)\\
-\tfrac1N\sum_{j=2}^mK(x_1-x_j)\cdot(\nabla_{1;\beta}-\nabla_{j;\beta})\sum_{n=1}^m(-1)^{m-n}\sum_{\sigma\in P_{n-1}^{m-1}}\mathds1_{j\in\sigma}\tfrac{f_N^n}{\mu_\beta^{\otimes n}}(x_1,x_\sigma)\\
-\tfrac1N\sum_{i,j=2}^mK(x_i-x_j)\cdot\nabla_{i;\beta}\sum_{n=1}^m(-1)^{m-n}\sum_{\sigma\in P_{n-1}^{m-1}}\mathds1_{i,j\in\sigma}\tfrac{f_N^n}{\mu_\beta^{\otimes n}}(x_1,x_\sigma)\\
-\sum_{n=1}^m(-1)^{m-n}\tfrac{N-n}{N}\sum_{\sigma\in P_{n-1}^{m-1}}\int_{\R^2}K(x_1-x_*)\cdot\Big(\nabla_{1;\beta}\tfrac{f_N^{n+1}}{\mu_\beta^{\otimes n+1}}\Big)(x_1,x_\sigma,x_*)\,\mu_\beta(x_*)\,dx_*\\
-\sum_{i=2}^m\sum_{n=1}^m(-1)^{m-n}\tfrac{N-n}{N}\sum_{\sigma\in P_{n-1}^{m-1}}\mathds1_{i\in\sigma}\int_{\R^2}K(x_i-x_*)\cdot\Big(\nabla_{i;\beta}\tfrac{f_N^{n+1}}{\mu_\beta^{\otimes n+1}}\Big)(x_1,x_\sigma,x_*)\,\mu_\beta(x_*)\,dx_*.
\end{multline}
Replacing the marginals in terms of cumulants, cf.~\eqref{eq:correl-def-clust-2}, we get for the first two right-hand side terms, for all~$j\in[m]\setminus\{1\}$,
\begin{eqnarray}
\lefteqn{\sum_{n=1}^m(-1)^{m-n}\sum_{\sigma\in P_{n-1}^{m-1}}\mathds1_{j\in\sigma}\tfrac{f_N^n}{\mu_\beta^{\otimes n}}(x_1,x_{\sigma})}\nonumber\\
&=&\sum_{n=1}^m(-1)^{m-n}\sum_{\sigma\in P_{n-1}^{m-1}}\mathds1_{j\in\sigma}\sum_{r=1}^{n}\sum_{\tau\in P_{r-1}^{\sigma}}g_N^r(x_1,x_{\tau})\nonumber\\
&=&\sum_{r=1}^{m}\sum_{\tau\in P_{r-1}^{m-1}} g_N^r(x_1,x_{\tau})\sum_{n=r}^m(-1)^{m-n}
\underbrace{\sharp\{\sigma\in P_{n-1}^{m-1}:j\in\sigma,\tau\subset\sigma\}}_{=\mathds1_{j\in\tau}\binom{m-r}{n-r}+\mathds1_{j\notin\tau}\binom{m-1-r}{n-1-r}}\nonumber\\
&=&g_N^m(x_{[m]})+g_N^{m-1}(x_{[m]\setminus\{j\}}),\label{eq:marg-cumul-combi}
\end{eqnarray}
where we used the combinatorial identity
\begin{equation}\label{eq:combin-ide}
\sum_{j=0}^p(-1)^{p-j}\binom{p}{j}=\delta_{p=0}.
\end{equation}
Similarly, for the third right-hand side term in~\eqref{eq:eqn-gNm-0-2}, we find for all $i,j\in[m]\setminus\{1\}$,
\begin{multline*}
\sum_{n=1}^m(-1)^{m-n}\sum_{\sigma\in P_{n-1}^{m-1}}\mathds1_{i,j\in\sigma}\tfrac{f_N^n}{\mu_\beta^{\otimes n}}(x_1,x_\sigma)\\
\,=\,g_N^m(x_{[m]})
+g_N^{m-1}(x_{[m]\setminus\{i\}})+g_N^{m-1}(x_{[m]\setminus\{j\}})
+g_N^{m-2}(x_{[m]\setminus\{i,j\}}).
\end{multline*}
For the fourth right-hand side term in~\eqref{eq:eqn-gNm-0-2}, replacing again marginals in terms of cumulants, we find
\begin{multline}\label{eq:decomp-4thtrermpl}
\sum_{n=1}^m(-1)^{m-n}\tfrac{N-n}N\sum_{\sigma\in P_{n-1}^{m-1}}\int_{\R^2}K(x_1-x_*)\cdot\Big(\nabla_{1;\beta}\tfrac{f_N^{n+1}}{\mu_\beta^{\otimes n+1}}\Big)(x_1,x_\sigma,x_*)\,\mu_\beta(x_*)\,dx_*\\
\,=\,\sum_{n=1}^m(-1)^{m-n}\tfrac{N-n}N\sum_{\sigma\in P_{n-1}^{m-1}}\sum_{r=1}^n\sum_{\tau\in P^\sigma_{r-1}}\int_{\R^2}K(x_1-x_*)\cdot\nabla_{1;\beta}g_N^{r+1}(x_1,x_\tau,x_*)\,\mu_\beta(x_*)\,dx_*\\
+\sum_{n=1}^m(-1)^{m-n}\tfrac{N-n}N\sum_{\sigma\in P_{n-1}^{m-1}}\sum_{r=1}^n\sum_{\tau\in P^\sigma_{r-1}}(K\ast\mu_\beta)(x_1)\cdot\nabla_{1;\beta}g_N^{r}(x_1,x_\tau).
\end{multline}
Note that
\begin{eqnarray}
\lefteqn{\sum_{n=1}^m(-1)^{m-n}\tfrac{N-n}N\sum_{\sigma\in P_{n-1}^{m-1}}\sum_{r=1}^n\sum_{\tau\in P_{r-1}^{\sigma}}\int_{\R^2}K(x_1-x_*)\cdot\nabla_{1;\beta}g_N^{r+1}(x_1,x_\tau,x_*)\,\mu_\beta(x_*)\,dx_*}\nonumber\\
&=&\sum_{r=1}^m\sum_{\tau\in P_{r-1}^{m-1}}\bigg(\sum_{n=r}^{m}(-1)^{m-n}\tfrac{N-n}N\underbrace{\sharp\{\sigma\in  P^{m-1}_{n-1}:\tau\subset\sigma\}}_{=\binom{m-r}{n-r}}\bigg)\nonumber\\
&&\hspace{3cm}\times\int_{\T^d}K(x_1-x_*)\cdot\nabla_{1;\beta}g_N^{r+1}(x_1,x_\tau,x_*)\,\mu_\beta(x_*)\,dx_*\nonumber\\
&=&\tfrac{N-m}N\int_{\T^d}K(x_1-x_*)\cdot\nabla_{1;\beta}g_N^{m+1}(x_{[m]},x_*)\,\mu_\beta(x_*)\,dx_*\nonumber\\
&&\hspace{2cm}-\tfrac{1}N\sum_{j=2}^{m}\int_{\T^d}K(x_1-x_*)\cdot\nabla_{1;\beta}g_N^{m}(x_{[m]\setminus\{j\}},x_*)\,\mu_\beta(x_*)\,dx_*,
\label{eq:integr-marg-cum-ag}
\end{eqnarray}
where the last identity follows from the following computation, based on~\eqref{eq:combin-ide},
\begin{eqnarray*}
\lefteqn{\sum_{n=r}^m(-1)^{m-n}\tfrac{N-n}N\binom{m-r}{n-r}}\\
&=&\sum_{n=0}^{m-r}(-1)^{m-n-r}\tfrac{N-n-r}N\binom{m-r}{n}\\
&=&\tfrac{N-r}N\sum_{n=0}^{m-r}(-1)^{m-n-r}\binom{m-r}{n}-\tfrac{m-r}N\sum_{n=1}^{m-r}(-1)^{m-n-r}\binom{m-r-1}{n-1}\\
&=&\tfrac{N-m}N\delta_{r=m}-\tfrac{1}N\delta_{r=m-1}.
\end{eqnarray*}
Inserting~\eqref{eq:integr-marg-cum-ag} into~\eqref{eq:decomp-4thtrermpl}, and arguing similarly for the last right-hand side term in~\eqref{eq:decomp-4thtrermpl}, we get
\begin{multline*}
\sum_{n=1}^m(-1)^{m-n}\tfrac{N-n}N\sum_{\sigma\in P_{n-1}^{m-1}}\int_{\R^2}K(x_1-x_*)\cdot\Big(\nabla_{1;\beta}\tfrac{f_N^{n+1}}{\mu_\beta^{\otimes n+1}}\Big)(x_1,x_\sigma,x_*)\,\mu_\beta(x_*)\,dx_*\\
\,=\,\tfrac{N-m}N\int_{\R^2}K(x_1-x_*)\cdot\nabla_{1;\beta}g_N^{m+1}(x_{[m]},x_*)\,\mu_\beta(x_*)\,dx_*\\
-\tfrac{1}N\sum_{j=2}^m\int_{\R^2}K(x_1-x_*)\cdot\nabla_{1;\beta}g_N^{m}(x_{[m]\setminus\{j\}},x_*)\,\mu_\beta(x_*)\,dx_*\\
+\tfrac{N-m}N(K\ast\mu_\beta)(x_1)\cdot\nabla_{1;\beta}g_N^{m}(x_{[m]})
-\tfrac{1}N\sum_{j=2}^m(K\ast\mu_\beta)(x_1)\cdot\nabla_{1;\beta}g_N^{m-1}(x_{[m]\setminus\{j\}}),
\end{multline*}
Similarly, for the last right-hand side term in~\eqref{eq:eqn-gNm-0-2}, we find
\begin{multline*}
{\sum_{i=2}^m\sum_{n=1}^m(-1)^{m-n}\tfrac{N-n}N\sum_{\sigma\in P_{n-1}^{m-1}}\mathds1_{i\in\sigma}\int_{\R^2} K(x_i-x_*)\cdot\Big(\nabla_{i;\beta}\tfrac{f_N^{n+1}}{\mu_\beta^{\otimes n+1}}\Big)(x_1,x_\sigma,x_*)\,\mu_\beta(x_*)\,dx_*}\\
\,=\,\tfrac{N-m}N\sum_{i=2}^m\int_{\R^2} K(x_i-x_*)\cdot \nabla_{i;\beta}g_N^{m+1}(x_{[m]},x_*)\,\mu_\beta(x_*)\,dx_*\\
-\tfrac1N\sum_{2\le i,j\le m}^{\ne}\int_{\R^2} K(x_i-x_*)\cdot \nabla_{i;\beta}g_N^{m}(x_{[m]\setminus\{j\}},x_*)\,\mu_\beta(x_*)\,dx_*\\
+\tfrac{N-m}N\sum_{i=2}^m\int_{\R^2} K(x_i-x_*)\cdot (\nabla\log\mu_\beta)(x_i)\,g_N^{m}(x_{[m]\setminus\{i\}},x_*)\,\mu_\beta(x_*)\,dx_*\\
-\tfrac1N\sum_{2\le i,j\le m}^{\ne}\int_{\R^2} K(x_i-x_*)\cdot (\nabla\log\mu_\beta)(x_i)\,g_N^{m-1}(x_{[m]\setminus\{i,j\}},x_*)\,\mu_\beta(x_*)\,dx_*\\
+\tfrac{N-m}N\sum_{i=2}^m(K\ast\mu_\beta)(x_i)\cdot \nabla_{i;\beta}g_N^{m}(x_{[m]})
-\tfrac{1}N\sum_{2\le i,j\le m}^{\ne}(K\ast\mu_\beta)(x_i)\cdot \nabla_{i;\beta}g_N^{m-1}(x_{[m]\setminus\{j\}})\\
+\tfrac{N-m}N\sum_{i=2}^m(K\ast\mu_\beta)(x_i)\cdot (\nabla\log\mu_\beta)(x_i)\,g_N^{m-1}(x_{[m]\setminus\{i\}})\\
-\tfrac{1}N\sum_{2\le i,j\le m}^{\ne}(K\ast\mu_\beta)(x_i)\cdot (\nabla\log\mu_\beta)(x_i)\,g_N^{m-2}(x_{[m]\setminus\{i,j\}}).
\end{multline*}
Combining all the above identities into~\eqref{eq:eqn-gNm-0-2}, and using various cancellations such as
\[F\cdot\nabla\log\mu_\beta\,=\,(K\ast\mu_\beta)\cdot\nabla\log\mu_\beta\,=\,0,\]
which follow from the radial nature of $V,W$,
the conclusion follows.
\end{proof}
\endgroup

\subsection{A priori correlation estimates}
We prove uniform-in-time propagation-of-chaos estimates for the particle system in form of a priori bounds on correlation functions.
This is deduced from a symmetry argument inspired by the work of Bodineau, Gallagher, and Saint-Raymond~\cite{BGSR-16}, which is combined as in~\cite{DSR-21} with some classical large deviation estimates to accommodate the fact that correlations are defined with respect to the mean-field equilibrium $\mu_\beta^{\otimes N}$ instead of the exact Gibbs measure $M_{N,\beta}$.
As formal BBGKY analysis leads to expect $g_N^2=O(N^{-1})$, the present estimates may a priori seem quite suboptimal: however, we will see that in some degenerate cases the present estimate $g_N^2=O(N^{-1/2})$ is in fact optimal on long timescales (see in particular Theorem~\ref{th:unif-wave-2}).

\begin{lem}\label{lem:est-cum-2}
Provided that $\beta\|W\|_{\Ld^\infty(\R^2)}\ll1$ is small enough, we have for all~$0\le m<N$, uniformly in time,
\[\|g_N^{m+1}\|_{\Ld^2_\beta((\R^2)^{m+1})}\,:=\,\Big(\int_{(\R^2)^{m+1}}|g_N^{m+1}|^2\,\mu_\beta^{\otimes m+1}\Big)^\frac12\,\lesssim_{\beta,m,f^\circ}\,N^{-\frac m2}.\]
\end{lem}

\begin{proof}
Recall that correlations satisfy $\int_{\R^2}g_N^m(x_{[m]})\,\mu_\beta(x_j)\,dx_j=0$ for all \mbox{$2\le j\le m$.} Computing the $\Ld^2$ norm of the $N$-point density $f_N$ and inserting the cluster expansion~\eqref{eq:correl-def-clust-2} in terms of correlation functions, we then get
\[\int_{(\R^2)^N}\tfrac1{\mu_\beta^{\otimes N}}|f_N|^2\,=\,\sum_{m=1}^N\binom{N-1}{m-1}\int_{(\R^2)^{m}}|g_N^{m}|^2\mu_\beta^{\otimes m},\]
and thus, for all $0\le m<N$,
\begin{equation}\label{eq:apriori-sym-est}
\int_{(\R^2)^{m+1}}|g_N^{m+1}|^2\mu_\beta^{\otimes m+1}\,\lesssim_m\, N^{-m}\int_{(\R^2)^N}\tfrac1{\mu_\beta^{\otimes N}}|f_N|^2.
\end{equation}
It remains to estimate the norm of $f_N$ in the right-hand side. In order to obtain a uniform-in-time estimate, we shall relate it to the conserved quantity
\[\int_{(\R^2)^N}\tfrac1{M_{N,\beta}}|f_N|^2.\]
Replacing $\mu_\beta^{\otimes N}$ by $M_{N,\beta}$ requires to appeal to some classical large deviation estimates and we split the proof into three steps.

\medskip
\step1 Proof that we have uniformly in time, for any $1<q<\infty$,
\begin{equation}\label{eq:apriori-almost-conserv}
\int_{(\R^2)^N}\tfrac1{\mu_\beta^{\otimes N}}|f_N|^2\,\le\,\|\tfrac1{\mu_\beta}f^\circ\|_{\Ld^{2q}_\beta(\R^2)}^2e^{4\beta \|W\|_{\Ld^\infty(\R^2)}}(Z_\beta)^{\frac1q-2}(A_{N,\beta,q})^{1-\frac1{q}}(B_{N,\beta})^{2-\frac1q},
\end{equation}
where we have set
\begin{eqnarray*}
A_{N,\beta,q}&:=&\Big(\tfrac{(Z_\beta)^{q'}Z_{\beta,-q'}}{(Z_{\beta,0})^{q'+1}}\Big)^{N}\,\tfrac{\int_{(\R^2)^N}\exp(-\frac{\beta(q'+1)}{2N}\sum_{i,j=1}^NW(x_i-x_j))\,d\nu_{\beta,-q'}^{\otimes N}(x_1,\ldots,x_N)}{\big(\int_{(\R^2)^N}\exp(-\frac{\beta}{2N}\sum_{i,j=1}^NW(x_i-x_j))\,d\nu_{\beta,0}^{\otimes N}(x_1,\ldots,x_N)\big)^{q'+1}},\\
B_{N,\beta}&:=&\tfrac{Z_{N,\beta}}{\tilde Z_{N,\beta}}\,=\,\tfrac{\int_{(\R^2)^N}\exp(-\frac\beta{2N}\sum_{i,j=1}^NW(x_i-x_j))\,d\nu_{\beta,0}^{\otimes N}(x_1,\ldots,x_N)}{\int_{(\R^2)^{N-1}}\exp(-\frac\beta{2N}\sum_{i,j=2}^NW(x_i-x_j))\,d\nu_{\beta,0}^{\otimes N-1}(x_2,\ldots,x_N)},
\end{eqnarray*}
and for any $\kappa\in\R$,
\[d\nu_{\beta,\kappa}(x)\,:=\,Z_{\beta,\kappa}^{-1}e^{-\beta(V+\kappa W\ast\mu_\beta)(x)}dx,\qquad Z_{\beta,\kappa}\,:=\,\int_{\R^2}e^{-\beta(V+\kappa W\ast\mu_\beta)}.\]
These $N$-dependent factors $A_{N,\beta,q},B_{N,\beta}$ take the form of ratios of partition functions, which we  estimate in the next steps by means of large deviation theory.

\medskip\noindent
We turn to the proof of~\eqref{eq:apriori-almost-conserv}.
H\"older's inequality yields
\[\int_{(\R^2)^N}\tfrac1{\mu_\beta^{\otimes N}}|f_N|^2\,\le\,\Big(\int_{(\R^2)^N}\tfrac{|M_{N,\beta}|^{q'+1}}{|\mu_\beta^{\otimes N}|^{q'}}\Big)^\frac1{q'}\Big(\int_{(\R^2)^N}\tfrac{|f_N|^{2q}}{|M_{N,\beta}|^{2q-1}}\Big)^\frac1q.\]
As $M_{N,\beta}$ is a global equilibrium for the Liouville equation~\eqref{eq:Liouville-2}, we note that the last factor is a conserved quantity,
\[\partial_t\int_{(\R^2)^N}\tfrac{|f_N|^{2q}}{M_{N,\beta}^{2q-1}}\,=\,0.\]
Recalling the initial condition~\eqref{eq:fN0-2}, we deduce
\begin{multline*}
\int_{(\R^2)^N}\tfrac1{\mu_\beta^{\otimes N}}|f_N|^2\,\le\,\|\tfrac1{\mu_\beta}f^\circ\|_{\Ld^{2q}_\beta(\R^2)}^2\Big(\int_{(\R^2)^N}\tfrac{|M_{N,\beta}|^{q'+1}}{|\mu_\beta^{\otimes N}|^{q'}}\Big)^\frac1{q'}\\
\times\sup_{x_1}\bigg(\mu_\beta(x_1)^{2q-1}\int_{(\R^2)^{N-1}}\tfrac{|\tilde M_{N,\beta}(x_2,\ldots,x_N)|^{2q}}{|M_{N,\beta}(x_1,\ldots,x_N)|^{2q-1}}dx_2\ldots dx_N\bigg)^\frac1q.
\end{multline*}
Inserting the definitions of $\mu_\beta,M_{N,\beta},\tilde M_{N,\beta}$, we find that the last two factors in the right-hand side satisfy
\[\int_{(\R^2)^N}\tfrac{|M_{N,\beta}|^{q'+1}}{|\mu_\beta^{\otimes N}|^{q'}}\,=\,A_{N,\beta,q},\]
and
\begin{equation*}
\sup_{x_1}\bigg(\mu_\beta(x_1)^{2q-1}\int_{(\R^2)^{N-1}}\tfrac{|\tilde M_{N,\beta}(x_2,\ldots,x_N)|^{2q}}{|M_{N,\beta}(x_1,\ldots,x_N)|^{2q-1}}dx_2\ldots dx_N\bigg)
\,\le\,e^{4q\beta \|W\|_{\Ld^\infty(\R^2)}}\,\big(\tfrac{Z_{N,\beta}}{Z_\beta \tilde Z_{N,\beta}}\big)^{2q-1},
\end{equation*}
where $A_{N,\beta,q}$ is defined above. The claim~\eqref{eq:apriori-almost-conserv} follows.

\medskip
\step2 Asymptotics of partition functions:
given a probability measure $\nu_0$ on $\R^d$, $d\ge1$, and given a bounded, continuous, even interaction potential $W_0$ with \mbox{$\|W_0\|_{\Ld^\infty(\R^d)}$} small enough,
there is a unique solution $\mu_0$ of the fixed-point equation
\begin{equation}\label{eq:fixed-point-mu0gen}
\mu_0\,=\,Z_0^{-1}e^{-2W_0\ast\mu_0}\nu_0,\qquad Z_0\,=\,\int_{\R^d}e^{-2W_0\ast\mu_0}\nu_0,
\end{equation}
and the limit
\[\lim_{N\uparrow\infty}e^{Nm_0}\int_{(\R^2)^N}\exp\bigg(-\tfrac1{N}\sum_{i,j=1}^NW_0(x_i-x_j)\bigg)\,d\nu_0^{\otimes N}(x_1,\ldots,x_N)\]
then exists and belongs to $(0,\infty)$,
where $m_0$ is defined as
\begin{equation}\label{eq:comput-m0-smallW}
m_0\,:=\,-\log Z_0-\iint_{\R^d\times\R^d}W_0(x-y)\,d\mu_0(x)d\mu_0(y).
\end{equation}
We briefly show how this can be deduced from the large deviation results summarized in~\cite{Benarous-Brunaud-90}.
Appealing to~\cite[Theorem~B(ii)]{Benarous-Brunaud-90},
in terms of
\begin{equation}\label{eq:m0-min}
n_0\,:=\,\inf_{\mu\in\Pc(\R^d)}\bigg(\int_{\R^d}\mu\log\Big(\frac\mu{\nu_0}\Big)+\iint_{\R^d\times\R^d}W_0(x-y)\,d\mu(x)d\mu(y)\bigg),
\end{equation}
we find that the limit
\[\lim_{N\uparrow\infty}e^{Nn_0}\int_{(\R^2)^N}\exp\bigg(-\tfrac1{N}\sum_{i,j=1}^NW_0(x_i-x_j)\bigg)\,d\nu_0^{\otimes N}(x_1,\ldots,x_N)\]
exists and belongs to $(0,\infty)$
provided that this minimization problem~\eqref{eq:m0-min} admits a unique minimizer $\mu_0$ and {as long as} the value~$1$ does not belong to the spectrum of the operator $\Sigma_{\mu_0}f:=-2 W_0\ast(f\mu_0)$ on~$\Ld^2(\R^d,\mu_0)$.
{Since the} interactions are weak in the sense of $\|W_0\|_{\Ld^\infty(\R^d)}<\frac12$, we first note that for any probability measure~$\mu$ the operator $\Sigma_{\mu}$ on $\Ld^2(\R^d,\mu)$ has operator norm $<1$, so that the value $1$ is indeed always regular.
In addition, for $\|W_0\|_{\Ld^\infty(\R^d)}$ small enough, we can easily check that the minimization problem~\eqref{eq:m0-min} has a unique solution $\mu_0$, which is precisely given by the fixed-point equation~\eqref{eq:fixed-point-mu0gen}. The infimum value~\eqref{eq:m0-min} is then equal to~\eqref{eq:comput-m0-smallW}, $n_0=m_0$, and the claim follows.

\medskip
\step3 Conclusion.\\
We turn to the asymptotic analysis of the two factors $A_{N,\beta,q}$ and $B_{N,\beta}$ in~\eqref{eq:apriori-almost-conserv} as~$N\uparrow\infty$.
On the one hand, appealing to the result of Step~2, provided that $\beta\|W\|_{\Ld^\infty(\R^2)}$ is small enough,
we find after straightforward computations
\begin{equation*}
\lim_{N\uparrow\infty}A_{N,\beta,q}\,\simeq_{\beta,q}\,\lim_{N\uparrow\infty}\Big(\tfrac{(Z_\beta)^{q'}Z_{\beta,-q'}}{(Z_{\beta,0})^{q'+1}}\Big)^{N}\,\tfrac{\exp(-N\gamma_{\beta,-q'})}{\exp(-N(q'+1)\gamma_{\beta,0})},
\end{equation*}
where for any $\kappa\in\R$ we have set
\[\gamma_{\beta,\kappa}\,:=\,-\log\big(\tfrac{Z_\beta}{Z_{\beta,\kappa}}\big)- \tfrac12\beta(1-\kappa)\iint_{\R^2\times\R^2} W(x-y)\,d\mu_\beta(x)d\mu_\beta(y).\]
After simplifications, this entails
\begin{equation*}
\lim_{N\uparrow\infty}A_{N,\beta,q}\,\simeq_{\beta,q}\,1.
\end{equation*}
A similar computation shows
$\lim_{N\uparrow\infty}B_{N,\beta}\,\simeq_\beta\,1$.
Combining this with~\eqref{eq:apriori-sym-est} and~\eqref{eq:apriori-almost-conserv}, the conclusion follows.
\end{proof}

\subsection{Weighted Sobolev spaces}\label{sec:def-weight-Sob}
As shown in Lemma~\ref{lem:est-cum-2} above, for $1\le m\le N$, the correlation function $g_N^{m}$ is an element of the following Hilbert space,
\begin{multline*}
\Ld^2_\beta((\R^2)^{m})\,:=\,\Big\{h^{m}\in\Ld^2_\loc((\R^2)^{m})~:~\int_{(\R^2)^{m}}|h^{m}|^2\,\mu_\beta^{\otimes m}<\infty,\\
\text{and $h^m$ is symmetric in its last $m-1$ entries}\Big\},
\end{multline*}
or equivalently, recalling the notation $\otimes_s$ for symmetrized tensor product,
\[\Ld^2_\beta((\R^2)^{m})\,=\,\Ld^2_\beta(\R^2)\otimes\Ld^2_\beta(\R^2)^{\otimes_s(m-1)},\]
endowed with the norm
\begin{eqnarray*}
\|h^m\|_{\Ld^2_\beta((\R^2)^m)}^2&:=&\langle h^m,h^m\rangle_{\Ld^2_\beta((\R^2)^m)},\\
\langle h^m,g^m\rangle_{\Ld^2_\beta((\R^2)^m)}&:=&\int_{(\R^2)^m}\overline{h^m}\,g^m\,\mu_\beta^{\otimes m}.
\end{eqnarray*}
We further define a sequence of Sobolev spaces with respect to $\mu_\beta$: for $s\in \mathbb{N}$, we define the Hilbert space $H^s_\beta((\R^2)^m)$ as the subset of $\Ld^2_\beta((\R^2)^m)$ that is the domain of the norm
\begin{equation}\label{eq:Hsbeta}
\|h^m\|_{H^s_\beta((\R^2)^m)}^2\,:=\,\sum_{j=0}^s\int_{(\R^2)^m}|(\nabla_{{[m]}})^jh^m|^2\,\mu_\beta^{\otimes m},
\end{equation}
and we denote by $H^{-s}_\beta((\R^2)^m)$ the dual of $H^s_\beta((\R^2)^m)$ with respect to the scalar product of~$\Ld^2_\beta((\R^2)^m)$.
Note that we only consider integer regularity $s$ for simplicity.
We shall frequently use the following embeddings, for $s\ge0$,
\begin{eqnarray}
\|h^m\|_{\Ld^2_\beta((\R^2)^m)}&\le&\|h^m\|_{\Ld^\infty((\R^2)^m)},\nonumber\\
\|h^m\|_{H^{s}_\beta((\R^2)^m)}&\le&\|\mu_\beta\|_{\Ld^\infty((\R^2)^m)}^{\frac m2}\|h^m\|_{H^{s}((\R^2)^m)},\nonumber\\
\|h^m\|_{H^{-s}((\R^2)^m)}&\le&\|\mu_\beta\|_{\Ld^\infty((\R^2)^m)}^{\frac m2}\big\|\tfrac1{\mu_\beta^{\otimes m}}h^m\big\|_{H^{-s}_\beta((\R^2)^m)}.\label{eq:embeddingHs}
\end{eqnarray}
For the sake of completeness, we also give the proof of the following useful interpolation result.

\begin{lem}[Interpolation]\label{lem:interpol-SobG}$ $
The following holds for all $0\le s\le r$ and $h\in C^\infty_c(\R^2)$,
\begin{equation}\label{eq:interpol-SobG}
\|h\|_{H^s_\beta(\R^2)}\,\lesssim_{r,s}\,\|h\|_{\Ld^2_\beta(\R^2)}^{1-\frac sr}\|h\|_{H^r_\beta(\R^2)}^{\frac sr},
\end{equation}
in each of the following two cases:
\begin{enumerate}[(i)]
\item in the non-Gaussian setting with non-degenerate monotone angular velocity $\Omega_\beta$ in the sense of~\eqref{eq:nondegenerate-0} (in which case the multiplicative constant in~\eqref{eq:interpol-SobG} may further depend on the constant $R$ in~\eqref{eq:nondegenerate-0} and on an upper bound on $\beta$ and $\|\nabla\Omega_\beta\|_{W^{r-1,\infty}(\R^2)}$);
\smallskip\item in the Gaussian setting~\eqref{eq:Gaussian0} (in which case the multiplicative constant in~\eqref{eq:interpol-SobG} is independent of $\beta,R$).
\end{enumerate}
\end{lem}

\begin{proof}
We focus on the Gaussian setting~(ii), while the proof in the non-Gaussian setting~(i) follows along the same lines using the specific properties of $\Omega_\beta$ in~\eqref{eq:nondegenerate-0}.
Using that in the Gaussian setting
\[\nabla\tfrac1{\sqrt{\mu_\beta}}\,=\,\tfrac12\beta Rx\tfrac1{\sqrt{\mu_\beta}},\]
and decomposing $h=\tfrac1{\sqrt{\mu_\beta}}(h\sqrt{\mu_\beta})$,
we first note that
\[\|h\|_{H^s_\beta(\R^2)}\,\lesssim_s\,\sum_{j=0}^s\|\langle\beta Rx\rangle^{s-j}\nabla^j(h\sqrt{\mu_\beta})\|_{\Ld^2(\R^2)},\]
and thus, by an interpolation inequality due to Lin~\cite{Lin-86},
\[\|h\|_{H^s_\beta(\R^2)}\,\lesssim_s\,\|\langle\beta Rx\rangle^{s}h\sqrt{\mu_\beta}\|_{\Ld^2(\R^2)}+\|h\sqrt{\mu_\beta}\|_{H^s(\R^2)}.\]
Standard interpolation then yields for $r\ge s$,
\[\|h\|_{H^s_\beta(\R^2)}\,\lesssim_s\,\|h\sqrt{\mu_\beta}\|_{\Ld^2(\R^2)}^{1-\frac sr}\Big(\|\langle\beta Rx\rangle^{r}h\sqrt{\mu_\beta}\|_{\Ld^2(\R^2)}+\|h\sqrt{\mu_\beta}\|_{H^r(\R^2)}\Big)^{\frac sr}.\]
Further using $\nabla\sqrt{\mu_\beta}=-\tfrac12\beta Rx\sqrt{\mu_\beta}$ to estimate the $H^r$-norm of $h\sqrt{\mu_\beta}$, we deduce
\begin{equation}\label{eq:pre-interpol-SobG}
\|h\|_{H^s_\beta(\R^2)}\,\lesssim_{s,r}\,\sum_{j=0}^r\|h\sqrt{\mu_\beta}\|_{\Ld^2(\R^2)}^{1-\frac sr}\|\langle\beta Rx\rangle^{r-j}(\nabla^jh)\sqrt{\mu_\beta}\|_{\Ld^2(\R^2)}^{\frac sr}.
\end{equation}
As $\nabla\mu_\beta=-\beta Rx\mu_\beta$, an integration by parts yields for any $k\ge1$,
\begin{eqnarray*}
\||\beta Rx|^kh\sqrt{\mu_\beta}\|_{\Ld^2(\R^2)}^2&=&\int_{\R^2}|\beta Rx|^{2k}|h|^2\mu_\beta\\
&=&\int_{\R^2}\beta Rx|\beta Rx|^{2k-2}|h|^2\cdot(- \nabla\mu_\beta)\\
&\lesssim_k&\int_{\R^2}|\beta Rx|^{2k-1}|h||\nabla h| \mu_\beta+\int_{\R^2}|\beta Rx|^{2k-2}|h|^2\mu_\beta.
\end{eqnarray*}
Hence, by the Cauchy--Schwarz inequality,
\begin{equation*}
\||\beta Rx|^kh\sqrt{\mu_\beta}\|_{\Ld^2(\R^2)}\,\lesssim_k\,\||\beta Rx|^{k-1}(\nabla h)\sqrt{\mu_\beta}\|_{\Ld^2(\R^2)}+\||\beta Rx|^{k-1}h\sqrt{\mu_\beta}\|_{\Ld^2(\R^2)},
\end{equation*}
which gives by induction,
\begin{equation*}
\|\langle\beta Rx\rangle^kh\sqrt{\mu_\beta}\|_{\Ld^2(\R^2)}\,\lesssim_k\,\sum_{j=0}^k\|(\nabla^jh)\sqrt{\mu_\beta}\|_{\Ld^2(\R^2)}.
\end{equation*}
Using this to post-process~\eqref{eq:pre-interpol-SobG}, the conclusion~(ii) follows.
\end{proof}

\section{Non-degenerate case: non-Gaussian equilibrium}\label{sec:non-Gauss}

This section is devoted to the proof of Theorem~\ref{th:main-nonGauss}. More precisely, we establish the following more detailed result. Note that the positivity of the expression~\eqref{eq:def-abeta} is not obvious and is part of the proof.

\begin{theor}[Non-Gaussian setting]\label{th:main-nonGauss-re}
Assume that the external potential $V$ further satisfies $\nabla(V'/r)\in C^\infty_b(\R^2)$.
In terms of the mean-field equilibrium $\mu_\beta$, we define the angular velocity $\Omega_\beta$ as the smooth radial function given by
\[(\log\mu_\beta)'\,=\,\beta r\Omega_\beta.\]
Consider the non-Gaussian setting when $\Omega_\beta$ is nowhere constant:
more precisely, we assume for simplicity that $\Omega_\beta$ is monotone and satisfies the following non-degeneracy condition, for some~{$R\in(0,\infty)$},
\begin{equation}\label{eq:nondegenerate}
|\Omega_\beta'(r)|\,\ge\,\tfrac1R(r\wedge1),\qquad|\Omega_\beta''(0)|\,\ge\,\tfrac1R,
\qquad\text{for all $r\ge0$},
\end{equation}
and we also assume that $\beta$ is small enough depending on~$V,W$ and on this constant $R$.
Then, for any~\mbox{$\sigma\in(0,\frac1{20})$}, the subcritically-rescaled tagged particle density
\[\bar f_N^1(\tau)\,:=\,f_N^1(N^\sigma\tau)\]
satisfies in the radial distributional sense on $\R^+\times\R^2$,
\begin{equation}\label{eq:FP-noGauss}
N^{1-\sigma}\partial_\tau\langle\bar f_N^1\rangle~\xrightarrow{N\uparrow\infty}~\tfrac1r\partial_r \Big(ra_\beta(r)\big(\partial_r-(\log\mu_\beta)'(r) \big)f^\circ\Big),
\end{equation}
where the coefficient field $a_\beta$ is a positive scalar radial function that can be expressed as
\begin{equation}\label{eq:def-abeta}
a_\beta(r)\,:=\,\int_{\Sp^1}\bigg(\int_{\R^2}H_0(re,y)\big[\Re(iL_\beta^2+0)^{-1}H_\beta\big]\!(re,y)\,\mu_\beta(y)\,dy\bigg)d\sigma(e),
\end{equation}
in terms of:
\begin{enumerate}[$\bullet$]
\item the operator $L_\beta^2$ given by
\[L_\beta^2\,:=\,L_\beta\otimes\Id+\Id\otimes(L_\beta+\beta T_\beta),\]
with
\begin{eqnarray*}
iL_\beta h(x)&:=&-\beta^{-1}\nabla\log\mu_\beta(x)\cdot\nabla^\bot h(x),\\
iT_\beta h(x)&:=&\beta^{-1}\nabla\log\mu_\beta(x)\cdot\int_{\R^2}K(x-x_*)h(x_*)\,\mu_\beta(x_*)\,dx_*;
\end{eqnarray*}
\item the functions $H_0,H_\beta$ given by
\begin{eqnarray}
H_0(x,y)&:=&\tfrac{x}{|x|}\cdot K(x-y),\label{eq:def-HbH0}\\
H_\beta(x,y)&:=&\tfrac{x}{|x|}\cdot(-\nabla_1^\bot W_\beta)(x,y),\nonumber
\end{eqnarray}
where the `renormalized' interaction potential $W_\beta=W+O(\beta)$ is defined as a $\beta$-expansion,
\begin{equation}\label{eq:def-Wbeta}
W_\beta(x,y)\,:=\,\sum_{n=0}^\infty(-\beta)^nW^{\ast_{\mu_\beta}(n+1)}(x,y),
\end{equation}
and where so-called $\mu_\beta$-convolution powers are defined for all $n\ge1$ by
\begin{multline}\label{eq:def-Wconvsp}
\qquad W^{\ast_{\mu_\beta}n}(x_0,x_n)\,:=\,\int_{(\R^d)^{n-1}}\Big(\prod_{k=1}^nW(x_{k-1}-x_k)\Big)\\
\times\mu_\beta(x_1)\ldots \mu_\beta(x_{n-1})\,dx_1\ldots dx_{n-1}.
\end{multline}
\end{enumerate}
\end{theor}

\begin{rem}
In terms of the renormalized interaction potential $W_\beta$, we can define the corresponding renormalized force kernel
\begin{equation*}
K_\beta(x,y)\,:=\,
\tfrac1{|x|^2}\Big(x+\tfrac{|x|^2-y\cdot x}{y\cdot x^\bot}x^\bot \Big)
\Big(x\cdot(-\nabla_1^\bot W_\beta)(x,y)\Big).
\end{equation*}
Noting that we have by symmetry $x\cdot(\nabla_1^\bot W_\beta)(x,y)=-y\cdot(\nabla_2^\bot W_\beta)(x,y)$, we find that this choice of $K_\beta$ satisfies
\[K_\beta(x,y)=-K_\beta(y,x),\qquad (x-y)\cdot K_\beta(x,y)=0.\]
The definition of~$H_\beta$ in the above statement is then equivalent to
\[H_\beta(x,y)=\tfrac{x}{|x|}\cdot K_\beta(x,y).\]
\end{rem}

\subsection{Preliminary BBGKY analysis}
As~$V$ and~$W$ are radial, the mean-field force can be written as
\begin{equation}\label{eq:non-Gauss-rad}
(F+K\ast\mu_\beta)(x)=x^\bot\Omega_\beta(x),\qquad\nabla\log\mu_\beta(x)=\beta x\Omega_\beta(x),
\end{equation}
where $\Omega_\beta$ is the smooth radial function
\[\Omega_\beta(x)\,:=\,\tfrac1{\beta r}(\log\mu_\beta)'\,=\,-\tfrac1r(V+W\ast\mu_\beta)'(r).\]
In these terms, the linearized mean-field operators~$\{L_{N,\beta}^m\}_{1\le m\le N}$ in Lemma~\ref{lem:eqn-gNm-21} can be expressed as Kronecker sums
\begin{multline}\label{eq:L-Kron-NG}
L_{N,\beta}^m\,=\,
L_\beta\otimes(\Id_{\Ld^2_\beta(\R^2)})^{\otimes(m-1)}\\
+\sum_{j=2}^m(\Id_{\Ld^2_\beta(\R^2)})^{\otimes(j-1)}\otimes \big(L_\beta+\beta\tfrac{N-m}NT_\beta\big)\otimes (\Id_{\Ld^2_\beta(\R^2)})^{\otimes(m-j)},
\end{multline}
where we have defined the following single-particle operators on $\Ld^2_\beta(\R^2)$,
\begin{eqnarray}
(iL_\beta h)(x)&:=&-x\Omega_\beta(x)\cdot\nabla^\bot h(x),\nonumber\\
(iT_\beta h)(x)&:=&x\Omega_\beta(x)\cdot\int_{\R^2}K(x-x_*)\,h(x_*)\,\mu_\beta(x_*)\,dx_*.\label{eq:Lbeta0}
\end{eqnarray}
Note that in polar coordinates $x=(r,\theta)$ the operator $L_\beta$ takes the form
\begin{equation}\label{eq:Lbeta0-polar}
iL_\beta \,=\,\Omega_\beta(r) \partial_\theta,
\end{equation}
thus showing that $\Omega_\beta$ indeed plays the role of an angular velocity.
In the limit $N\uparrow\infty$, the linearized mean-field operators~\eqref{eq:L-Kron-NG} are replaced by
\begin{equation*}
L_{\beta}^m\,:=\,
L_\beta\otimes(\Id_{\Ld^2_\beta(\R^2)})^{\otimes(m-1)}\\
+\sum_{j=2}^m(\Id_{\Ld^2_\beta(\R^2)})^{\otimes(j-1)}\otimes (L_\beta+\beta T_\beta) \otimes (\Id_{\Ld^2_\beta(\R^2)})^{\otimes(m-j)}.
\end{equation*}
We start by studying the spectral properties of these operators.
For that purpose, we can focus on the following single-particle operators on $\Ld^2_\beta(\R^2)$,
\[L_\beta(\gamma)\,:=\,L_\beta+\gamma T_\beta,\qquad\text{for $\gamma\in\R$}.\]
The main difficulty is that we do not have a closed formula for the resolvent of $L_\beta(\gamma)$ for $\gamma\ne0$, in contrast with the situation for the linearized Vlasov operators in~\cite{DSR-21}: for that reason, the spectral analysis requires more care.\footnote{See however explicit resolvent computations in~\cite{Chavanis-12a,Chavanis-12b} in the specific Coulomb setting.}
Note that in the proof of item~(i) below we further show that $L_\beta(\gamma)$ is actually self-adjoint when viewed as acting on a suitably deformed Lipschitz-equivalent Hilbert space. In that deformed self-adjoint setting, item~(iii) then entails that the restriction of $L_\beta(\gamma)$ to the orthogonal complement of its kernel has purely absolutely continuous spectrum.

\begin{lem}[Properties of linearized mean-field operators]\label{lem:nonGaus}
{Let $\beta>0$ and let $V,W,\Omega_\beta$ be as in the statement of Theorem~\ref{th:main-nonGauss-re}.}
Then, for all $\gamma\ll_R\beta^{-2/3}$ small enough (hence in particular for all $\gamma\le\beta$ provided that $\beta\ll_R1$ is itself small enough), the following properties hold:
\begin{enumerate}[(i)]
\item The operator $L_\beta(\gamma)$ generates a $C_0$-group $\{e^{itL_\beta(\gamma)}\}_{t\in\R}$ that is uniformly bounded,
\[\qquad\sup_{t\in\R}\|e^{itL_\beta(\gamma)}h\|_{\Ld^2_\beta(\R^2)}\,\lesssim\,\|h\|_{\Ld^2_\beta(\R^2)},\qquad\text{for all $h\in\Ld^2_\beta(\R^2)$}.\]
\item The kernel of $L_\beta(\gamma)$ coincides with the set of radial functions.
\smallskip\item On the orthogonal complement
\begin{multline}\label{eq:def-Ebeta}
\hspace{1cm}\textstyle E_\beta\,:=\,\ker(L_\beta(\gamma))^\bot
\,=\,\ker(L_\beta)^\bot\\
\,=\,\bigg\{h\in\Ld^2_\beta(\R^2):\int_{\Sp^1}h(re)\,d\sigma(e)=0~\text{for almost all $r$}\bigg\},
\end{multline}
the restriction $L_\beta(\gamma)|_{E_\beta}$
satisfies the following limiting absorption principle: for all $g,h\in C^\infty_c(\R^2)\cap E_\beta$,
\begin{equation}\label{eq:LAP-Lbeta}
\sup_{\omega\in\C,\Im\omega\ne0}\big|\big\langle g,\big(L_\beta(\gamma)-\omega\big)^{-1}h\big\rangle_{\Ld^2_\beta(\R^2)}\big|
\,\lesssim_R\,\3g\3_\beta\3h\3_\beta,
\end{equation}
where we have set for abbreviation
\begin{equation}\label{eq:def-N3}
\3g\3_\beta\,:=\,\|\nabla g\|_{\Ld^2_\beta(\R^2)}
+\|\langle\tfrac1{x}\rangle g\|_{\Ld^2_\beta(\R^2)}
+\|(\nabla\log\mu_\beta)g\|_{\Ld^2_\beta(\R^2)}.
\end{equation}
\end{enumerate}
\end{lem}

\begin{rem}\label{rem:nonGaus}
By definition, the operators $L_{N,\beta}^m$ and $L_\beta^m$ are Kronecker sums of $L_\beta(\gamma)$ for different values of $0\le\gamma\le\beta$.
As a direct consequence of the above result, provided that $\beta\ll_R1$ is small enough, we deduce that for all $m\ge1$ the kernels of~$L_{N,\beta}^m$ and of~$L_\beta^m$ on~$\Ld^2_\beta((\R^2)^m)$ are both given by the set of functions that are radial in each of their $m$ entries. Hence, we have
\begin{multline*}
E_\beta^m\,:=\,\ker(L_{N,\beta}^m)^\bot\,=\,\ker(L_{\beta}^m)^\bot\\
\,=\,\bigg\{h\in\Ld^2_\beta((\R^2)^m)~\,:\,~\int_{(\Sp^1)^m}h(r_1e_1,\ldots,r_me_m)\,d\sigma(e_1)\ldots d\sigma(e_m)=0\\[-4mm]
\text{for almost all $r_1,\ldots,r_m$}\bigg\}.
\end{multline*}
Moreover, the proof of item~(iii) is easily repeated in this multi-particle setting: provided that $\beta\ll_{R,m}1$ is small enough (further depending on $m$), the restriction of $L_{N,\beta}^m$ (resp. $L_\beta^m$) to~$E_\beta^m$ satisfies the following limiting absorption principle, for all $g,h\in C^\infty_c((\R^2)^m)\cap E_\beta^m$,
\[\sup_{\omega\in\C,\Im\omega\ne0}\big|\big\langle g,(L_{N,\beta}^m-\omega)^{-1}h\big\rangle_{\Ld^2_\beta((\R^2)^m)}\big|
\,\lesssim_{R,m}\,\3g\3_{\beta;m}\3h\3_{\beta;m},\]
where we have set for abbreviation
\begin{equation*}
\3g\3_{\beta;m}\,:=\,\sum_{j=1}^m\Big(\|\nabla_j g\|_{\Ld^2_\beta((\R^2)^m)}
+\|\langle\tfrac1{x_j}\rangle g\|_{\Ld^2_\beta((\R^2)^m)}
+\|(\nabla\log\mu_\beta)(x_j)g\|_{\Ld^2_\beta((\R^2)^m)}\Big).
\end{equation*}
\end{rem}

\begin{proof}[Proof of Lemma~\ref{lem:nonGaus}]
We split the proof into three steps: we start with the proof of item~(i), then we establish fine spectral properties of the unperturbed operator $L_\beta$, proving items~(ii) and~(iii) with $\gamma=0$, before concluding perturbatively for $|\gamma|$ small enough.

\medskip
\step1 Proof of~(i).\\
The operator~$L_\beta$ in~\eqref{eq:Lbeta0-polar} is clearly self-adjoint on its natural domain in~$\Ld^2_\beta(\R^2)$. {Next, by definition of $T_\beta$, cf.~\eqref{eq:Lbeta0}, we can bound
\begin{equation*}
\|T_\beta h\|_{\Ld^2_\beta(\R^2)}\,\le\,\|x\Omega_\beta\|_{\Ld^2_\beta(\R^2)}\|K\|_{\Ld^\infty(\R^2)}\|h\|_{\Ld^2_\beta(\R^2)}.
\end{equation*}
The non-degeneracy condition~\eqref{eq:nondegenerate} implies that~$\mu_\beta$ has a decay faster than Gaussian: for~$r\ge1$,
\begin{equation}\label{eq:non-Gauss-mu-est}
\mu_\beta(0)\exp\Big(\tfrac12\Omega_\beta(0)\beta r^2-C\beta r^3\Big)\,\le\,\mu_\beta(r)\,\le\,\mu_\beta(0)\exp\Big(\tfrac12(\tfrac1R+\Omega_\beta(0))\beta r^2-\tfrac1{CR}\beta r^3\Big).
\end{equation}
As $|\Omega_\beta(r)|\lesssim 1+r$,
this decay ensures that $\|x\Omega_\beta\|_{\Ld^2_\beta(\R^2)}$ is finite, so that $T_\beta$ defines a bounded operator on $\Ld^2_\beta(\R^2)$. More precisely, we find $\|x\Omega_\beta\|_{\Ld^2_\beta(\R^2)}\lesssim_R\beta^{-\frac23}$, thus leading to the following estimate on the operator norm of $T_\beta$,}
\[\|T_\beta\|_{\Ld^2_\beta(\R^2)\to\Ld^2_\beta(\R^2)}\,\lesssim_R\,\beta^{-\frac23}.\]
For all $\gamma\in\R$, standard perturbation theory (e.g.~\cite[Theorem~IX.2.1]{Kato}) then ensures that the perturbed operator $L_\beta(\gamma)=L_\beta+\gamma T_\beta$ generates a $C_0$-group on $\Ld^2_\beta(\R^2)$, and it remains to show that this $C_0$-group is uniformly bounded.
For that purpose, we appeal to an energy conservation argument. Consider the following deformed scalar product on $\Ld^2_\beta(\R^2)$,
\begin{equation}\label{eq:deformed-Hilbert}
\langle g,h\rangle_{\widetilde \Ld{}^2_{\beta,\gamma}(\R^2)}\,:=\,\langle g,h\rangle_{\Ld^2_\beta(\R^2)}+\gamma\iint_{\R^2\times\R^2} W(x-y)\,\overline{g(x)}h(y)\,\mu_\beta(x)\mu_\beta(y)\,dxdy,
\end{equation}
as well as the associated norm
\[\|h\|_{\widetilde\Ld{}^2_{\beta,\gamma}(\R^2)}^2:=\langle h,h\rangle_{\widetilde\Ld{}^2_{\beta,\gamma}(\R^2)}.\]
For $|\gamma|\|W\|_{\Ld^\infty(\R^2)}<1$, this norm
is Lipschitz-equivalent to the standard norm on~$\Ld^2_\beta(\R^2)$,
\begin{equation}\label{eq:lip-norms-beta}
\big(1-\gamma\|W\|_{\Ld^\infty(\R^2)}\big)\|h\|_{\Ld^2_\beta(\R^2)}^2\,\le\,\|h\|_{\widetilde\Ld{}^2_{\beta,\gamma}(\R^2)}^2\,\le\,\big(1+\gamma\|W\|_{\Ld^\infty(\R^2)}\big)\|h\|_{\Ld^2_\beta(\R^2)}^2,
\end{equation}
and we denote by $\widetilde\Ld{}^2_{\beta,\gamma}(\R^2)$ the Hilbert space $\Ld^2_\beta(\R^2)$ endowed with this new structure.
From definition~\eqref{eq:Lbeta0}, a straightforward computation yields
\begin{multline*}
\langle g,iL_\beta(\gamma) h\rangle_{\widetilde{\Ld}{}^2_{\beta,\gamma}(\R^2)}
\,=\,
\beta^{-1}\int_{\R^2}(\overline{\nabla g})\cdot(\nabla^\bot h)\,\mu_\beta\\
+\gamma\iint_{\R^2\times\R^2} \overline{g(x)}h(y)\,K(x-y)\cdot\big(x\Omega_\beta(|x|)+y\Omega_\beta(|y|)\big)\,\mu_\beta(x)\mu_\beta(y)\,dxdy\\
-\gamma^2\beta^{-1}\iint_{\R^2\times\R^2}\overline{g(x)}h(y)\,\bigg(\int_{\R^2}K^\bot(x-z)\cdot K(y-z)\,\mu_\beta(z)\,dz\bigg)\,\mu_\beta(x)\mu_\beta(y)\,dxdy,
\end{multline*}
which shows that $L_\beta(\gamma)$ is symmetric on the deformed space $\widetilde\Ld{}^2_{\beta,\gamma}(\R^2)$.
As the generator of a $C_0$-group, it is therefore self-adjoint on this space.
Combining Stone's theorem with the Lipschitz property~\eqref{eq:lip-norms-beta}, we deduce for $|\gamma|\|W\|_{\Ld^\infty(\R^2)}\le\frac12$,
\[\|e^{itL_\beta(\gamma)}h\|_{\Ld^2_\beta(\R^2)}\,\lesssim\,\|e^{itL_\beta(\gamma)}h\|_{\widetilde\Ld{}^2_{\beta,\gamma}(\R^2)}\,=\,\|h\|_{\widetilde\Ld{}^2_{\beta,\gamma}(\R^2)}\,\lesssim\,\|h\|_{\Ld^2_\beta(\R^2)},\]
which proves that the generated $C_0$-group is indeed uniformly bounded on $\Ld^2_\beta(\R^2)$.

\medskip
\step2 Proof of~(ii) and~(iii) for $\gamma=0$.\\
As by assumption the function $\Omega_\beta$ is absolutely continuous and nowhere constant, we deduce from~\eqref{eq:Lbeta0-polar} that the kernel of~$L_\beta$ coincides with the set of radial functions and that its restriction to the orthogonal complement $E_\beta:=\ker(L_\beta)^\bot$ has purely absolutely continuous spectrum,
\begin{gather*}
\sigma_{\operatorname{ess}}(L_\beta|_{E_\beta})\,=\,\sigma_{\operatorname{ac}}(L_\beta|_{E_\beta})\,=\,\big\{k\lambda:k\in\Z\setminus\{0\},~\lambda\in\operatorname{ess.im}(\Omega_\beta)\big\},\\
\sigma_{\operatorname{sc}}(L_\beta|_{E_\beta})\,=\,\sigma_{\operatorname{pp}}(L_\beta|_{E_\beta})\,=\,\varnothing.
\end{gather*}
Moreover, we shall show that the restriction $L_\beta|_{E_\beta}$ satisfies the following limiting absorption principle: for all $g,h\in C^\infty_c(\R^2)\cap E_\beta$,
\begin{equation}\label{eq:LAP-0}
\sup_{\omega\in\C,\Im\omega\ne0}\big|\big\langle g,(L_\beta-\omega)^{-1}h\big\rangle_{\Ld^2_\beta(\R^2)}\big|\\
\,\lesssim_R\,
\3g\3_\beta\3h\3_\beta,
\end{equation}
where we recall that the norm $\3\cdot\3_\beta$ is defined in~\eqref{eq:def-N3}.

\medskip\noindent
Let $g,h\in C^\infty_c(\R^2)\cap E_\beta$ be fixed. In order to prove~\eqref{eq:LAP-0}, we start by using polar coordinates $x=(r,\theta)$ and Fourier series to write, for~$\Im\omega\ne0$,
\begin{equation}\label{eq:represent-Fourier-resolv}
\big\langle g,(L_\beta-\omega)^{-1}h\big\rangle_{\Ld^2_\beta(\R^2)}\,=\,\tfrac1{2\pi}\sum_{k\in\Z\setminus\{0\}}\int_0^\infty\frac{\overline{\hat g(r,k)}\hat h(r,k)}{k\Omega_\beta(r)-\omega}\,r\mu_\beta(r)\,dr,
\end{equation}
where we use the notation $\hat g(r,k):=\int_{\Sp^1}e^{-i\theta k}g(r,\theta)\,d\theta$ for Fourier coefficients with respect to the angle in polar coordinates,
and where we noticed that the condition $g,h\in E_\beta$ is equivalent to $\hat g(r,0)=\hat h(r,0)=0$.
To estimate this expression~\eqref{eq:represent-Fourier-resolv}, we proceed to a local analysis of the integral close to singularities $k\Omega_\beta(r)-\omega\approx0$. We start with the following two general estimates:
\begin{enumerate}[---]
\item for all $\phi\in C^\infty_c(\R)$ and $1<p<\infty$,
\begin{equation}\label{eq:model-crit-est}
\sup_{\e\ne0}\bigg|\int_{-1}^1 \tfrac{\phi(t)}{t+ i\e}\,dt\bigg|\,\lesssim_p\,\|\phi\|_{W^{\frac1p,p}(-1,1)},
\end{equation}
\item for all $\phi\in C^\infty_c(\R)$,
\begin{equation}\label{eq:model-crit-est-2}
\sup_{\e\ne0}\int_{0}^1 \Big|\tfrac{t\phi(t)}{t^2+ i\e}\Big|\,dt\,\lesssim\,
\|\tfrac1t\phi\|_{\Ld^1(0,1)}.
\end{equation}
\end{enumerate}
The first estimate follows from the Sobolev embedding $W^{\frac1p,p}(-1,1)\subset \Ld^\infty(-1,1)$, combined with the $\Ld^p$~theory for the Hilbert transform. The second estimate is obvious.

\medskip\noindent
{Using these bounds~\eqref{eq:model-crit-est}--\eqref{eq:model-crit-est-2} to estimate the integral in~\eqref{eq:represent-Fourier-resolv} close to singularities, using local deformations to reduce to these model situations, and recalling the non-degeneracy assumption~\eqref{eq:nondegenerate} for~$\Omega_\beta$, we are led to
\begin{multline}\label{eq:pre-est-L-om-res}
\sup_{\omega\in\C,\,\Im\omega\ne0}\big|\big\langle g,(L_\beta-\omega)^{-1}h\big\rangle_{\Ld^2_\beta(\R^2)}\big|\\
\,\lesssim_{R}\,\sum_{k\in\Z\setminus\{0\}}|k|^{-1}
\Big(\|\partial_r(\hat g(\cdot,k)\sqrt{\mu_\beta})\sqrt{r}\|_{\Ld^2(\R^+)}
+\| \langle\tfrac1r\rangle\hat g(\cdot,k)\sqrt{r\mu_\beta}\|_{\Ld^2(\R^+)}\Big)\\
\times\Big(\|\partial_r(\hat h(\cdot,k)\sqrt{\mu_\beta})\sqrt{r}\|_{\Ld^2(\R^+)}
+\| \langle\tfrac1r\rangle\hat h(\cdot,k)\sqrt{r\mu_\beta}\|_{\Ld^2(\R^+)}\Big).
\end{multline}
From this, the claim~\eqref{eq:LAP-0} follows by Plancherel's theorem.
For completeness, we include a detailed proof of the above estimate~\eqref{eq:pre-est-L-om-res}, which we split into three further substeps, distinguishing different types of contributions in the integral~\eqref{eq:represent-Fourier-resolv}.

\medskip
\substep{2.1} Reduction to singularities.\\
For $k\in\Z\setminus\{0\}$, we set $\omega/k=\Omega_\beta(0)-\alpha_k+i\e_k$ with $\alpha_k\in\R$ and $\e_k\ne0$.
In these terms, the integral in~\eqref{eq:represent-Fourier-resolv} reads
\begin{equation*}
\sum_{k\in\Z\setminus\{0\}}\bigg|\int_0^\infty\frac{\overline{\hat g(r,k)}\hat h(r,k)}{k\Omega_\beta(r)-\omega}\,r\mu_\beta(r)\,dr\bigg|
\,=\,\sum_{k\in\Z\setminus\{0\}}|k|^{-1}\bigg|\int_0^\infty\frac{\overline{\hat g(r,k)}\hat h(r,k)}{\tilde\Omega_\beta(r)-\alpha_k+i\e_k}\,r\mu_\beta(r)\,dr\bigg|,
\end{equation*}
with the short-hand notation $\tilde\Omega_\beta(r):=\Omega_\beta(0)-\Omega_\beta(r)$.
Given a constant $c_0>0$, the contribution of the integrals for $|\tilde\Omega_\beta(r)-\alpha_k|\ge c_0(r^2\wedge 1)$ is easily controlled: by the Cauchy--Schwarz inequality, we find
\begin{eqnarray*}
\lefteqn{\sum_{k\in\Z\setminus\{0\}}|k|^{-1}\int_0^\infty\mathds1_{|\tilde\Omega_\beta(r)-\alpha_k|\ge c_0(r^2\wedge 1)}\bigg|\frac{\overline{\hat g(r,k)}\hat h(r,k)}{\tilde\Omega_\beta(r)-\alpha_k+i\e_k}\bigg|\,r\mu_\beta(r)\,dr}\\
&\le&c_0^{-1}\sum_{k\in\Z\setminus\{0\}}|k|^{-1}\int_0^\infty\langle\tfrac1r\rangle^2|\hat g(r,k)||\hat h(r,k)|r\mu_\beta(r)\,dr\\
&\le&c_0^{-1}\sum_{k\in\Z\setminus\{0\}}|k|^{-1}\|\langle\tfrac1r\rangle\hat g(r,k)\sqrt{r\mu_\beta}\|_{\Ld^2(\R^+)}\|\langle\tfrac1r\rangle\hat h(r,k)\sqrt{r\mu_\beta}\|_{\Ld^2(\R^+)},
\end{eqnarray*}
and it thus remains the estimate the contribution of the integral for $|\tilde\Omega_\beta(r)-\alpha_k|\le c_0(r^2\wedge1)$, that is, close to singularities.

\medskip
\substep{2.2} Contribution of singularities with $r\gtrsim1$.\\
Given $r_0>0$, if $r\ge r_0$, the assumptions on $\Omega_\beta$ entail $\tilde\Omega_\beta(r)\gtrsim r$.
Given $c_0>0$, by Step~2.1, it suffices to consider the contribution of $|\tilde\Omega_\beta(r)-\alpha_k|\le c_0$. Choosing $c_0$ small enough (independently of $\alpha_k$), we note that these restrictions actually ensure $|\alpha_k|\gtrsim1$.
As $\tilde\Omega_\beta:\R^+\to\R^+$ is a homeomorphism, changing variables in the restricted integral, we find
\begin{multline*}
\sum_{k\in\Z\setminus\{0\}}|k|^{-1}\bigg|\int_0^\infty\mathds1_{|\tilde\Omega_\beta(r)-\alpha_k|\le c_0}\frac{\overline{\hat g(r,k)}\hat h(r,k)}{\tilde\Omega_\beta(r)-\alpha_k+i\e_k}\,r\mu_\beta(r)\,dr\bigg|\\
\,=\,\sum_{k\in\Z\setminus\{0\}}|k|^{-1}\bigg|\int_{-c_0}^{c_0}\frac{\overline{\hat g(\tilde\Omega_\beta^{-1}(s+\alpha_k),k)}\hat h(\tilde\Omega_\beta^{-1}(s+\alpha_k),k)}{s+i\e_k}\\[-2mm]
\times\frac{\tilde\Omega_\beta^{-1}(s+\alpha_k)\mu_\beta(\tilde\Omega_\beta^{-1}(s+\alpha_k))}{\tilde\Omega_\beta'(\tilde\Omega_\beta^{-1}(s+\alpha_k))}\,ds\bigg|.
\end{multline*}
We now appeal to~\eqref{eq:model-crit-est} to estimate the right-hand side.
To avoid fractional Sobolev norms, we shall actually combine this estimate with H\"older's inequality and the Sobolev embedding $H^1(\R^+)\subset W^{\frac1p,q}(\R^+)$ for $1<p,q<\infty$ with~$\frac1p=\frac1q+\frac12$, in form of
\begin{eqnarray}
\|ab\|_{W^{\frac1p,p}(\R^+)}&\lesssim_{p,q}&\|a\|_{W^{\frac1p,q}(\R^+)}\|b\|_{\Ld^2(\R^+)}+\|a\|_{\Ld^2(\R^+)}\|b\|_{W^{\frac1p,q}(\R^+)}\nonumber\\
&\lesssim_{p,q}&\|a\|_{H^1(\R^+)}\|b\|_{\Ld^2(\R^+)}+\|a\|_{\Ld^2(\R^+)}\|b\|_{H^1(\R^+)}.\label{eq:estim-prod-sob}
\end{eqnarray}
Using this together with~\eqref{eq:model-crit-est}, the above integral is controlled as follows,
\begin{multline*}
\sum_{k\in\Z\setminus\{0\}}|k|^{-1}\bigg|\int_0^\infty\mathds1_{|\tilde\Omega_\beta(r)-\alpha_k|\le c_0}\frac{\overline{\hat g(r,k)}\hat h(r,k)}{\tilde\Omega_\beta(r)-\alpha_k+i\e_k}\,r\mu_\beta(r)\,dr\bigg|\\
\,\lesssim\,\sum_{k\in\Z\setminus\{0\}}|k|^{-1}\bigg\|\hat g(\tilde\Omega_\beta^{-1},k)\Big(\frac{\tilde\Omega_\beta^{-1}\mu_\beta(\tilde\Omega_\beta^{-1})}{\tilde\Omega_\beta'(\tilde\Omega_\beta^{-1})}\Big)^\frac12\bigg\|_{H^1(\alpha_k-c_0,\alpha_k+c_0)}\\
\times\bigg\|\hat h(\tilde\Omega_\beta^{-1},k)\Big(\frac{\tilde\Omega_\beta^{-1}\mu_\beta(\tilde\Omega_\beta^{-1})}{\tilde\Omega_\beta'(\tilde\Omega_\beta^{-1})}\Big)^\frac12\bigg\|_{\Ld^2(\alpha_k-c_0,\alpha_k+c_0)}+\sym,
\end{multline*}
where the symbol `$\sym$' stands for the expression preceding it with $g$ and $h$ swapped. Computing the derivative in the $H^1$ norm, changing variables back to the original ones, and noting that the assumptions on $\Omega_\beta$ ensure $\tilde\Omega_\beta'\gtrsim1$ and $|\tilde\Omega_\beta''|\lesssim1$ on the integration domain, we deduce after straightforward calculations,
\begin{multline*}
\sum_{k\in\Z\setminus\{0\}}|k|^{-1}\bigg|\int_0^\infty\mathds1_{|\tilde\Omega_\beta(r)-\alpha_k|\le c_0}\frac{\overline{\hat g(r,k)}\hat h(r,k)}{\tilde\Omega_\beta(r)-\alpha_k+i\e_k}\,r\mu_\beta(r)\,dr\bigg|\\
\,\lesssim\,\sum_{k\in\Z\setminus\{0\}}|k|^{-1}\Big(\|\hat g(\cdot,k)\sqrt{r\mu_\beta}\|_{\Ld^2(\R^+)}+\|\partial_r(\hat g(\cdot,k)\sqrt{r\mu_\beta})\|_{\Ld^2(\R^+)}\Big)\\[-3mm]
\times\|\hat h(\cdot,k)\sqrt{r\mu_\beta}\|_{\Ld^2(\R^+)}
+\sym,
\end{multline*}
which gives the desired estimate~\eqref{eq:pre-est-L-om-res} for the restricted integral.

\medskip
\substep{2.3} Contribution of singularities with $r\ll1$.\\
For $r\le1$, by Step~2.1, given $c_0>0$, it suffices to consider the contribution of $|\tilde\Omega_\beta(r)-\alpha_k|\le c_0r^2$. As the assumptions on $\Omega_\beta$ ensure $\tilde\Omega_\beta(r)\simeq r^2$ for $r\le1$, we deduce that it suffices to consider the contribution of $|\tilde\Omega_\beta(r)-\alpha_k|\le c_0\tilde\Omega_\beta(r)$.
Choosing $c_0<\frac12$, this condition implies in particular $\alpha_k\simeq\tilde\Omega_\beta(r)$, and thus $r\simeq\sqrt\alpha_k$.
As by Step~2.2 we can restrict to the case $r\ll1$, this means $0<\alpha_k\ll1$, and in this setting it remains to consider the contribution of $|\tilde\Omega_\beta(r)-\alpha_k|\le c_0\alpha_k$.
In terms of the homeomorphism $\Psi_k(r):=\alpha_k(\tilde\Omega_\beta(r)-\alpha_k)$, changing variables in the restricted integral, we find
\begin{multline*}
\sum_{k\in\Z\setminus\{0\}}|k|^{-1}\bigg|\int_{0}^{\infty}\mathds1_{|\tilde\Omega_\beta(r)-\alpha_k|\le c_0\alpha_k}\frac{\overline{\hat g(r,k)}\hat h(r,k)}{\tilde\Omega_\beta(r)-\alpha_k+i\e_k}\,r\mu_\beta(r)\,dr\bigg|\\
\,=\,\sum_{k\in\Z\setminus\{0\}}|k|^{-1}\alpha_k\bigg|\int_{-c_0}^{c_0}\frac{\overline{\hat g(\Psi_k^{-1}(s),k)}\hat h(\Psi_k^{-1}(s),k)}{s+i\e_k\alpha_k}\,\frac{\Psi_k^{-1}(s)\mu_\beta(\Psi_k^{-1}(s))}{\Psi_k'(\Psi_k^{-1}(s))}\,ds\bigg|.
\end{multline*}
From here, using~\eqref{eq:model-crit-est}, \eqref{eq:estim-prod-sob}, and the assumptions on $\Omega_\beta$ similarly as in Step~2.2, the desired estimate~\eqref{eq:pre-est-L-om-res} follows for the restricted integral. Combining with Steps~2.1 and~2.2, this concludes the proof of~\eqref{eq:pre-est-L-om-res}, hence of the limiting absorption principle~\eqref{eq:LAP-0} by Plancherel's theorem.
}

\medskip
\step3 Proof of~(ii) and~(iii).\\
With the above spectral properties of $L_\beta$ at hand, we now turn to the corresponding properties of the perturbed operator $L_\beta(\gamma)=L_\beta+\gamma T_\beta$ and we conclude the proof of items~(ii) and~(iii).
First note that by definition~\eqref{eq:Lbeta0} the kernel of $L_\beta(\gamma)$ clearly contains the set of radial functions.
By density, we then deduce that item~(ii) would follow from item~(iii), and it thus remains to prove the latter.
For that purpose, we shall argue perturbatively for~$|\gamma|$ small enough, combining the above limiting absorption principle~\eqref{eq:LAP-0} for $L_\beta$ together with regularizing properties of $T_\beta$. By definition~\eqref{eq:Lbeta0}, we can write for $\Im\omega\ne0$,
\begin{equation*}
\big(T_\beta(L_\beta-\omega)^{-1}h\big)(x)\,=\,-i x\Omega_\beta(x)\cdot\big\langle K(x-\cdot),(L_\beta-\omega)^{-1}h\big\rangle_{\Ld^2_\beta(\R^2)}.
\end{equation*}
Taking the norm $\3\cdot\3_\beta$ defined in~\eqref{eq:def-N3},
appealing to the result~\eqref{eq:LAP-0} of Step~2,
{using that the assumptions on $V$ and $\Omega_\beta$ ensure $|\Omega_\beta(r)|\lesssim1+r$ and $|\Omega'_\beta(r)|\lesssim1$, and recalling~\eqref{eq:non-Gauss-mu-est},}
we deduce
\begin{eqnarray*}
\sup_{\omega\in\C,\Im\omega\ne0}\3 T_\beta(L_\beta-\omega)^{-1}h\3_\beta
&\lesssim_R&\big(1+\|\langle \cdot\rangle^2\|_{\Ld^2_\beta(\R^2)}+\beta\|\langle \cdot\rangle^4\|_{\Ld^2_\beta(\R^2)}\big)\3h\3_\beta\\
&\lesssim_R&\beta^{-\frac23}\3h\3_\beta.
\end{eqnarray*}
Hence, for $\gamma\ll_R\beta^{\frac23}$ small enough,
\[\sup_{\omega\in\C,\Im \omega\ne0}\3\gamma T_\beta(L_\beta-\omega)^{-1}h\3_\beta\,\le\,\tfrac12\3h\3_\beta.\]
This allows us to construct the Neumann series
\begin{equation}\label{eq:Neumann/Lresolv}
\big\langle g,(L_\beta(\gamma)-\omega)^{-1}h\big\rangle_{\Ld^2_\beta(\R^2)}\,=\,\sum_{n=0}^\infty\big\langle g,(L_\beta-\omega)^{-1}\big[\gamma T_\beta (L_\beta-\omega)^{-1}\big]^nh\big\rangle_{\Ld^2_\beta(\R^2)},
\end{equation}
and the bound~\eqref{eq:LAP-Lbeta} follows.
This ends the proof of item~(iii).
\end{proof}

Next, we establish the following estimates on BBGKY operators in the weighted negative Sobolev spaces defined in Section~\ref{sec:def-weight-Sob}.

\begingroup\allowdisplaybreaks
\begin{lem}\label{lem:prelest-LS-re}$ $
\begin{enumerate}[(i)]
\item \emph{Weak bounds on BBGKY operators:}\\
For all $1\le m\le N$, $s\ge0$, and $h^{m+r}\in C^\infty_c((\R^2)^{m+r})$ for $r\in\{-2,-1,0,1\}$, we have
\begin{eqnarray*}
\|L^{m}_{N,\beta}h^{m}\|_{H^{-s-1}_\beta((\R^2)^m)}&\lesssim_{m,s}&\|h^{m}\|_{H^{-s}_\beta((\R^2)^{m})},\\
\|S^{m,\circ}_{N,\beta}h^{m}\|_{H^{-s-1}_\beta((\R^2)^m)}&\lesssim_{m,s}&\|h^{m}\|_{H^{-s}_\beta((\R^2)^{m})},\\
\|S^{m,+}_{N,\beta}h^{m+1}\|_{H^{-s-1}_\beta((\R^2)^m)}&\lesssim_{m,s}&\|h^{m+1}\|_{H^{-s}_\beta((\R^2)^{m+1})},\\
\|S^{m,-}_{N,\beta}h^{m-1}\|_{H^{-s-1}_\beta((\R^2)^m)}&\lesssim_{m,s}&\|h^{m-1}\|_{H^{-s}_\beta((\R^2)^{m-1})},\\
\|S^{m,=}_{N,\beta}h^{m-2}\|_{H^{-s}_\beta((\R^2)^m)}&\lesssim_{m,s}&\|h^{m-2}\|_{H^{-s}_\beta((\R^2)^{m-2})}.
\end{eqnarray*}
\item \emph{Weak bounds on linearized mean-field evolutions:}\\
For all $1\le m\le N$, $s\ge0$, $\delta>0$, and $h^m\in C^\infty_c((\R^2)^{m})$, we have
\[\|e^{iL_{N,\beta}^mt}h^m\|_{H^{-\lceil3s/2\rceil}_\beta((\R^2)^m)}\,\lesssim_{R,\beta,m,s,\delta}\,\langle t\rangle^{s+\delta}\|h^m\|_{H^{-s}_\beta((\R^2)^m)}.\]
\end{enumerate}
\end{lem}

\begin{proof}
By duality, item~(i) follows from the following corresponding estimates on the adjoint operators, for all $s\ge0$ and $h^m\in C^\infty_c((\R^2)^m)$,
\begin{eqnarray*}
\|(L^{m}_{N,\beta})^*h^{m}\|_{H^{s}_\beta((\R^2)^m)}&\lesssim_{m,s}&\|h^{m}\|_{H^{s+1}_\beta((\R^2)^{m})},\\
\|(S^{m,\circ}_N)^*h^{m}\|_{H^{s}_\beta((\R^2)^m)}&\lesssim_{m,s}&\|h^{m}\|_{H^{s+1}_\beta((\R^2)^{m})},\\
\|(S^{m,+}_N)^*h^m\|_{H^{s}_\beta((\R^2)^{m+1})}&\lesssim_{m,s}&\|h^{m}\|_{H^{s+1}_\beta((\R^2)^{m})},\\
\|(S^{m,-}_N)^*h^{m}\|_{H^{s}_\beta((\R^2)^{m-1})}&\lesssim_{m,s}&\|h^{m}\|_{H^{s+1}_\beta((\R^2)^{m})},\\
\|(S^{m,=}_N)^*h^{m}\|_{H^{s}_\beta((\R^2)^{m-2})}&\lesssim_{m,s}&\|h^{m}\|_{H^{s}_\beta((\R^2)^{m})}.
\end{eqnarray*}
Adjoints can be computed explicitly and these estimates easily follow; we omit the details.
We turn to the proof of~(ii).
By duality, it suffices to prove that, for all $s\ge0$ and $\delta>0$,
\[\|e^{i(L_{N,\beta}^m)^*t}h^m\|_{H^{s}_\beta((\R^2)^m)}\,\lesssim_{R,\beta,m,s,\delta}\,\langle t\rangle^{s+\delta}\|h^m\|_{H^{3s/2}_\beta((\R^2)^m)}.\]
As by definition $L_{N,\beta}^m$ is a Kronecker sum of the perturbed operator $L_\beta(\gamma)=L_\beta+\gamma T_\beta$ for different values of $0\le\gamma\le\beta$, cf.~\eqref{eq:L-Kron-NG}, it suffices to show that for all $s\ge0$, $\delta>0$, $\gamma\le\beta$, and $h\in C^\infty_c(\R^2)$,
\begin{eqnarray}
\|e^{iL_\beta t}h\|_{H^{s}_\beta(\R^2)}&\lesssim_{R,\beta,s}&\langle t\rangle^{s}\|h\|_{H^{\lceil3s/2\rceil}_\beta(\R^2)},\label{eq:estimHs-Lbeta-1}\\
\|e^{iL_\beta(\gamma)^*t}h\|_{H^{s}_\beta(\R^2)}&\lesssim_{R,\beta,s,\delta}&\langle t\rangle^{s+\delta}\|h\|_{H^{\lceil3s/2\rceil}_\beta(\R^2)}.\label{eq:estimHs-Lbeta-2}
\end{eqnarray}
We split the proof into two steps, separately proving these two estimates.

\medskip\noindent
\step1 Proof of~\eqref{eq:estimHs-Lbeta-1}.\\
Starting from the explicit expression for the flow
\[e^{iL_\beta t}h(x)\,=\,h\big(x\cos(t\Omega_\beta(x))+x^\bot\sin(t\Omega_\beta(x))\big),\]
we find for all $s\ge0$,
\begin{equation}\label{eq:estimHs-iLt-pr}
\int_{\R^2} |\nabla^s e^{iL_\beta t}h|^2\mu_\beta\,\lesssim_{s}\,\langle t\rangle^{2s}\int_{\R^2}\langle\cdot\rangle^{2s}|\langle\nabla\rangle^s h|^2\mu_\beta,
\end{equation}
where the multiplicative constant depends on $\|\nabla\Omega_\beta\|_{W^{s,\infty}(\R^2)}$.
Using assumption~\eqref{eq:nondegenerate} in form of
\[\langle x\rangle^2\mu_\beta(x)\,\lesssim_R\,\mu_\beta(x)-\beta^{-1}\tfrac{x}{|x|}\cdot\nabla\mu_\beta(x),\]
an integration by parts gives for any $s\ge0$,
\[\int_{\R^2}\langle\cdot\rangle^{2s}|\nabla^s h|^2\mu_\beta\,\lesssim_{R,\beta,s}\,\int_{\R^2}\langle\cdot\rangle^{2s-2}|\nabla^s h|^2\mu_\beta
+\int_{\R^2}\langle\cdot\rangle^{2s-2}|\nabla^s h||\nabla^{s+1} h|\mu_\beta.\]
Hence, by the Cauchy--Schwarz inequality,
\[\int_{\R^2}\langle\cdot\rangle^{2s}|\nabla^s h|^2\mu_\beta\,\lesssim_{R,\beta,s}\,\int_{\R^2}\langle\cdot\rangle^{2s-2}|\nabla^s h|^2\mu_\beta
+\int_{\R^2}\langle\cdot\rangle^{2s-4}|\nabla^{s+1} h|^2\mu_\beta,\]
which yields by induction for all $s\ge0$,
\[\int_{\R^2}\langle\cdot\rangle^{2s}|\nabla^s h|^2\mu_\beta\,\lesssim_{R,\beta,s}\,\|h\|_{H^{\lceil3s/2\rceil}_\beta(\R^2)}^2.\]
Combined with~\eqref{eq:estimHs-iLt-pr}, we obtain the claim~\eqref{eq:estimHs-Lbeta-1}.

\medskip\noindent
\step2 Proof of~\eqref{eq:estimHs-Lbeta-2}.\\
Let $0\le\gamma\le\beta$. Decomposing $L_\beta(\gamma)^*=L_\beta+\gamma T_\beta^*$, we start with Duhamel's formula in the form
\[e^{iL_\beta(\gamma)^*t}h\,=\,e^{iL_\beta t}h+\gamma\int_0^t e^{iL_\beta(t-t')}iT_\beta^*e^{iL_\beta(\gamma)^*t'}h\,dt'.\]
Taking the Sobolev norm and applying~\eqref{eq:estimHs-Lbeta-1}, we deduce for all $r\ge0$,
\[\big\|e^{iL_\beta(\gamma)^*t}h-e^{iL_\beta t}h\big\|_{H^r_\beta(\R^2)}\,\lesssim_{R,\beta,r}\,\langle t\rangle^r\int_0^t \|\gamma T_\beta^*e^{iL_\beta(\gamma)^*t'}h\|_{H^{\lceil 3r/2\rceil}_\beta(\R^2)}\,dt'.\]
As the operator $T_\beta^*$ satisfies the following regularizing property, for all $r\ge0$,
\[\|\gamma T_\beta^* h\|_{H^r_\beta(\R^2)}\,\lesssim_r\,\|h\|_{\Ld^2_\beta(\R^2)},\]
we deduce
\begin{equation*}
\big\|e^{iL_\beta(\gamma)^*t}h-e^{iL_\beta t}h\big\|_{H^r_\beta(\R^2)}
\,\lesssim_{R,\beta,r}\,\langle t\rangle^r\int_0^t \|e^{iL_\beta(\gamma)^*t'}h\|_{\Ld^2_\beta(\R^2)}\,dt'.
\end{equation*}
Now we combine this with the interpolation inequality of Lemma~\ref{lem:interpol-SobG}(i): for all $r\ge s\ge0$, we get
\begin{eqnarray*}
\lefteqn{\big\|e^{iL_\beta(\gamma)^*t}h-e^{iL_\beta t}h\big\|_{H^s_\beta(\R^2)}}\\
&\lesssim_{R,\beta,r,s}&\big\|e^{iL_\beta(\gamma)^*t}h-e^{iL_\beta t}h\big\|_{\Ld^2_\beta(\R^2)}^{1-\frac sr}\big\|e^{iL_\beta(\gamma)^*t}h-e^{iL_\beta t}h\big\|_{H^r_\beta(\R^2)}^{\frac sr}\\
&\lesssim_{R,\beta,r}&\langle t\rangle^s\big\|e^{iL_\beta(\gamma)^*t}h-e^{iL_\beta t}h\big\|_{\Ld^2_\beta(\R^2)}^{1-\frac sr}\Big(\int_0^t\|e^{iL_\beta(\gamma)^*t'}h\|_{\Ld^2_\beta(\R^2)}\,dt'\Big)^{\frac sr}.
\end{eqnarray*}
By the triangle inequality and by the uniform boundedness of Lemma~\ref{lem:nonGaus}(i), we finally arrive at
\begin{eqnarray*}
\|e^{iL_\beta(\gamma)^*t}h\|_{H^s_\beta(\R^2)}
\,\lesssim_{R,\beta,r,s}\,\|e^{iL_\beta t}h\|_{H^s_\beta(\R^2)}
+\langle t\rangle^{s(1+\frac1r)}\|h\|_{\Ld^2_\beta(\R^2)}.
\end{eqnarray*}
Combined with~\eqref{eq:estimHs-Lbeta-1}, this yields the claim~\eqref{eq:estimHs-Lbeta-2} after choosing $r\ge \frac s\delta$.
\end{proof}
\endgroup

\subsection{Proof of Theorem~\ref{th:main-nonGauss-re}}
We start by using the cumulant estimates of Lemma~\ref{lem:est-cum-2} to truncate the BBGKY hierarchy and get a closed description of the tagged particle density.

\begin{lem}\label{lem:exp-gn1-NG}
For all $t\ge0$ and $\delta>0$, we have
\begin{align*}
&\Big\|N\partial_t\langle g_N^1\rangle-\int_0^t\big\langle iS_{N,\beta}^{1,+}e^{-i(t-s)L_{N,\beta}^2} iS_{N,\beta}^{2,-}(\tfrac1{\mu_\beta}f^\circ)\big\rangle\,ds\\
&\hspace{1cm}-\int_0^t\!\!\int_0^s\big\langle iS_{N,\beta}^{1,+}e^{-i(t-s)L_{N,\beta}^2}iS_{N,\beta}^{2,+}e^{-i(s-s')L_{N,\beta}^3}iS_{N,\beta}^{3,=}(\tfrac1{\mu_\beta}f^\circ)\big\rangle\,ds'\,ds\Big\|_{H^{-7}_\beta(\R^2)}\\
&\hspace{9cm}\,\lesssim_{R,\beta,\delta}\,\langle t\rangle^{10+\delta} N^{-\frac12}.
\end{align*}
\end{lem}

\begin{proof}
First, for the tagged particle density, the BBGKY hierarchy~\eqref{eq:eqn-gNm-21} in Lemma~\ref{lem:eqn-gNm-21} yields
\[\partial_tg_N^1+iL_{N,\beta}^1g_N^1\,=\,iS_{N,\beta}^{1,+}g_N^2+\tfrac1NiS_{N,\beta}^{1,\circ}g_N^1,\]
and thus, for the radial density, noting that the contributions of $L_{N,\beta}^1$ and $S_{N,\beta}^{1,\circ}$ disappear when taking angular averages,
\begin{equation}\label{eq:gN1-1-NG}
N\partial_t\langle g_N^1\rangle\,=\,\langle iS_{N,\beta}^{1,+}Ng_N^2\rangle.
\end{equation}
The Duhamel formula for the above equation also yields
\[g_N^{1}(t)\,=\,e^{-itL_{N,\beta}^1}(\tfrac1{\mu_\beta}f^\circ)+\int_0^te^{-i(t-s)L_{N,\beta}^1}\Big(iS_{N,\beta}^{1,+}g_N^{2}(s)+\tfrac1NiS_{N,\beta}^{1,\circ}g_N^{1}(s)\Big)\,ds,\]
or equivalently, since $\tfrac1{\mu_\beta}f^\circ$ is radial and thus belongs to the kernel of~$L_{N,\beta}^1$,
\begin{equation}\label{eq:Duhamel-g1}
g_N^{1}(t)\,=\,\tfrac1{\mu_\beta}f^\circ+\int_0^te^{-i(t-s)L_{N,\beta}^1}\Big(iS_{N,\beta}^{1,+}g_N^{2}(s)+\tfrac1NiS_{N,\beta}^{1,\circ}g_N^{1}(s)\Big)\,ds.
\end{equation}
Next, the BBGKY hierarchy~\eqref{eq:eqn-gNm-21} for $g_N^2$ and $g_N^3$ yields
\begin{eqnarray*}
\partial_tNg_N^2+iL_{N,\beta}^2Ng_N^2&=&iS_{N,\beta}^{2,+}Ng_N^3+iS_{N,\beta}^{2,\circ}g_N^2+iS_{N,\beta}^{2,-}g_N^1,\\
\partial_tNg_N^3+iL_{N,\beta}^3Ng_N^3&=&iS_{N,\beta}^{3,+}Ng_N^4+iS_{N,\beta}^{3,\circ}g_N^3+iS_{N,\beta}^{3,-}g_N^2+iS_{N,\beta}^{3,=}g_N^1,
\end{eqnarray*}
and thus, combining the corresponding Duhamel formulas,
\begin{multline*}
Ng_N^{2}(t)\,=\,\int_0^te^{-i(t-s)L_{N,\beta}^2}\Big(iS_{N,\beta}^{2,\circ}g_N^{2}(s)+iS_{N,\beta}^{2,-}g_N^{1}(s)\Big)\,ds\\
+\int_0^t\int_0^se^{-i(t-s)L_{N,\beta}^2}iS_{N,\beta}^{2,+}e^{-i(s-s')L_{N,\beta}^3}\\
\times\Big(iS_{N,\beta}^{3,+}Ng_N^{4}(s')+iS_{N,\beta}^{3,\circ}g_N^{3}(s')+iS_{N,\beta}^{3,-}g_N^{2}(s')+iS_{N,\beta}^{3,=}g_N^{1}(s')\Big)\,ds'\,ds.
\end{multline*}
Replacing $g_N^{1}$ in this expression by~\eqref{eq:Duhamel-g1}, and reorganizing the terms, we deduce
\begingroup\allowdisplaybreaks
\begin{multline*}
Ng_N^{2}(t)\,=\,\int_0^te^{-i(t-s)L_{N,\beta}^2} iS_{N,\beta}^{2,-}(\tfrac1{\mu_\beta}f^\circ)\,ds\\
+\int_0^t\int_0^se^{-i(t-s)L_{N,\beta}^2}iS_{N,\beta}^{2,+}e^{-i(s-s')L_{N,\beta}^3}iS_{N,\beta}^{3,=}(\tfrac1{\mu_\beta}f^\circ)\,ds'\,ds\\
+\int_0^te^{-i(t-s)L_{N,\beta}^2}iS_{N,\beta}^{2,\circ}g_N^{2}(s)\,ds\\
+\int_0^t\int_0^se^{-i(t-s)L_{N,\beta}^2}iS_{N,\beta}^{2,+}e^{-i(s-s')L_{N,\beta}^3}\Big(iS_{N,\beta}^{3,+}Ng_N^{4}(s')+iS_{N,\beta}^{3,\circ}g_N^{3}(s')+iS_{N,\beta}^{3,-}g_N^{2}(s')\Big)\,ds'\,ds\\
+\int_0^t\int_0^se^{-i(t-s)L_{N,\beta}^2} iS_{N,\beta}^{2,-}e^{-i(s-s')L_{N,\beta}^1}\Big(iS_{N,\beta}^{1,+}g_N^{2}(s')+\tfrac1NiS_{N,\beta}^{1,\circ}g_N^{1}(s')\Big)\,ds'\,ds\\
+\int_0^t\int_0^s\int_0^{s'}e^{-i(t-s)L_{N,\beta}^2}iS_{N,\beta}^{2,+}e^{-i(s-s')L_{N,\beta}^3}iS_{N,\beta}^{3,=}e^{-i(s'-s'')L_{N,\beta}^1}\\
\times\Big(iS_{N,\beta}^{1,+}g_N^{2}(s'')+\tfrac1NiS_{N,\beta}^{1,\circ}g_N^{1}(s'')\Big)\,ds''\,ds'\,ds.
\end{multline*}
\endgroup
Inserting this into~\eqref{eq:gN1-1-NG} and appealing to Lemma~\ref{lem:prelest-LS-re} to estimate the different terms, the conclusion follows.
\end{proof}

We appeal to Laplace transform to express long-time linear evolutions more conveniently in terms of associated resolvents. The representation is further simplified by noting that the resolvent of $iL_{N,\beta}^3$ can be explicitly computed on some specific test functions.
Before stating the result, we introduce some notation:
given $R_1,R_2\in C^\infty_b((\R^2)^2)$, we define the $\mu_\beta$-convolution product
\[R_1\ast_{\mu_\beta}R_2(x,y)\,:=\,\int_{\R^2}R_1(x,z)R_2(z,y)\,\mu_\beta(z)\,dz.\]
Note that this product is in general not commutative, but is always associative. We define corresponding $\mu_\beta$-convolution powers of an element $R\in C^\infty_b((\R^2)^2)$ iteratively by
\[R^{\ast_{\mu_\beta}(n+1)}\,:=\,R\ast_{\mu_\beta}R^{\ast_{\mu_\beta}n},\qquad R^{\ast_{\mu_\beta}1}=R,\qquad \text{for all $n\ge1$}.\]
For a function $S\in C^\infty(\R^2)$, identifying it with $\tilde S(x,y):=S(x-y)$, we similarly define with some mild abuse of notation,
\[S^{\ast_{\mu_\beta}n}(x,y)\,:=\,(\tilde S)^{\ast_{\mu_\beta}n}(x,y),\qquad \text{for all $n\ge1$},\]
which coincides with the more explicit definition given in~\eqref{eq:def-Wconvsp}.
In these terms, the following result holds.

\begin{lem}\label{lem:laplace-NG}
For all $\sigma\ge0$ and $\phi\in C^\infty_c(\R^+)$, considering the subcritically time-rescaled tagged particle density
\[\bar g_N^{1}(\tau)\,:=\, g_N^1(N^\sigma\tau),\]
we have for all $\delta>0$,
\begin{multline*}
\Big\|N^{1-\sigma}\int_0^\infty\phi\,\partial_\tau \langle\bar g_N^1\rangle\,d\tau-\int_\R g_\phi(\alpha)\,\big\langle iS_{N,\beta}^{1,+}(iL_{N,\beta}^2+\tfrac{i\alpha+1}{N^\sigma})^{-1} iS_{N,\beta}^{2,-}(\tfrac1{\mu_\beta}f^\circ)\big\rangle\,d\alpha\\
-\int_\R g_\phi(\alpha)\,\big\langle iS_{N,\beta}^{1,+}(iL_{N,\beta}^2+\tfrac{i\alpha+1}{N^\sigma})^{-1}iS_{N,\beta}^{2,+}H_\beta^{3;\circ}\big\rangle\,d\alpha\\
+\int_\R\tfrac{i\alpha+1}{N^\sigma} g_\phi(\alpha)\,\big\langle iS_{N,\beta}^{1,+}(iL_{N,\beta}^2+\tfrac{i\alpha+1}{N^\sigma})^{-1}iS_{N,\beta}^{2,+}(iL_{N,\beta}^3+\tfrac{i\alpha+1}{N^\sigma})^{-1}H_\beta^{3;\circ}\big\rangle\,d\alpha\Big\|_{H^{-7}_\beta(\R^2)}\\
\,\lesssim_{R,\beta,\delta}\,N^{(10+\delta)\sigma-\frac12}\|\langle\cdot\rangle^{10+\delta}\phi\|_{\Ld^1(\R^+)},
\end{multline*}
where
\begin{eqnarray}
H^{3;\circ}_\beta(x_1,x_2,x_3)&:=&-\beta(\tfrac1{\mu_\beta}f^\circ)(x_1)W_{N,\beta}(x_2,x_3),\label{eq:def-Hbeta30-00}\\
W_{N,\beta}(x_2,x_3)&:=&\sum_{n=0}^\infty(-\beta)^n(\tfrac{N-3}N)^n\,W^{\ast_{\mu_\beta}(n+1)}(x_2,x_3),\nonumber
\end{eqnarray}
while $g_\phi(\alpha):=\frac1{2\pi}\int_0^\infty\frac{e^{(i\alpha+1)\tau}}{i\alpha+1}\phi(\tau)\,d\tau$ belongs to $C_b^\infty(\R)$ and satisfies
\[|g_\phi(\alpha)|\lesssim_\phi\langle\alpha\rangle^{-2}\qquad\text{and}\qquad\int_\R g_\phi=\int_0^\infty\phi.\]
\end{lem}

\begin{proof}
Appealing to the product formula for the Laplace transform, as e.g.~\cite[Lemma~5.2]{DSR-21}, we can readily deduce from Lemma~\ref{lem:exp-gn1-NG} that
\begin{multline}\label{eq:pre/lem:laplace-NG}
\Big\|N^{1-\sigma}\int_0^\infty\phi\,\partial_\tau \langle\bar g_N^1\rangle-\int_\R g_\phi(\alpha)\,\big\langle iS_{N,\beta}^{1,+}(iL_{N,\beta}^2+\tfrac{i\alpha+1}{N^\sigma})^{-1} iS_{N,\beta}^{2,-}(\tfrac1{\mu_\beta}f^\circ)\big\rangle\,d\alpha\\
-\int_\R g_\phi(\alpha)\,\big\langle iS_{N,\beta}^{1,+}(iL_{N,\beta}^2+\tfrac{i\alpha+1}{N^\sigma})^{-1}iS_{N,\beta}^{2,+}(iL_{N,\beta}^3+\tfrac{i\alpha+1}{N^\sigma})^{-1}iS_{N,\beta}^{3,=}(\tfrac1{\mu_\beta}f^\circ)\big\rangle\,d\alpha\Big\|_{H^{-7}_\beta(\R^2)}\\
\,\lesssim_\delta\,N^{(10+\delta)\sigma-\frac12}\|\langle\cdot\rangle^{10+\delta}\phi\|_{\Ld^1(\R^+)},
\end{multline}
where the transformation $g_\phi$ is as in the statement.
By definition of $S_{N,\beta}^{3,=}$ in Lemma~\ref{lem:eqn-gNm-21}, we have
\begin{eqnarray*}
iS_{N,\beta}^{3,=}(\tfrac1{\mu_\beta}f^\circ)(x_1,x_2,x_3)&=&-\beta (\tfrac1{\mu_\beta}f^\circ)(x_1) \,K(x_2-x_3)\cdot \big(x_2\Omega_\beta(x_2)-x_3\Omega_\beta(x_3)\big)\\
&=&-\beta (\tfrac1{\mu_\beta}f^\circ)(x_1) \,\nabla W(x_2-x_3)\cdot \big(x_2^\bot\Omega_\beta(x_2)-x_3^\bot\Omega_\beta(x_3)\big).
\end{eqnarray*}
Now compare this expression with the definition of $L_{N,\beta}^3$: for any radial $f\in C^\infty_c(\R^2)$ and any $h\in C^\infty_c((\R^2)^2)$,
\begin{multline*}
iL_{N,\beta}^3\big[(x_1,x_2,x_3)\mapsto f(x_1)h(x_2,x_3)\big](x_1,x_2,x_3)\\
\,=\,f(x_1)x_2^\bot\Omega_\beta(x_2)\cdot\nabla_2\Big(h(x_2,x_3)+\beta\tfrac{N-3}N\int_{\R^2}W(x_2-x_*)h(x_*,x_3)\,\mu_\beta(x_*)dx_*\Big)\\
+f(x_1)x_3^\bot\Omega_\beta(x_3)\cdot\nabla_3\Big(h(x_2,x_3)+\beta\tfrac{N-3}N\int_{\R^2} h(x_2,x_*)W(x_*-x_3)\,\mu_\beta(x_*)dx_*\Big).
\end{multline*}
We deduce
\begin{equation*}
iL_{N,\beta}^3H^{3;\circ}_\beta
\,=\,iS_{N,\beta}^{3,=}(\tfrac1{\mu_\beta}f^\circ),
\end{equation*}
where $H^{3;\circ}_\beta\in C^\infty_b((\R^2)^3)$ is the smooth function defined in~\eqref{eq:def-Hbeta30-00}.
Using this identity in form of
\[(iL_{N,\beta}^3+\tfrac{i\alpha+1}{N^\sigma})^{-1}iS_{N,\beta}^{3,=}(\tfrac1{\mu_\beta}f^\circ)\,=\,H^{3;\circ}_\beta-\tfrac{i\alpha+1}{N^\sigma}(iL_{N,\beta}^3+\tfrac{i\alpha+1}{N^\sigma})^{-1}H^{3;\circ}_\beta,\]
and inserting it into~\eqref{eq:pre/lem:laplace-NG}, the conclusion follows.
\end{proof}

It remains to pass to the limit in the different terms of the above representation of~$\partial_\tau\langle\bar g_N^1\rangle$. For that purpose, we use the fact that the resolvents of $iL_{N,\beta}^2$ and $iL_{N,\beta}^3$ can be computed in form of explicit Neumann series, cf.~\eqref{eq:Neumann/Lresolv}.

\begin{proof}[Proof of Theorem~\ref{th:main-nonGauss-re}]
We split the proof into four steps, separately evaluating the limit of the different terms in the representation of $\partial_\tau\langle\bar g_N^1\rangle$ given in Lemma~\ref{lem:laplace-NG}. The last step is devoted to the proof of the positivity of the limiting coefficient field $a_\beta$.

\medskip
\step1 Proof that, provided $\beta\ll1$ is small enough, we have for all $\phi\in C^\infty_c(\R^+)$, all radial $h\in C^\infty_c(\R^2)$, and $0<\sigma<\frac12$,
\begin{multline}\label{eq:lapl-lim-1st}
\lim_{N\uparrow\infty}\int_\R g_\phi(\alpha)\Big\langle h,\Big(iS_{N,\beta}^{1,+}(iL_{N,\beta}^2+\tfrac{i\alpha+1}{N^\sigma})^{-1}iS_{N,\beta}^{2,-}(\tfrac1{\mu_\beta}f^\circ)\\
+iS_{N,\beta}^{1,+}(iL_{N,\beta}^2+\tfrac{i\alpha+1}{N^\sigma})^{-1}iS_{N,\beta}^{2,+}H_{N,\beta}^{3;\circ}\Big)\Big\rangle_{\Ld^2_\beta(\R^2)}\,d\alpha\\
\,=\,-\Big(\int_0^\infty\phi\Big)\iint_{(\R^2)^2}\overline{\nabla h(x_1)}\cdot K(x_1-x_2)\,\big((iL^2_\beta+0)^{-1}\pi_\beta^2G_{\beta}^{2;\circ}\big)(x_1,x_2)\\
\times\mu_\beta(x_1)\mu_\beta(x_2)\, dx_1dx_2,
\end{multline}
where we define
\begin{multline}\label{eq:def-Gbetainfty}
G_{\beta}^{2;\circ}(x_1,x_2):=
(\tfrac1{\mu_\beta}\nabla f^\circ)(x_1)\cdot(-\nabla^\bot_1W_\beta)(x_1,x_2)\\
-(\tfrac{1}{\mu_\beta}f^\circ)(x_1)\nabla\log\mu_\beta(x_2)\cdot K(x_1-x_2),
\end{multline}
and where the orthogonal projection $\pi_\beta^2:\Ld^2_\beta((\R^2)^2)\to E_\beta^2$ is given by
\[\pi_\beta^2 h(x_1,x_2)\,:=\,h(x_1,x_2)-\fint\!\!\fint_{(\Sp^1)^2}h(|x_1|e_1,|x_2|e_2)\,d\sigma(e_1)d\sigma(e_2).\]

\medskip\noindent
By definition of $S_{N,\beta}^{1,+},S_{N,\beta}^{2,-}$ in Lemma~\ref{lem:eqn-gNm-21}, using that $V,W,f^\circ$ are radial, we can compute
\begin{multline*}
\Big\langle h,\Big(iS_{N,\beta}^{1,+}(iL_{N,\beta}^2+\tfrac{i\alpha+1}{N^\sigma})^{-1}iS_{N,\beta}^{2,-}(\tfrac1{\mu_\beta}f^\circ) +iS_{N,\beta}^{1,+}(iL_{N,\beta}^2+\tfrac{i\alpha+1}{N^\sigma})^{-1}iS_{N,\beta}^{2,+}H_{N,\beta}^{3;\circ}\Big)\Big\rangle_{\Ld^2_\beta(\R^2)}\\
\,=\,-\tfrac{N-1}N\iint_{(\R^2)^2}\overline{\nabla h(x_1)}\cdot K(x_1-x_2)\,\big((iL_{N,\beta}^2+\tfrac{i\alpha+1}{N^\sigma})^{-1}G_{N,\beta}^{2;\circ}\big)(x_1,x_2)\\
\times\mu_\beta(x_1)\mu_\beta(x_2)\, dx_1dx_2,
\end{multline*}
where
\begin{multline*}
G_{N,\beta}^{2;\circ}(x_1,x_2)\,:=\,\tfrac{1}{\mu_\beta(x_1)\mu_\beta(x_2)}K(x_1-x_2)\cdot(\nabla_1-\nabla_2)(f^\circ(x_1)\mu_\beta(x_2))\\
+\tfrac{N-2}N\sum_{j=1}^2\int_{\R^2}K(x_j-x_*)\cdot(\nabla_j+\nabla\log\mu_\beta(x_j))H_{N,\beta}^{3;\circ}(x_1,x_2,x_*)\,\mu_\beta(x_*)\,dx_*.
\end{multline*}
As $h$ and $W$ are radial, we note that the map
\[(x_1,x_2)\,\mapsto\, \overline{\nabla h(x_1)}\cdot K(x_1-x_2)\,=\,- \overline{\nabla h(x_1)}\cdot \nabla^\bot W(x_1-x_2)\]
belongs to $E_\beta^2$. {As the map $L_{N,\beta}^2$ leaves the subspace $E_\beta^2=\ker(L_{N,\beta}^2)^\bot$ invariant, the orthogonal projection $\pi_\beta^2$ can thus be smuggled in as follows,
\begin{eqnarray*}
\lefteqn{\iint_{(\R^2)^2}\overline{\nabla h(x_1)}\cdot K(x_1-x_2)\,\big((iL_{N,\beta}^2+\tfrac{i\alpha+1}{N^\sigma})^{-1}G_{N,\beta}^{2;\circ}\big)(x_1,x_2)\,\mu_\beta(x_1)\mu_\beta(x_2)\, dx_1dx_2}\\
&=&\iint_{(\R^2)^2}\overline{\nabla h(x_1)}\cdot K(x_1-x_2)\,\big(\pi^2_\beta(iL_{N,\beta}^2+\tfrac{i\alpha+1}{N^\sigma})^{-1}G_{N,\beta}^{2;\circ}\big)(x_1,x_2)\,\mu_\beta(x_1)\mu_\beta(x_2)\, dx_1dx_2\\
&=&\iint_{(\R^2)^2}\overline{\nabla h(x_1)}\cdot K(x_1-x_2)\,\big((iL_{N,\beta}^2+\tfrac{i\alpha+1}{N^\sigma})^{-1}\pi^2_\beta G_{N,\beta}^{2;\circ}\big)(x_1,x_2)\,\mu_\beta(x_1)\mu_\beta(x_2)\, dx_1dx_2,
\end{eqnarray*}
so that the above identity becomes}
\begin{multline}\label{eq:pre-lim-mainterms-noGauss}
\Big\langle h,\Big(iS_{N,\beta}^{1,+}(iL_{N,\beta}^2+\tfrac{i\alpha+1}{N^\sigma})^{-1}iS_{N,\beta}^{2,-}(\tfrac1{\mu_\beta}f^\circ) +iS_{N,\beta}^{1,+}(iL_{N,\beta}^2+\tfrac{i\alpha+1}{N^\sigma})^{-1}iS_{N,\beta}^{2,+}H_{N,\beta}^{3;\circ}\Big)\Big\rangle_{\Ld^2_\beta(\R^2)}\\
\,=\,-\tfrac{N-1}N\iint_{(\R^2)^2}\overline{\nabla h(x_1)}\cdot K(x_1-x_2)\,\big((iL_{N,\beta}^2+\tfrac{i\alpha+1}{N^\sigma})^{-1}\pi_\beta^2 G_{N,\beta}^{2;\circ}\big)(x_1,x_2)\\
\times\mu_\beta(x_1)\mu_\beta(x_2)\, dx_1dx_2.
\end{multline}
Next, we proceed to a suitable reformulation of $G_{N,\beta}^{2;\circ}$.
Recalling the definition of $H^{3;\circ}_{N,\beta}$ in Lemma~\ref{lem:laplace-NG},
as well as $K=-\nabla^\bot W$,
we find that
\begin{multline*}
G_{N,\beta}^{2;\circ}(x_1,x_2)\,=\,\tfrac{1}{\mu_\beta(x_1)\mu_\beta(x_2)}K(x_1-x_2)\cdot(\nabla_1-\nabla_2)(f^\circ(x_1)\mu_\beta(x_2))\\
+\beta\tfrac{N-2}N(\tfrac1{\mu_\beta}\nabla f^\circ)(x_1)\cdot\nabla^\bot_1(W\ast_{\mu_\beta}W_{N,\beta})(x_1,x_2)\\
+\beta\tfrac{N-2}N(\tfrac1{\mu_\beta}f^\circ)(x_1)\int_{\R^2}\nabla^\bot W(x_2-x_*)\cdot(\nabla_2W_{N,\beta})(x_*,x_2)\,\mu_\beta(x_*)\,dx_*\\
+\beta\tfrac{N-2}N(\tfrac1{\mu_\beta}f^\circ)(x_1)\nabla\log\mu_\beta(x_2)\cdot \int_{\R^2}\nabla^\bot W(x_2-x_*)W_{N,\beta}(x_*,x_2)\,\mu_\beta(x_*)\,dx_*.
\end{multline*}
From the definition of~$W_{N,\beta}$ in Lemma~\ref{lem:laplace-NG}, we note that
\[W(x_1-x_2)-W_{N,\beta}(x_1,x_2)\,=\,\beta\tfrac{N-3}N(W\ast_{\mu_\beta}W_{N,\beta})(x_1,x_2).\]
This allows to reorganize the first two right-hand side terms above,
\begin{multline}\label{eq:GNbeta-rewr}
G_{N,\beta}^{2;\circ}(x_1,x_2)\,=\,
(\tfrac1{\mu_\beta}\nabla f^\circ)(x_1)\cdot(-\nabla^\bot_1W_{N,\beta})(x_1,x_2)\\
-(\tfrac{1}{\mu_\beta}f^\circ)(x_1)\nabla\log\mu_\beta(x_2)\cdot K(x_1-x_2)\\
-\tfrac{\beta}{N}(\tfrac{1}{\mu_\beta}\nabla f^\circ)(x_1)\cdot\int_{\R^2}K(x_1-x_*)W_{N,\beta}(x_*,x_2)\,\mu_\beta(x_*)\,dx_*\\
+\beta\tfrac{N-2}N(\tfrac{1}{\mu_\beta}f^\circ)(x_1)\int_{\R^2}\nabla^\bot W(x_2-x_*)\cdot(\nabla_2W_{N,\beta})(x_*,x_2)\,\mu_\beta(x_*)\,dx_*\\
+\beta\tfrac{N-2}N(\tfrac{1}{\mu_\beta}f^\circ)(x_1)\nabla\log\mu_\beta(x_2)\cdot \int_{\R^2}\nabla^\bot W(x_2-x_*)W_{N,\beta}(x_*,x_2)\,\mu_\beta(x_*)\,dx_*.
\end{multline}
In order to evaluate the last two terms in the right-hand side, we note that the very definition of~$W_{N,\beta}$ further lets us compute
\begin{multline}\label{eq:decompnabbot-Wbet}
\int_{\R^2}\nabla^\bot W(x_2-x_*)W_{N,\beta}(x_*,x_2)\,\mu_\beta(x_*)\,dx_*\\
\,=\,\tfrac12\nabla^\bot_{x_2}\Big(\sum_{n=0}^\infty(-\beta\tfrac{N-3}N)^n W^{\ast_{\mu_\beta}(n+2)}(x_2,x_2)\Big),
\end{multline}
together with
\begin{multline*}
\int_{\R^2}\nabla^\bot W(x_2-x_*)\cdot(\nabla_2W_{N,\beta})(x_*,x_2)\,\mu_\beta(x_*)\,dx_*\\
\,=\,-\sum_{\alpha=1}^d\sum_{n=0}^\infty(-\beta\tfrac{N-3}N)^n \big(\nabla^\bot_\alpha W\ast_{\mu_\beta}W^{\ast_{\mu_\beta}n}\ast_{\mu_\beta}\nabla_\alpha W\big)(x_2,x_2)\,=\,0.
\end{multline*}
As by symmetry the function $x\mapsto W^{\ast_{\mu_\beta}n}(x,x)$ is radial for all $n\ge1$, these two identities ensure that the last two right-hand side terms in~\eqref{eq:GNbeta-rewr} actually vanish identically. We are thus left with
\begin{multline*}
G_{N,\beta}^{2;\circ}(x_1,x_2)\,=\,
(\tfrac1{\mu_\beta}\nabla f^\circ)(x_1)\cdot(-\nabla^\bot_1W_{N,\beta})(x_1,x_2)\\
-(\tfrac{1}{\mu_\beta}f^\circ)(x_1)\nabla\log\mu_\beta(x_2)\cdot K(x_1-x_2)\\
-\tfrac{\beta}{N}(\tfrac{1}{\mu_\beta}\nabla f^\circ)(x_1)\cdot\int_{\R^2}K(x_1-x_*)W_{N,\beta}(x_*,x_2)\,\mu_\beta(x_*)\,dx_*.
\end{multline*}
Comparing this with the definition of $G_\beta^{2;\circ}$ in~\eqref{eq:def-Gbetainfty}, and in particular comparing $W_{N,\beta}$ to~$W_\beta$, we easily find that for $\beta\le\frac12\|W\|_{\Ld^\infty(\R^2)}^{-1}$,
\[\|G_{N,\beta}^{2;\circ}-G_\beta^{2;\circ}\|_{\Ld^2_\beta((\R^2)^2)}\,\lesssim\,\tfrac1N\|\tfrac1{\mu_\beta}\nabla f^\circ\|_{\Ld^2_\beta(\R^2)}.\]
Further noting that
\[(iL_{N,\beta}^2h-iL_\beta^2h)(x_1,x_2)\,=\,-\tfrac{2\beta}Nx_2\Omega_\beta(x_2)\cdot\int_{\R^2}K(x_2-x_*)h(x_1,x_*)\,\mu_\beta(x_*)\,dx_*,\]
we can estimate for all $\Re\omega\ne0$,
\[\big\|(iL_{N,\beta}^2+\omega)^{-1}-(iL_{\beta}^2+\omega)^{-1}\big\|_{\Ld^2_\beta((\R^2)^2)\to\Ld^2_\beta((\R^2)^2)}\,\lesssim\,\tfrac{1}{N}|\Re\omega|^{-2}\|\nabla\log\mu_\beta\|_{\Ld^2_\beta(\R^2)}.\]
These error estimates allow to replace the representation formula~\eqref{eq:pre-lim-mainterms-noGauss} by the following approximation, for $\beta\le\frac12\|W\|_{\Ld^\infty(\R^2)}^{-1}$ and $0<\sigma<\frac12$,
\begin{multline*}
\lim_{N\uparrow\infty}\sup_{\alpha\in\R}\bigg|\Big\langle h,\Big(iS_{N,\beta}^{1,+}(iL_{N,\beta}^2+\tfrac{i\alpha+1}{N^\sigma})^{-1}iS_{N,\beta}^{2,-}(\tfrac1{\mu_\beta}f^\circ)\\
+iS_{N,\beta}^{1,+}(iL_{N,\beta}^2+\tfrac{i\alpha+1}{N^\sigma})^{-1}iS_{N,\beta}^{2,+}H_{N,\beta}^{3;\circ}\Big)\Big\rangle_{\Ld^2_\beta(\R^2)}\\
+\iint_{(\R^2)^2}\overline{\nabla h(x_1)}\cdot K(x_1-x_2)\,\big((iL_{\beta}^2+\tfrac{i\alpha+1}{N^\sigma})^{-1}\pi_\beta^2 G_{\beta}^{2;\circ}\big)(x_1,x_2)\\
\times\mu_\beta(x_1)\mu_\beta(x_2)\, dx_1dx_2\bigg|
\,=\,0.
\end{multline*}
{Finally, the limiting absorption principle of Remark~\ref{rem:nonGaus} allows to pass to the limit in the second left-hand side term. As $\int_\R g_\phi=\int_0^\infty \phi$ and $|g_\phi(\alpha)|\lesssim_\phi\langle\alpha\rangle^{-2}$, the claim~\eqref{eq:lapl-lim-1st} follows.}

\medskip
\step2 Proof that for all $\phi\in C^\infty_c(\R^+)$, all radial $h\in C^\infty_c(\R^2)$, and $\sigma>0$,
\begin{multline}\label{eq:cancel-term}
\lim_{N\uparrow\infty}\int_\R g_\phi(\alpha)\\
\times\tfrac{i\alpha+1}{N^\sigma}\Big\langle h,iS_{N,\beta}^{1,+}(iL_{N,\beta}^2+\tfrac{i\alpha+1}{N^\sigma})^{-1}iS_{N,\beta}^{2,+}(iL_{N,\beta}^3+\tfrac{i\alpha+1}{N^\sigma})^{-1}H_\beta^{3;\circ}\Big\rangle_{\Ld^2_\beta(\R^2)}\,d\alpha
=0.
\end{multline}
Let the radial test function $h\in C^\infty_c(\R^2)$ be fixed.
By the definition of $S_{N,\beta}^{1,+},S_{N,\beta}^{2,+}$ in Lemma~\ref{lem:eqn-gNm-21}, we can compute for $\Re\omega\ne0$,
\begin{multline}\label{eq:decomp-omexpr-0}
\omega \Big\langle h,iS_{N,\beta}^{1,+}(iL_{N,\beta}^2+\omega)^{-1}iS_{N,\beta}^{2,+}(iL_{N,\beta}^3+\omega)^{-1}H_\beta^{3;\circ}\Big\rangle_{\Ld^2_\beta(\R^2)}\\
\,=\,-\tfrac{N-1}N\tfrac{N-2}N\sum_{j=1}^2\omega\iiint_{(\R^2)^3} K(x_j-x_3)\cdot\big[\nabla_{j;\beta}(iL_{N,\beta}^3+\omega)^{-1}H_\beta^{3;\circ}\big](x_1,x_2,x_3)\\
\times \overline{\big[((iL_{N,\beta}^2)^*+\overline\omega)^{-1}H^{2;h}\big](x_1,x_2)}\,\mu_\beta(x_1)\mu_\beta(x_2)\mu_\beta(x_3)\,dx_1dx_2dx_3,
\end{multline}
where we define $H^{2;h}(x_1,x_2):= K(x_1-x_2)\cdot \nabla_1h(x_1)$.
Alternatively, by further integrating by parts, this reads
\begin{multline*}
\omega \Big\langle h,iS_{N,\beta}^{1,+}(iL_{N,\beta}^2+\omega)^{-1}iS_{N,\beta}^{2,+}(iL_{N,\beta}^3+\omega)^{-1}H_\beta^{3;\circ}\Big\rangle_{\Ld^2_\beta(\R^2)}\\
\,=\,\tfrac{N-1}N\tfrac{N-2}N\sum_{j=1}^2\omega\iiint_{(\R^2)^3} \big[(iL_{N,\beta}^3+\omega)^{-1}H_\beta^{3;\circ}\big](x_1,x_2,x_3)\\
\times \overline{K(x_j-x_3)\cdot\nabla_{j}\big[((iL_{N,\beta}^2)^*+\overline\omega)^{-1}H^{2;h}\big](x_1,x_2)}\,\mu_\beta(x_1)\mu_\beta(x_2)\mu_\beta(x_3)\,dx_1dx_2dx_3.
\end{multline*}
As $h$ is radial, we note that $H^{2;h}$ belongs to $E^2_\beta$.
Recalling that $E^2_\beta=\ker(L_{N,\beta}^2)^\bot$, cf.~Remark~\ref{rem:nonGaus}, we deduce that
\[((iL_{N,\beta}^2)^*+\overline\omega)^{-1}H^{2;h}~\in~E^2_\beta.\]
Now, for any $g\in E^2_\beta$, we note that for $j=1,2$ the function
\[(x_1,x_2,x_3)\mapsto K(x_j-x_3)\cdot\nabla_jg(x_1,x_2)\]
also belongs to $E^3_\beta$. As $E^3_\beta=\ker(L_{N,\beta}^3)^\bot$, cf.~Remark~\ref{rem:nonGaus}, this allows to replace~$H^{3;\circ}_\beta$ by~$\pi_\beta^3H^{3;\circ}_\beta$ in~\eqref{eq:decomp-omexpr-0},
\begin{multline*}
\omega \Big\langle h,iS_{N,\beta}^{1,+}(iL_{N,\beta}^2+\omega)^{-1}iS_{N,\beta}^{2,+}(iL_{N,\beta}^3+\omega)^{-1}H_\beta^{3;\circ}\Big\rangle_{\Ld^2_\beta(\R^2)}\\
\,=\,-\tfrac{N-1}N\tfrac{N-2}N\sum_{j=1}^2\omega\iiint_{(\R^2)^3} K(x_j-x_3)\cdot\big[\nabla_{j;\beta}(iL_{N,\beta}^3+\omega)^{-1}\pi_\beta^3H_\beta^{3;\circ}\big](x_1,x_2,x_3)\\
\times \overline{\big[((iL_{N,\beta}^2)^*+\overline\omega)^{-1}H^{2;h}\big](x_1,x_2)}\,\mu_\beta(x_1)\mu_\beta(x_2)\mu_\beta(x_3)\,dx_1dx_2dx_3.
\end{multline*}
Now decomposing
\[L_{N,\beta}^2\,=\,A_{\beta}^{2}+\beta \tfrac{N-2}NT_\beta^{2},\qquad L_{N,\beta}^3\,=\,A_{\beta}^{3}+\beta \tfrac{N-3}NT_\beta^{3},\]
in terms of
\begin{equation*}\begin{array}{rlrl}
A_{\beta}^{2}\,&:=\,L_\beta\otimes\Id+\Id\otimes L_\beta,\quad &T_\beta^2\,&:=\,\Id\otimes T_\beta,\\
A_{\beta}^{3}\,&:=\,L_\beta\otimes\Id^{\otimes2}+\Id\otimes L_\beta\otimes\Id+\Id^{\otimes2}\otimes L_\beta,\quad &T_\beta^3\,&:=\,\Id\otimes T_\beta\otimes\Id+\Id^{\otimes2}\otimes T_\beta,
\end{array}
\end{equation*}
and further arguing as in the proof of Lemma~\ref{lem:nonGaus}(iii) to express the resolvents of~$L^2_{N,\beta}$ and~$L^3_{N,\beta}$ as power series, we are led to
\begin{multline*}
\omega \Big\langle h,iS_{N,\beta}^{1,+}(iL_{N,\beta}^2+\omega)^{-1}iS_{N,\beta}^{2,+}(iL_{N,\beta}^3+\omega)^{-1}H_\beta^{3;\circ}\Big\rangle_{\Ld^2_\beta(\R^2)}\\
\,=\,-\tfrac{N-1}N\tfrac{N-2}N\sum_{j=1}^2\omega\iiint_{(\R^2)^3} {K(x_j-x_3)\cdot\big[\nabla_{j;\beta}}(iA_{\beta}^3+\omega)^{-1}G^{3;\circ}_{\beta,\omega}\big](x_1,x_2,x_3)\\
\times \overline{\big[((iA_{\beta}^2)^*+\overline\omega)^{-1}G^{2;h}_{\beta,\omega}\big](x_1,x_2)}\,\mu_\beta(x_1)\mu_\beta(x_2)\mu_\beta(x_3)\,dx_1dx_2dx_3,
\end{multline*}
where we define the modified test functions
\begin{eqnarray*}
G^{2;h}_{N,\beta,\omega}&:=&\sum_{n=0}^\infty(-\beta\tfrac{N-2}N)^n{[(iT_\beta^2)^*((iA_\beta^2)^*+\overline\omega)^{-1}]^nH^{2;h}},\\
G^{3;\circ}_{N,\beta,\omega}&:=&\sum_{n=0}^\infty(-\beta\tfrac{N-3}N)^n[iT_\beta^3(iA_\beta^3+\omega)^{-1}]^n\pi_\beta^3H^{3;\circ}_\beta.
\end{eqnarray*}
Using polar coordinates $x=(r,\theta)$, noting that resolvents of $A^2_\beta,A^3_\beta$ are explicit using Fourier series, and writing $K(x)=-W'(|x|)\tfrac{x^\bot}{|x|}$, the above can be reformulated into
\begin{multline*}
\omega \Big\langle h,iS_{N,\beta}^{1,+}(iL_{N,\beta}^2+\omega)^{-1}iS_{N,\beta}^{2,+}(iL_{N,\beta}^3+\omega)^{-1}H_\beta^{3;\circ}\Big\rangle_{\Ld^2_\beta(\R^2)}\\
\,=\,\tfrac{N-1}N\tfrac{N-2}N\sum_{j=1}^2\sum_{k_1,k_2,k_3,k_1',k_2'}\omega\iiint_{(\R^+)^3}\Big(\iint_{(0,2\pi)^2}e^{i(k_1-k_1')\theta_1+i(k_2-k_2')\theta_2+ik_3\theta_j}\,d\theta_1d\theta_2\Big)\\
\times\Big(\big[-r_3S^{k_3}(r_j,r_3)\partial_{r_j}+( ik_j+\beta r_j^2\Omega_\beta(r_j))R^{k_3}(r_j,r_3)\big]\tfrac{\widehat G^{3;\circ}_{N,\beta,\omega}(r_1,k_1;r_2,k_2;r_3,k_3)}{ik_1\Omega_\beta(r_1)+ik_2\Omega_\beta(r_2)+ik_3\Omega_\beta(r_3)+\omega}\Big)\\
\times \tfrac{\overline{\widehat G^{2;h}_{N,\beta,\omega}(r_1,k_1';r_2,k_2')}}{ik_1'\Omega_\beta(r_1)+ik_2'\Omega_\beta(r_2)+\omega}\,\mu_\beta(r_1)\mu_\beta(r_2)\mu_\beta(r_3)\,r_1r_2r_3\,dr_1dr_2dr_3,
\end{multline*}
where we denote
\begin{eqnarray*}
R^k(r_1,r_3)&:=&\int_0^{2\pi}\tfrac{W'((r_1^2+r_3^2-2r_1r_3\cos(\theta))^{1/2})}{(r_1^2+r_3^2-2r_1r_3\cos(\theta))^{1/2}}\big(1-\tfrac{r_3}{r_1}\cos(\theta)\big)\,e^{-ik\theta}\,d\theta,\\
S^k(r_1,r_3)&:=&\int_0^{2\pi}\tfrac{W'((r_1^2+r_3^2-2r_1r_3\cos(\theta))^{1/2})}{(r_1^2+r_3^2-2r_1r_3\cos(\theta))^{1/2}}\sin(\theta)\,e^{-ik\theta}\,d\theta.
\end{eqnarray*}
Performing integrals over angles, this leads  to
\begin{multline*}
\Big|\omega \Big\langle h,iS_{N,\beta}^{1,+}(iL_{N,\beta}^2+\omega)^{-1}iS_{N,\beta}^{2,+}(iL_{N,\beta}^3+\omega)^{-1}H_\beta^{3;\circ}\Big\rangle_{\Ld^2_\beta(\R^2)}\Big|\\
\,\lesssim\,\sum_{k_1,k_2,k_3}\bigg|\omega\iiint_{(\R^+)^3}\tfrac{\overline{\widehat G^{2;h}_{N,\beta,\omega}(r_1,k_1+k_3;r_2,k_2)}}{i(k_1+k_3)\Omega_\beta(r_1)+ik_2\Omega_\beta(r_2)+\omega}\\
\times\Big(\big[\!-r_3S^{k_3}(r_1,r_3)\partial_{r_1}+( ik_1+\beta r_1^2\Omega_\beta(r_1))R^{k_3}(r_1,r_3)\big]\tfrac{\widehat G^{3;\circ}_{N,\beta,\omega}(r_1,k_1;r_2,k_2;r_3,k_3)}{ik_1\Omega_\beta(r_1)+ik_2\Omega_\beta(r_2)+ik_3\Omega_\beta(r_3)+\omega}\Big)\\
\times\mu_\beta(r_1)\mu_\beta(r_2)\mu_\beta(r_3)\,r_1r_2r_3\,dr_1dr_2dr_3\bigg|\qquad+\text{sym},
\end{multline*}
where the symbol `$\text{sym}$' stands for the expression preceding it up to exchanging the variables $(r_1,k_1)$ and $(r_2,k_2)$ in the integrand.

We turn to the evaluation of the derivative $\partial_{r_1}$ in the integrand.
For that purpose, we separately consider the cases $k_1=0$ and $k_3=0$. Noting that $S^{k_3=0}(r_1,r_3)=0$, using the following identities,
\begin{eqnarray*}
\big(\tfrac{1}{ik_1\Omega_\beta(r_1)+ik_2\Omega_\beta(r_2)+\omega}\big)^2&=&-\tfrac1{ik_1\Omega_\beta'(r_1)}\partial_{r_1}\big(\tfrac{1}{ik_1\Omega_\beta(r_1)+ik_2\Omega_\beta(r_2)+\omega}\big),\\
\partial_{r_1}\big(\tfrac{1}{ik_1\Omega_\beta(r_1)+ik_2\Omega_\beta(r_2)+ik_3\Omega_\beta(r_3)+\omega}\big)&=&\tfrac{k_1\Omega_\beta'(r_1)}{k_3\Omega_\beta'(r_3)}\partial_{r_3}\big(\tfrac{1}{ik_1\Omega_\beta(r_1)+ik_2\Omega_\beta(r_2)+ik_3\Omega_\beta(r_3)+\omega}\big),
\end{eqnarray*}
and performing several integrations by parts,
we are led to
\begin{align}\label{eq:bigreformulation}
&\Big|\omega \Big\langle h,iS_{N,\beta}^{1,+}(iL_{N,\beta}^2+\omega)^{-1}iS_{N,\beta}^{2,+}(iL_{N,\beta}^3+\omega)^{-1}H_\beta^{3;\circ}\Big\rangle_{\Ld^2_\beta(\R^2)}\Big|\\
&\,\lesssim\,
\sum_{k_1,k_2}\bigg|\omega\iint_{(\R^+)^2}
\big(\tfrac{1}{ik_1\Omega_\beta(r_1)+ik_2\Omega_\beta(r_2)+\omega}\big)
A_{N,\beta,\omega}^{k_1,k_2}(r_1,r_2)
\mu_\beta(r_1)\mu_\beta(r_2)\,r_1r_2\,dr_1dr_2\bigg|\nonumber\\
&+\sum_{k_1,k_2,k_3}\mathds1_{k_3\ne0}\bigg|\omega\iiint_{(\R^+)^3}
\big(\tfrac{1}{ik_1\Omega_\beta(r_1)+ik_2\Omega_\beta(r_2)+ik_3\Omega_\beta(r_3)+\omega}\big)\big(\tfrac{1}{i(k_1+k_3)\Omega_\beta(r_1)+ik_2\Omega_\beta(r_2)+\omega}\big)\nonumber\\
&\hspace{4.5cm}\times B_{N,\beta,\omega}^{k_1,k_2,k_3}(r_1,r_2,r_3)
\mu_\beta(r_1)\mu_\beta(r_2)\mu_\beta(r_3)\,r_1r_2r_3\,dr_1dr_2dr_3\bigg|\nonumber\\
&+\sum_{k_1,k_2,k_3}\mathds1_{k_3\ne0}\big|\tfrac{k_1}{k_3}\big|\bigg|\omega\iiint_{(\R^+)^3}
\big(\tfrac{1}{ik_1\Omega_\beta(r_1)+ik_2\Omega_\beta(r_2)+ik_3\Omega_\beta(r_3)+\omega}\big)\big(\tfrac1{i(k_1+k_3)\Omega_\beta(r_1)+ik_2\Omega_\beta(r_2)+\omega}\big)\nonumber\\
&\hspace{3cm}\times C_{N,\beta,\omega}^{k_1,k_2,k_3}(r_1,r_2,r_3)\mu_\beta(r_1)\mu_\beta(r_2)\mu_\beta(r_3)\,r_1r_2r_3\,dr_1dr_2dr_3\bigg|\qquad+\text{sym},\nonumber
\end{align}
where we use the notations
\begin{eqnarray*}
A_{N,\beta,\omega}^{k_1,k_2}(r_1,r_2)&\!\!:=\!\!&\big(\partial_{r_1}+\beta r_1\Omega_\beta(r_1)+\tfrac1{r_1}\big)\\
&&\hspace{0.2cm}\times\bigg(\tfrac1{\Omega_\beta'(r_1)}\big( 1+\tfrac\beta{ik_1} r_1^2\Omega_\beta(r_1)\big)\overline{\widehat G^{2;h}_{\beta,\omega}(r_1,k_1;r_2,k_2)}\\
&&\hspace{1.6cm}\times\int_{\R^+}R^0(r_1,r_3)\,\widehat G^{3;\circ}_{\beta,\omega}(r_1,k_1;r_2,k_2;r_3,0)\,\mu_\beta(r_3)\,r_3\,dr_3\bigg),\\
B_{N,\beta,\omega}^{k_1,k_2,k_3}(r_1,r_2,r_3)&\!\!:=\!\!&\overline{\widehat G^{2;h}_{\beta,\omega}(r_1,k_1+k_3;r_2,k_2)}\\
&&\hspace{0.3cm}\times\big(r_3S^{k_3}(r_1,r_3)\partial_{r_1}-( ik_1+\beta r_1^2\Omega_\beta(r_1))R^{k_3}(r_1,r_3)\big)\\
&&\hspace{5cm}\times\widehat G^{3;\circ}_{\beta,\omega}(r_1,k_1;r_2,k_2;r_3,k_3),\\
C_{N,\beta,\omega}^{k_1,k_2,k_3}(r_1,r_2,r_3)&\!\!:=\!\!&\Omega_\beta'(r_1)\overline{\widehat G^{2;h}_{\beta,\omega}(r_1,k_1+k_3;r_2,k_2)}\\
&&\hspace{0.3cm}\times(\partial_{r_3}+\beta r_3\Omega_\beta(r_3)+\tfrac{1}{r_3})
\Big(\tfrac{r_3S^{k_3}(r_1,r_3)}{\Omega_\beta'(r_3)}\widehat G^{3;\circ}_{\beta,\omega}(r_1,k_1;r_2,k_2;r_3,k_3)\Big).
\end{eqnarray*}
{With this reformulation~\eqref{eq:bigreformulation}, we are now in position to perform direct estimates similarly as in Step~2 of the proof of Lemma~\ref{lem:nonGaus}.}
More precisely, instead of~\eqref{eq:model-crit-est}, we can use here the following rougher model estimate: for any $\phi\in C^\infty_b(-1,1)$ and $\e\ne0$,
\[\bigg|\int_{-1}^1\tfrac{\phi(t)}{t+i\e}\,dt\bigg|+\bigg|\int_0^1\tfrac{t\phi(t)}{t^2+i\e}\,dt\bigg|\,\lesssim\,\log(2+\tfrac1\e)\|\phi\|_{\Ld^\infty(-1,1)}.\]
Using this bound to estimate~\eqref{eq:bigreformulation} close to singularities, using again local deformations to reduce to these model situations, recalling the non-degeneracy assumption~\eqref{eq:nondegenerate} for~$\Omega_\beta$, and recalling $|g_\phi(\alpha)|\lesssim_\phi\langle\alpha\rangle^{-2}$, the claim~\eqref{eq:cancel-term} easily follows.
The key is the transversality of singularities in~\eqref{eq:bigreformulation} (cf.~$k_3\ne0$ in the last two terms). {For brevity, we skip the details, which are analogous to the computations in Step~2 of the proof of Lemma~\ref{lem:nonGaus}.}

\medskip
\step3 Proof of~\eqref{eq:FP-noGauss}.\\
Combined with Lemma~\ref{lem:laplace-NG}, for $0<\sigma<\frac1{20}$, the results of the first two steps yield in the distributional sense, for all radial $h\in C^\infty_c(\R^2)$, as $N\uparrow\infty$,
\begin{equation*}
\langle h,\partial_\tau\bar g_N^1\rangle_{\Ld^2_\beta(\R^2)}\,\to\,T_\beta(h),
\end{equation*}
where we denote
\begin{multline}\label{eq:preconcl-noGauss}
T_\beta(h)\,:=\,-\iint_{(\R^2)^2}\overline{\nabla h(x_1)}\cdot K(x_1-x_2)\,\big((iL^2_\beta+0)^{-1}\pi_\beta^2G_{\beta}^{2;\circ}\big)(x_1,x_2)\\
\times\mu_\beta(x_1)\mu_\beta(x_2)\, dx_1dx_2.
\end{multline}
It remains to proceed to a slight reformulation of this limiting expression.
For that purpose, we note that by definition of $iL_\beta^{2}$,
\begin{multline*}
iL_\beta^2\big[(x_1,x_2)\mapsto \beta(\tfrac1{\mu_\beta}f^\circ)(x_1)W_\beta(x_1,x_2)\big]
\,=\,
-(\tfrac1{\mu_\beta}f^\circ)(x_1)\nabla\log\mu_\beta(x_1)\cdot\nabla_1^\bot W_\beta(x_1,x_2)\\
-(\tfrac1{\mu_\beta}f^\circ)(x_1)\nabla\log\mu_\beta(x_2)\cdot\nabla_2^\bot\Big( W_\beta+\beta W_\beta\ast_{\mu_\beta}W\Big)(x_1,x_2).
\end{multline*}
Using the definition of $W_\beta$ in form of
\[W_\beta+\beta W_\beta\ast_{\mu_\beta}W\,=\,W,\]
cf.~\eqref{eq:def-Wbeta}, we deduce
\begin{multline*}
iL_\beta^2\big[(x_1,x_2)\mapsto \beta(\tfrac1{\mu_\beta}f^\circ)(x_1)W_\beta(x_1,x_2)\big]
\,=\,
(\tfrac1{\mu_\beta}f^\circ)(x_1)\nabla\log\mu_\beta(x_1)\cdot(-\nabla_1^\bot W_\beta)(x_1,x_2)\\
-(\tfrac1{\mu_\beta}f^\circ)(x_1)\nabla\log\mu_\beta(x_2)\cdot K(x_1-x_2).
\end{multline*}
Recalling the definition of $G_\beta^{2;\circ}$ in Step~1, this allows us to rewrite, after straightforward simplifications, for $\e>0$,
\begin{equation*}
(iL_\beta^2+\e)^{-1}\pi_\beta^2G_{\beta}^{2;\circ}
\,=\,(iL_\beta^2+\e)^{-1}\pi_\beta^2R_\beta^{2;\circ}+\pi_\beta^2S_\beta^{2;\circ}
-\e(iL_\beta^2+\e)^{-1}\pi_\beta^2S_\beta^{2;\circ},
\end{equation*}
where we take
\begin{eqnarray*}
R_\beta^{2;\circ}(x_1,x_2)&:=&\nabla(\tfrac1{\mu_\beta} f^\circ)(x_1)\cdot(-\nabla^\bot_1W_\beta)(x_1,x_2),\\
S_\beta^{2;\circ}(x_1,x_2)&:=&\beta(\tfrac1{\mu_\beta}f^\circ)(x_1)W_\beta(x_1,x_2).
\end{eqnarray*}
Letting $\e\downarrow0$, appealing to the limiting absorption principle for the restriction of $L_\beta^2$ to~$\operatorname{ran}(\pi_\beta^2)=E_\beta^2=\ker(L_\beta^2)^\bot$, cf.~Remark~\ref{rem:nonGaus}, we deduce
\begin{equation*}
(iL_\beta^2+0)^{-1}\pi_\beta^2G_{\beta}^{2;\circ}
\,=\,(iL_\beta^2+0)^{-1}\pi_\beta^2R_\beta^{2;\circ}+\pi_\beta^2 S_\beta^{2;\circ}.
\end{equation*}
The limit~\eqref{eq:preconcl-noGauss} can then be reformulated as follows,
\begin{multline*}
T_\beta(h)\,=\,-\iint_{(\R^2)^2}\overline{\nabla h(x_1)}\cdot K(x_1-x_2)\,(\pi_\beta^2S_{\beta}^{2;\circ})(x_1,x_2)
\mu_\beta(x_1)\mu_\beta(x_2)\, dx_1dx_2\\
-\iint_{(\R^2)^2}\overline{\nabla h(x_1)}\cdot K(x_1-x_2)\,\big((iL^2_\beta+0)^{-1}\pi_\beta^2R_{\beta}^{2;\circ}\big)(x_1,x_2)\\
\times\mu_\beta(x_1)\mu_\beta(x_2)\, dx_1dx_2.
\end{multline*}
Recalling that the test function $(x_1,x_2)\mapsto\overline{\nabla h(x_1)}\cdot K(x_1,x_2)$ belongs to~$E_\beta^2$, and inserting the definition of $S_\beta^{2;\circ}$, the first right-hand side term takes on the following form,
\begin{multline*}
\iint_{(\R^2)^2}\overline{\nabla h(x_1)}\cdot K(x_1-x_2)\,(\pi_\beta^2S_{\beta}^{2;\circ})(x_1,x_2)
\mu_\beta(x_1)\mu_\beta(x_2)\, dx_1dx_2\\
\,=\,-\beta\int_{\R^2}f^\circ(x_1)\overline{\nabla h(x_1)}\cdot\Big(\int_{\R^2} \nabla^\bot W(x_1-x_*)\, W_\beta(x_*,x_1)
\mu_\beta(x_*)\,dx_*\Big) dx_1,
\end{multline*}
and thus, using the definition of $W_\beta$ as in~\eqref{eq:decompnabbot-Wbet}, noting that the function $x\mapsto W^{\ast_{\mu_\beta}n}(x,x)$ is radial for all $n\ge1$,
\begin{equation*}
\iint_{(\R^2)^2}\overline{\nabla h(x_1)}\cdot K(x_1-x_2)\,(\pi_\beta^2S_{\beta}^{2;\circ})(x_1,x_2)
\mu_\beta(x_1)\mu_\beta(x_2)\, dx_1dx_2=0.
\end{equation*}
Further noting that $R_\beta^{2;\circ}$ actually belongs to $E_\beta^2$, hence $\pi_\beta^2R_\beta^{2;\circ}=R_\beta^{2;\circ}$, we get that
\begin{equation*}
T_\beta(h)\,=\,-\iint_{(\R^2)^2}\overline{\nabla h(x_1)}\cdot K(x_1-x_2)\,\big((iL^2_\beta+0)^{-1}R_{\beta}^{2;\circ}\big)(x_1,x_2)\,\mu_\beta(x_1)\mu_\beta(x_2)\, dx_1dx_2.
\end{equation*}
Using polar coordinates $x=re$, inserting the definition of $R_\beta^{2;\circ}$, noting that the operator~$L^2_\beta$ commutes with multiplication by radial functions of the first variable $x_1$, and taking the real part, this proves the conclusion~\eqref{eq:FP-noGauss}.

\medskip
\step4 Positivity of $a_\beta$.\\
In the spirit of~\eqref{eq:deformed-Hilbert}, consider the following deformed Hilbert structure on $\Ld^2_\beta((\R^2)^2)$: for all $G,H\in \Ld^2_\beta((\R^2)^2)$, define
\begin{multline*}
\langle G,H\rangle_{\widetilde \Ld^2_\beta((\R^2)^2)}\,:=\,\iint_{(\R^2)^2}\overline{G(x_1,x_2)}H(x_1,x_2)\,\mu_\beta(x_1)\mu_\beta(x_2)\,dx_1dx_2\\
+\beta\iint_{(\R^2)^3}\overline{G(x_1,x_2)}H(x_1,x_3)\,W(x_2-x_3)\,\mu_\beta(x_1)\mu_\beta(x_2)
\mu_\beta(x_3)\,dx_1dx_2dx_3,
\end{multline*}
that is, using $\mu_\beta$-convolution notation,
\begin{equation*}
\langle G,H\rangle_{\widetilde \Ld^2_\beta((\R^2)^2)}\,=\,\iint_{(\R^2)^2}\overline{\big(G+\beta(G\ast_{\mu_\beta}W)\big)(x_1,x_2)}\,H(x_1,x_2)\,\mu_\beta(x_1)\mu_\beta(x_2)\,dx_1dx_2.
\end{equation*}
Noting that the definition of $H_0,H_\beta$ in~\eqref{eq:def-HbH0} yields
\[H_\beta+\beta(H_\beta \ast_{\mu_\beta}W)\,=\,H_0,\]
and recalling that $L_\beta^2$ commutes with multiplication by radial functions of the first variable~$x_1$,
we deduce for all nonnegative radial $h\in C^\infty_c(\R^2)$,
\begin{multline*}
\big\langle \sqrt{h}H_\beta,(iL_\beta^2+0)^{-1}\sqrt{h}H_\beta\big\rangle_{\widetilde \Ld^2_\beta((\R^2)^2)}\\
\,=\,\iint_{(\R^2)^2}h(x_1)\,H_0(x_1,x_2)\big[(iL_\beta^2+0)^{-1}H_\beta\big](x_1,x_2)\,\mu_\beta(x_1)\mu_\beta(x_2)\,dx_1dx_2,
\end{multline*}
which means, by definition~\eqref{eq:def-abeta} of $a_\beta$,
\begin{equation}\label{eq:reform-hany}
\int_{\R^+}h(r)\,a_\beta(r)\,\mu_\beta(r)\,r\,dr\,=\,\big\langle \sqrt{h}H_\beta,\Re(iL_\beta^2+0)^{-1}\sqrt{h}H_\beta\big\rangle_{\widetilde \Ld{}^2_\beta((\R^2)^2)}.
\end{equation}
As $L^2_\beta$ is self-adjoint on the deformed Hilbert structure $\widetilde \Ld{}^2_\beta((\R^2)^2)$, as shown in the proof of Lemma~\ref{lem:nonGaus}(i),
we deduce that the right-hand side of~\eqref{eq:reform-hany} is nonnegative. The nonnegativity of $a_\beta$ follows by the arbitrariness of $h$.
\end{proof}

\section{Degenerate case: Gaussian equilibrium}\label{sec:Gauss}

This section is devoted to the case of a Gaussian mean-field equilibrium $\mu_\beta$, that is, we assume that potentials $V,W$ satisfy for some $R\in(0,\infty)$,
\begin{equation}\label{eq:Gaussian}
(V+W\ast\mu_\beta)(x)=\tfrac12R|x|^2,\qquad\mu_\beta(x)=\tfrac{\beta R}{2\pi}e^{-\frac12\beta R|x|^2},
\end{equation}
and we shall then prove Theorem~\ref{th:main-Gauss}.

\subsection{Preliminary BBGKY analysis}
We start by noting that the BBGKY hierarchy for correlations as derived in Lemma~\ref{lem:eqn-gNm-21} simplifies drastically in the Gaussian setting.

\begin{lem}[BBGKY hierarchy for correlations]\label{lem:eqn-gNm-21-G}
In the Gaussian setting~\eqref{eq:Gaussian},
the correlation functions satisfy the BBGKY hierarchy~\eqref{eq:eqn-gNm-21}
where the defining operators are now given for all $1\le m\le N$ by
\begingroup\allowdisplaybreaks
\begin{eqnarray*}
iL_{N,\beta}^mg_N^m&:=&-\tfrac{N-m}N\sum_{j=2}^m\beta Rx_j\cdot\int_{\R^2} K(x_j-x_*)\,g_N^{m}(x_{[m]\setminus\{j\}},x_*)\,\mu_\beta(x_*)\,dx_*,\\
iS_{N,\beta}^{m,+}g_N^{m+1}&:=&-\tfrac{N-m}N\sum_{j=1}^m\nabla_{j;\beta}\cdot \int_{\R^2} K(x_j-x_*)\,g_N^{m+1}(x_{[m]},x_*)\,\mu_\beta(x_*)\,dx_*,\\
iS_{N,\beta}^{m,\circ}g_N^{m}&:=&\sum_{i=1}^m\sum_{2\le j\le m\atop i\ne j}\nabla_{i;\beta}\cdot\int_{\R^2}K(x_i-x_*)\,g_N^{m}(x_{[m]\setminus\{j\}},x_*)\,\mu_\beta(x_*)\,dx_*\\
&&-\sum_{i,j=1}^m\Big(K(x_i-x_j)-K\ast\mu_\beta(x_i)\Big)\cdot\nabla_{i;\beta}g_N^m,\\
iS_{N,\beta}^{m,-}g_N^{m-1}&:=&-\sum_{2\le i,j\le m}^{\ne}\beta Rx_i\cdot\int_{\R^2} K(x_i-x_*)\,g_N^{m-1}(x_{[m]\setminus\{i,j\}},x_*)\,\mu_\beta(x_*)\,dx_*\\
&&-\sum_{i=1}^m\sum_{2\le j\le m\atop i\ne j}\Big(K(x_i-x_j)-(K\ast\mu_\beta)(x_i)\Big)\cdot\nabla_{i;\beta}g_N^{m-1}(x_{[m]\setminus\{j\}})\\
&&+\sum_{i=2}^m\sum_{j=1}^m\beta Rx_i\cdot K(x_i-x_j)\,g_N^{m-1}(x_{[m]\setminus\{i\}}),\\
iS_{N,\beta}^{m,=}g_N^{m-2}&:=&0,
\end{eqnarray*}
\endgroup
with the short-hand notation $\nabla_{i;\beta}=\nabla_i-\beta Rx_i$.
\end{lem}

\begin{proof}
As $V,W,f^\circ$ are radial, we first note that the initial data~\eqref{eq:fN0-2} satisfies
\[f_N^\circ(Ox_1,\ldots,Ox_N)\,=\,f_N^\circ(x_1,\ldots,x_N),\qquad\text{for all $O\in O(2)$,}\]
which remains true over time by the Liouville equation. In particular, via~\eqref{eq:correl-def-2}, we deduce
\[g_N^m(Ox_1,\ldots,Ox_m)\,=\,g_N^m(x_1,\ldots,x_m),\qquad\text{for all $1\le m\le N$ and $O\in O(2)$,}\]
which implies the differential identity
\[\sum_{j=1}^mx_j^\bot\cdot\nabla_jg_N^m\,=\,0.\]
In the Gaussian setting~\eqref{eq:Gaussian}, as we have $F+K\ast\mu_\beta=-Rx^\bot$, we deduce
\[\sum_{j=1}^m(K\ast\mu_\beta+F)(x_j)\cdot\nabla_jg_N^m\,=\,0,\]
which yields the different simplifications in the definition of the relevant operators compared to Lemma~\ref{lem:eqn-gNm-21}.
\end{proof}

The above shows that the linearized mean-field operators $\{L_N^m\}_{1\le m\le N}$ can now be written as Kronecker sums
\begin{equation}\label{eq:linMF-G}
L_{N,\beta}^m\,=\,\tfrac{N-m}N\sum_{j=2}^m(\Id_{\Ld^2_\beta(\R^2)})^{\otimes j-1}\otimes T_\beta\otimes (\Id_{\Ld^2_\beta(\R^2)})^{\otimes(m-j)},
\end{equation}
involving the following single-particle operator on $\Ld^2_\beta(\R^2)$,
\begin{equation}\label{eq:def-Lbeta}
(iT_\beta h)(x)\,:=\,-\beta Rx\cdot\int_{\R^2} K(x-x_*) \,h(x_*)\,\mu_\beta(x_*)\,dx_*.
\end{equation}
The following result states that this operator is compact, and identifies its kernel.

\begin{lem}\label{lem:Gaus}
Consider the Gaussian setting~\eqref{eq:Gaussian} and assume that $W$ does not vanish identically.
Then, for all $\beta>0$, the above-defined operator~$T_\beta$ is compact and self-adjoint on $\Ld^2_\beta(\R^2)$. Moreover, its kernel coincides with the set of radial functions.
\end{lem}

\begin{proof}
We split the proof into three steps.

\medskip
\step1 Proof that $T_\beta$ is compact.\\
Given a weakly converging sequence $h_n\cvf h$ in $\Ld^2_\beta(\R^2)$, we find $T_\beta h_n\to T_\beta h$ a.e. By the Cauchy--Schwarz inequality, the definition~\eqref{eq:def-Lbeta} of $T_\beta$ can be bounded by
\[|T_\beta h_n(x)|\,\le\,|\beta Rx|\,\|K\|_{\Ld^\infty(\R^2)}\big(\textstyle\sup_n\|h_n\|_{\Ld^2_\beta(\R^2)}\big)~\,\in\,~\Ld^2_\beta(\R^2),\]
so that the dominated convergence theorem entails $T_\beta h_n\to T_\beta h$ strongly in $\Ld^2_\beta(\R^2)$. This proves that $T_\beta$ is compact.

\medskip
\step2 Proof that $T_\beta$ is self-adjoint.\\
By Step~1, in order to show that $T_\beta$ is self-adjoint on $\Ld^2_\beta(\R^2)$, it remains to check that it is symmetric. Recalling that~\mbox{$x\cdot K(x)=0$}, we can write
\begin{eqnarray*}
\langle h',T_\beta h\rangle_{\Ld^2_\beta(\R^2)}
&=&i\beta R\iint_{\R^2\times\R^2}x\cdot K(x-x_*)\, \overline{h'(x)}\,h(x_*)\,\mu_\beta(x)\,\mu_\beta(x_*)\,dxdx_*\\
&=&\tfrac i2\beta R\iint_{\R^2\times\R^2}(x+x_*)\cdot K(x-x_*)\,\overline{h'(x)}\,h(x_*)\,\mu_\beta(x)\,\mu_\beta(x_*)\,dxdx_*.
\end{eqnarray*}
As $K$ satisfies $K(-x)=-K(x)$, this proves that $T_\beta$ is symmetric.

\medskip
\step3 Identification of $\ker(T_\beta)$.\\
Recalling again that~\mbox{$x\cdot K(x)=0$}, as well as $\nabla\mu_\beta=-\beta Rx\mu_\beta$ and $K=-\nabla^\bot W$, and integrating by parts, we can rewrite the definition~\eqref{eq:def-Lbeta} of $T_\beta$ as
\begin{eqnarray*}
(iT_\beta h)(x)&=&-\beta R\int_{\R^2} K(x-x_*) \cdot x_*\,h(x_*)\,\mu_\beta(x_*)\,dx_*\\
&=&2\int_{\R^2} K(x-x_*) \cdot (h\sqrt{\mu_\beta})(x_*)\nabla\sqrt{\mu_\beta}(x_*)\,dx_*\\
&=&-2\int_{\R^2} W(x-x_*) \nabla^\bot(h\sqrt{\mu_\beta})(x_*)\cdot\nabla\sqrt{\mu_\beta}(x_*)\,dx_*.
\end{eqnarray*}
If $h\in\Ld^2_\beta(\R^2)$ belongs to $\ker(T_\beta)$, we deduce from this identity that
\begin{equation*}
H:=F\big\{\nabla^\bot(h\sqrt{\mu_\beta})\cdot\nabla\sqrt{\mu_\beta}\big\}=0~~~\text{a.e. on the support of $F\{W\}$,}
\end{equation*}
where we use here the notation $F\{g\}$ for the Fourier transform of a function $g$ on $\R^2$. As~$\mu_\beta$ is Gaussian and as $h\sqrt{\mu_\beta}\in\Ld^2(\R^d)$, we note that $H$ is real analytic. Given that it vanishes on the support of $F\{W\}$, and noting that the latter contains an open set as $W$ is integrable and does not vanish identically, we deduce that $H$ vanishes everywhere on $\R^2$. Inverting the Fourier transform, this means
\[\nabla^\bot(h\sqrt{\mu_\beta})\cdot\nabla\sqrt{\mu_\beta}=0~~~\text{a.e.},\]
or equivalently $x^\bot\cdot\nabla h(x)=0$ a.e., which precisely means that $h$ is radial.
\end{proof}

Next, we establish the following useful preliminary estimates on BBGKY operators in the weighted negative Sobolev spaces defined in~\eqref{eq:Hsbeta}. We emphasize that item~(ii) establishes that linearized mean-field evolutions are almost uniformly bounded in negative Sobolev spaces in sharp contrast with the non-Gaussian case of Lemma~\ref{lem:prelest-LS-re}(ii).

\begin{lem}\label{lem:prelest-LS}\
\begin{enumerate}[(i)]
\item \emph{Weak bounds on BBGKY operators:}\\
For all $1\le m\le N$, $s\ge0$, and $h^{m+r}\in C^\infty_c((\R^2)^{m+r})$ for $r\in\{-1,0,1\}$, we have
\begin{eqnarray*}
\|L^{m}_{N,\beta}h^{m}\|_{H^{-s}_\beta((\R^2)^m)}&\lesssim_{m,s}&\|h^{m}\|_{H^{-s}_\beta((\R^2)^{m})},\\
\|S^{m,\circ}_{N,\beta}h^{m}\|_{H^{-s-1}_\beta((\R^2)^m)}&\lesssim_{m,s}&\|h^{m}\|_{H^{-s}_\beta((\R^2)^{m})},\\
\|S^{m,+}_{N,\beta}h^{m+1}\|_{H^{-s-1}_\beta((\R^2)^m)}&\lesssim_{m,s}&\|h^{m+1}\|_{H^{-s}_\beta((\R^2)^{m+1})},\\
\|S^{m,-}_{N,\beta}h^{m-1}\|_{H^{-s-1}_\beta((\R^2)^m)}&\lesssim_{m,s}&\|h^{m-1}\|_{H^{-s}_\beta((\R^2)^{m-1})}.
\end{eqnarray*}
\item \emph{Weak bounds on linearized mean-field evolutions:}\\
For all $1\le m\le N$, $s\ge0$, $\delta>0$, and $h^m\in C^\infty_c((\R^2)^{m})$, we have
\[\|e^{iL_{N,\beta}^mt}h^m\|_{H^{-s}_\beta((\R^2)^m)}\,\lesssim_{m,s,\delta}\,\langle t\rangle^{\delta}\|h^m\|_{H^{-s}_\beta((\R^2)^m)}.\]
\end{enumerate}
\end{lem}

\begin{proof}
Item~(i) is obtained by duality, as in the proof of Lemma~\ref{lem:prelest-LS-re}, and we skip the detail. 
We turn to the proof of item~(ii). By duality, recalling that the operator $L_{N,\beta}^m$ in~\eqref{eq:linMF-G} is self-adjoint by Lemma~\ref{lem:Gaus}, it suffices to prove for all $s\ge0$, $\delta>0$, and $h^m\in C^\infty_c((\R^2)^m)$,
\begin{equation*}
\|e^{iL_{N,\beta}^mt}h^m\|_{H_\beta^s((\R^2)^m)}\,\lesssim_{m,s,\delta}\,\langle t\rangle^{\delta}\|h^m\|_{H^s_\beta((\R^2)^m)}.
\end{equation*}
As by definition $L_{N,\beta}^m$ is (a multiple of) the Kronecker sum of $T_\beta$ over the last $m-1$ directions on $\Ld^2_\beta((\R^2)^m)$, cf.~\eqref{eq:linMF-G}, it suffices to show, for all $s\ge0$, $\delta>0$, and $h\in C^\infty_c(\R^2)$,
\begin{equation}\label{eq:estim-Hs-growth-Gaus}
\|e^{iT_\beta t}h\|_{H_\beta^s(\R^2)}\,\lesssim_{s,\delta}\,\langle t\rangle^\delta\|h\|_{H^s_\beta(\R^2)}.
\end{equation}
For that purpose, we first note that by definition~\eqref{eq:def-Lbeta} the operator $T_\beta$ satisfies the following regularizing property: for all $r\ge0$,
\[\|T_\beta h\|_{H^r_\beta(\R^2)}\,\lesssim_r\,\|h\|_{\Ld^2_\beta(\R^2)}.\]
Writing $e^{iT_\beta t}h-h=\int_0^tiT_\beta(e^{iT_\beta t'}h)\,dt'$ and using this regularizing property, we then get
\[\|e^{iT_\beta t}h-h\|_{H^r_\beta(\R^2)}\,\le\,\int_0^t\|T_\beta(e^{iT_\beta t'}h)\|_{H^r_\beta(\R^2)}\,dt'\,\lesssim_r\,\int_0^t\|e^{iT_\beta t'}h\|_{\Ld^2_\beta(\R^2)}\,dt'.\]
Now we combine this with the interpolation inequality of Lemma~\ref{lem:interpol-SobG}(ii): for all $r\ge s\ge0$, we get
\begin{eqnarray*}
\|e^{iT_\beta t}h-h\|_{H_\beta^s(\R^2)}&\lesssim_{s,r}&\|e^{iT_\beta t}h-h\|_{\Ld^2_\beta(\R^2)}^{1-\frac sr}\|e^{iT_\beta t}h-h\|_{H_\beta^r(\R^2)}^\frac sr\\
&\lesssim_{r}&\|e^{iT_\beta t}h-h\|_{\Ld^2_\beta(\R^2)}^{1-\frac sr}\Big(\int_0^t\|e^{iT_\beta t'}h\|_{\Ld^2_\beta(\R^2)}\,dt'\Big)^\frac sr.
\end{eqnarray*}
By the triangle inequality and the self-adjointness of $T_\beta$ on $\Ld^2_\beta(\R^2)$, cf.~Lemma~\ref{lem:Gaus}, we conclude that
\begin{equation*}
\|e^{iT_\beta t}h\|_{H_\beta^s(\R^2)}\,\lesssim_{s,r}\,\|h\|_{H_\beta^s(\R^2)}+\langle t\rangle^\frac sr\|h\|_{\Ld^2_\beta(\R^2)},
\end{equation*}
and the claim~\eqref{eq:estim-Hs-growth-Gaus} follows by choosing $r\ge \frac s\delta$.
\end{proof}

\subsection{Proof of Theorem~\ref{th:main-Gauss}}
We start by using the cumulant estimates of Lemma~\ref{lem:est-cum-2} to truncate the BBGKY hierarchy and get a closed description of the tagged particle density. We manage to get arbitrarily close to the critical timescale $t=O(N^{1/2})$ by considering multiple iterations of the hierarchy. This iteration procedure works precisely thanks to the almost uniform bound on linearized mean-field evolutions that we proved above, cf.~Lemma~\ref{lem:prelest-LS}(ii), in contrast with the non-Gaussian case where we could not avoid a strong time restriction, cf.~Lemma~\ref{lem:exp-gn1-NG}.

\begin{lem}\label{lem:approx-dt2g1-gaus}
Consider the Gaussian setting~\eqref{eq:Gaussian}.
For all $\sigma<\frac12$, there exist $\delta,s>0$
such that
\begin{equation*}
\sup_{0\le t\le N^\sigma}\Big\|N\partial_t^2g_N^{1;t}-iS_{N,\beta}^{1,+}e^{-it(\Id\otimes T_\beta)}iS_{N,\beta}^{2,-}(\tfrac1{\mu_\beta}f^\circ)\Big\|_{H^{-s}_\beta(\R^2)}\,\lesssim_\sigma\,N^{-\delta}.
\end{equation*}
\end{lem}

\begin{proof}
Recall that in the Gaussian setting the BBGKY hierarchy~\eqref{eq:eqn-gNm-21} holds with $L_{N,\beta}^1=0$, with $S_{N,\beta}^{m,=}=0$ for all $m\ge1$, and with $L_{N,\beta}^2=\frac{N-2}N(\Id\otimes T_\beta)$, where $T_\beta$ is defined in~\eqref{eq:def-Lbeta}, cf.~Lemma~\ref{lem:eqn-gNm-21-G}.
For the tagged particle density, we then find
\begin{equation}\label{eq:dt1g1}
\partial_tg_N^1\,=\,iS_{N,\beta}^{1,+}g_N^2+\tfrac1NiS_{N,\beta}^{1,\circ}g_N^1,
\end{equation}
and thus, taking another time derivative and iterating,
\begin{equation}\label{eq:dt2g1}
N\partial_t^2g_N^1-iS_{N,\beta}^{1,+}N\partial_tg_N^2\,=\,-S_{N,\beta}^{1,\circ}S_{N,\beta}^{1,+}g_N^2-\tfrac1{N}S_{N,\beta}^{1,\circ}S_{N,\beta}^{1,\circ}g_N^1.
\end{equation}
Next, we further appeal to the BBGKY hierarchy to express the correlation function $N\partial_tg_N^2$ in the left-hand side. For that purpose, we shall use the following version of the Duhamel formula: if $g,h$ satisfy an equation of the form
\[\partial_tg+iLg=h,\qquad g|_{t=0}=0,\]
for a self-adjoint operator $L$, then we can write the solution as $g^t=\int_0^te^{-iL(t-s)}h^s\,ds$, from which we can deduce for instance, taking a time derivative and integrating by parts,
\begin{eqnarray}
\partial_tg^t&=&\int_0^t\big(\delta(s-t)-e^{-iL(t-s)}iL\big)h^s\,ds\nonumber\\
&=&e^{-iLt}h^\circ+\int_0^te^{-iL(t-s)}\partial_sh^s\,ds.\label{eq:Duhamel-re}
\end{eqnarray}
Now, as $L^2_{N,\beta}=\frac{N-2}N(\Id\otimes T_\beta)$, the BBGKY equation~\eqref{eq:eqn-gNm-21} for the correlation function $g_N^2$ takes the form
\[\partial_tg^2_N+i(\Id\otimes T_\beta)g^2_N=iS^{2,+}_{N,\beta}g_N^3+\tfrac1N\Big(iS^{2,\circ}_{N,\beta}g_N^2+iS_{N,\beta}^{2,-}g^1_N+2i(\Id\otimes T_\beta)g_N^2\Big),\]
hence, using $g_N^1|_{t=0}=\tfrac1{\mu_\beta}f^\circ$ and $g_N^m|_{t=0}=0$ for all $m\ge2$, the above Duhamel formula~\eqref{eq:Duhamel-re} yields
\begin{multline*}
N\partial_tg_N^{2;t}\,=\,e^{-it(\Id\otimes T_\beta)}iS_{N,\beta}^{2,-}(\tfrac1{\mu_\beta}f^\circ)
+\int_0^te^{-i(t-s)(\Id\otimes T_\beta)}iS_{N,\beta}^{2,+}N\partial_sg_N^{3;s}\,ds\\
+\int_0^te^{-i(t-s)(\Id\otimes T_\beta)}\Big(iS_{N,\beta}^{2,\circ}\partial_sg_N^{2;s}+iS_{N,\beta}^{2,-}\partial_sg_N^{1;s}+2i(\Id\otimes T_\beta)\partial_sg_N^{2;s}\Big)ds.
\end{multline*}
Further replacing $\partial_sg_N^1$ by~\eqref{eq:dt1g1} in the right-hand side, and integrating by parts as in~\eqref{eq:Duhamel-re} to remove the time derivative on $g_N^2$, we get
\begin{multline}\label{eq:ide-gN2Ndt}
N\partial_tg_N^{2;t}\,=\,e^{-it(\Id\otimes T_\beta)}iS_{N,\beta}^{2,-}(\tfrac1{\mu_\beta}f^\circ)
+\int_0^te^{-i(t-s)(\Id\otimes T_\beta)}iS_{N,\beta}^{2,+}N\partial_sg_N^{3;s}\,ds\\
+\int_0^te^{-i(t-s)(\Id\otimes T_\beta)}\Big(iS_{N,\beta}^{2,-}iS_{N,\beta}^{1,+}g_N^{2;s}+\tfrac1NiS_{N,\beta}^{2,-}iS_{N,\beta}^{1,\circ}g_N^{1;s}\Big)ds\\
+\int_0^t\Big(\delta(s-t)- e^{-i(t-s)(\Id\otimes T_\beta)}i(\Id\otimes T_\beta)\Big)\big(iS_{N,\beta}^{2,\circ}+2i(\Id\otimes T_\beta)\big)g_N^{2;s}\,ds.
\end{multline}
{Appealing to Lemma~\ref{lem:prelest-LS} to estimate the different operators in the right-hand side, we get for all~$s\ge1$ and $T,\delta>0$,
\begin{multline*}
\sup_{0\le t\le T}\|N\partial_tg_N^{2;t}\|_{H_\beta^{-s-1}(\R^2)}\,\lesssim_{s,\delta}\,\langle T\rangle^\delta
+\langle T\rangle^{1+\delta}\Big(N^{-1}\sup_{0\le t\le T}\|g_N^{1;t}\|_{H_\beta^{1-s}(\R^2)}\\
+\sup_{0\le t\le T}\|g_N^{2;t}\|_{H_\beta^{1-s}(\R^2)}
+\sup_{0\le t\le T}\|N\partial_tg_N^{3;t}\|_{H_\beta^{-s}(\R^2)}\Big),
\end{multline*}
and thus, by the a priori estimates of Lemma~\ref{lem:est-cum-2},
\begin{multline}\label{eq:estim-dt2g1}
\sup_{0\le t\le T}\|N\partial_tg_N^{2;t}\|_{H_\beta^{-s-1}(\R^2)}\\[-3mm]
\,\lesssim_{s,\delta}\,\langle T\rangle^\delta
+N^{-\frac12}\langle T\rangle^{1+\delta}
+\langle T\rangle^{1+\delta}\sup_{0\le t\le T}\|N\partial_tg_N^{3;t}\|_{H_\beta^{-s}(\R^2)}.
\end{multline}
Inserting the identity~\eqref{eq:ide-gN2Ndt} for $N\partial_tg_N^2$ into~\eqref{eq:dt2g1}, and similarly estimating the error terms as above using Lemmas~\ref{lem:est-cum-2} and~\ref{lem:prelest-LS}, we also find}
\begin{multline}\label{eq:estim-dt2g1-main}
\sup_{0\le t\le T}\Big\|N\partial_t^2g_N^{1;t}-iS_{N,\beta}^{1,+}e^{-it(\Id\otimes T_\beta)}iS_{N,\beta}^{2,-}(\tfrac1{\mu_\beta}f^\circ)\Big\|_{H^{-s-2}_\beta(\R^2)}\\
\,\lesssim_{s,\delta}\,N^{-\frac12}\langle T\rangle^{1+\delta} +\langle T\rangle^{1+\delta}\sup_{0\le t\le T}\|N\partial_tg_N^{3;t}\|_{H^{-s}_\beta(\R^2)}.
\end{multline}
It remains to estimate the $3$-particle correlation function $\partial_tg_N^3$ in~\eqref{eq:estim-dt2g1} and~\eqref{eq:estim-dt2g1-main}. In order to get an estimate valid up to the critical timescale $t=O(N^{1/2})$, some special care is needed and we shall argue by iterations on the whole hierarchy.
For any $m\ge3$, applying the Duhamel formula~\eqref{eq:Duhamel-re} to the BBGKY equation~\eqref{eq:eqn-gNm-21} for $g_N^m$, we get
\begin{multline*}
N\partial_tg_N^{m;t}=\int_0^te^{-iL_N^n(t-s)}iS_{N,\beta}^{m,+}N\partial_sg_N^{m+1;s}\,ds\\
+\int_0^te^{-iL_N^n(t-s)}\Big(iS_{N,\beta}^{m,\circ}\partial_sg_N^{m;s}+iS_{N,\beta}^{m,-}\partial_sg_N^{m-1;s}\Big)\,ds,
\end{multline*}
and therefore, by Lemmas~\ref{lem:est-cum-2} and~\ref{lem:prelest-LS}, we get for any $s\ge0$ and $T,\delta>0$,
\begin{multline*}
\sup_{0\le t\le T}\|N\partial_tg_N^{m;t}\|_{H^{-s-1}_\beta(\R^2)}\,\lesssim_{s,\delta}\,N^{-1}\langle T\rangle^{1+\delta}\sup_{0\le t\le T}\|N\partial_tg_N^{m;t}\|_{H^{-s}_\beta(\R^2)}\\
+\langle T\rangle^{1+\delta}\sup_{0\le t\le T}\|N\partial_tg_N^{m+1;t}\|_{H^{-s}_\beta(\R^2)}
+N^{-1}\langle T\rangle^{1+\delta}\sup_{0\le t\le T}\|N\partial_tg_N^{m-1;t}\|_{H^{-s}_\beta(\R^2)}.
\end{multline*}
Denoting
\begin{equation}\label{eq:notation-Asm}
A_{s}^m(T)\,:=\,\sup_{0\le t\le T}\|N^{\frac{m-1}2}\partial_tg_N^{m;t}\|_{H^{-s}_\beta(\R^2)},
\end{equation}
the above means for all $m\ge3$, $s\ge0$, and $T,\delta>0$,
\begin{equation}\label{eq:toiterate-Asm}
A_{s+1}^m(T)\,\lesssim_{s,\delta}\,N^{-1}\langle T\rangle^{1+\delta}A_s^m(T)
+N^{-\frac12}\langle T\rangle^{1+\delta}\big(A_s^{m+1}(T)+A_s^{m-1}(T)\big).
\end{equation}
Further note that the BBGKY equation~\eqref{eq:eqn-gNm-21} for $\partial_tg_N^m$, combined with Lemmas~\ref{lem:est-cum-2} and~\ref{lem:prelest-LS}, yields the following a priori estimates, for all $m\ge2$, $s\ge1$, and $T>0$,
\begin{eqnarray}
A_s^m(T)&\le&\sup_{0\le t\le T}N^{\frac{m-1}2}\Big( \|iL_{N,\beta}^mg^{m;t}_N\|_{H^{-s}_\beta(\R^2)}+\|iS_{N,\beta}^{m,+}g_N^{m+1;t}\|_{H^{-s}_\beta(\R^2)}\nonumber\\
&&\hspace{2cm}+\tfrac1N\|iS^{m,\circ}_{N,\beta}g_N^{m;t}\|_{H^{-s}_\beta(\R^2)}+\tfrac1N\|iS_{N,\beta}^{m,-}g_N^{m-1;t}\|_{H^{-s}_\beta(\R^2)}\Big)\nonumber\\
&\lesssim&1.\label{eq:apriori-Asm}
\end{eqnarray}
Now combining~\eqref{eq:toiterate-Asm} and~\eqref{eq:apriori-Asm},
we deduce by direct iteration, for all $s\ge1$, $k\ge0$, and~$T,\delta>0$, provided that $N^{-1/2}\langle T\rangle^{1+\delta}\le1$,
\[A_{s+k}^3(T)\,\lesssim_{s,k,\delta}\,N^{-1}\langle T\rangle^{1+\delta}+(N^{-\frac12}\langle T\rangle^{1+\delta})^k+N^{-\frac12}\langle T\rangle^{1+\delta}A_s^2(T).\]
Noting that~\eqref{eq:estim-dt2g1} means, with the notation~\eqref{eq:notation-Asm}, for $s\ge1$,
\begin{equation}
A_{s+1}^2(T)
\,\lesssim_{s,\delta}\,N^{-\frac12}{\langle T\rangle^\delta}+N^{-1}\langle T\rangle^{1+\delta} +N^{-\frac12}\langle T\rangle^{1+\delta}A_s^3(T),
\end{equation}
and inserting this into the above,
we deduce for all $s\ge1$, $k\ge0$, and $T,\delta>0$, provided that~$N^{-1/2}\langle T\rangle^{1+\delta}\le1$,
\[A_{s+k+1}^3(T)\,\lesssim_{s,k,\delta}\,N^{-1}\langle T\rangle^{1+2\delta}+(N^{-\frac12}\langle T\rangle^{1+\delta})^k+(N^{-\frac12}\langle T\rangle^{1+\delta})^2A_{s}^3(T).\]
Now further iterating this inequality for $A^3$, together with the a priori estimate~\eqref{eq:apriori-Asm}, we deduce for all $k\ge1$, $s\ge k^2$, and $T>0$, provided that $N^{-1/2}\langle T\rangle^{1+\delta}\le1$,
\[\sup_{0\le t\le T}\|N\partial_tg_N^{3;t}\|_{H^{-s}_\beta(\R^2)}\,=\,A_{s}^3(T)\,\lesssim_{s,k,\delta}\,N^{-1}\langle T\rangle^{1+2\delta}+(N^{-\frac12}\langle T\rangle^{1+\delta})^k.\]
Combining this with~\eqref{eq:estim-dt2g1-main} and optimizing in $k$ yield the conclusion.
\end{proof}

In view of the compact nature of the linearized mean-field operator $T_\beta$, as formulated in Lemma~\ref{lem:Gaus}, we may now pass to the limit in the above closed description of the tagged particle density and conclude the proof of Theorem~\ref{th:main-Gauss}.

\begin{proof}[Proof of Theorem~\ref{th:main-Gauss}]
By definition of BBGKY operators in Lemma~\ref{lem:eqn-gNm-21}, we have that
\begin{multline}\label{eq:defiSevS-expl}
iS_{N,\beta}^{1,+}e^{-it(\Id\otimes T_\beta)}iS_{N,\beta}^{2,-}(\tfrac{1}{\mu_\beta}f^{\circ})\\
\,=\,-\tfrac{N-1}N(\nabla-\beta Rx)\cdot\int_{\R^2}K(x-x_*)\,(e^{-it(\Id\otimes T_\beta)}H)(x,x_*)\,\mu_\beta(x_*)\,dx_*,
\end{multline}
where $H\in \Ld^2_\beta((\R^2)^2)$ is given by
\begin{eqnarray*}
H(x,x_*)&=&-\Big(K(x-x_*)-(K\ast\mu_\beta)(x)\Big)\cdot(\nabla-\beta Rx)(\tfrac{1}{\mu_\beta}f^{\circ})(x)\\
&=&-K(x-x_*)\cdot(\tfrac{1}{\mu_\beta}\nabla f^{\circ})(x).
\end{eqnarray*}
Since $T_\beta$ is compact, cf.~Lemma~\ref{lem:Gaus}, we find as in~\eqref{eq:norelax-comp}, for all $h\in\Ld^2_\beta(\R^2)$ and $\sigma>0$,
\begin{equation*}
e^{-i\tau N^\sigma T_\beta}h\,\xrightarrow{N\uparrow\infty}\,\pi_0h,
\end{equation*}
in the weak-* sense of $\Ld^\infty(\R^+;\Ld^2_\beta(\R^2))$, where $\pi_0$ stands for the orthogonal projection of~$\Ld^2_\beta(\R^2)$ onto the kernel of $T_\beta$, that is, by Lemma~\ref{lem:Gaus}, onto the subspace of radial functions.
Using this in~\eqref{eq:defiSevS-expl} with $t=N^\sigma\tau$, we deduce the following weak-* convergence in $\Ld^\infty(\R^+;H^{-1}_\beta(\R^2))$,
\begin{multline*}
iS_N^{1,+}e^{-i\tau N^\sigma(\Id\otimes T_\beta)}iS_N^{2,-}(\tfrac{1}{\mu_\beta}f^{\circ})\\
\,\xrightarrow{N\uparrow\infty}\,-(\nabla-\beta Rx)\cdot\int_{\R^2}K(x-x_*)\,((\Id\otimes\pi_0)H)(x,x_*)\,\mu_\beta(x_*)\,dx_*.
\end{multline*}
By definition of $H$, the weak-* limit can be rewritten as
\begin{equation*}
iS_N^{1,+}e^{-i\tau N^\sigma(\Id\otimes T_\beta)}iS_N^{2,-}(\tfrac{1}{\mu_\beta}f^{\circ})\\
\,\xrightarrow{N\uparrow\infty}\,\tfrac1{\mu_\beta}\Div( A\nabla f^\circ),\qquad\text{in $\Ld^\infty(\R^+;H^{-1}_\beta(\R^2))$},
\end{equation*}
where the coefficient field $A\in\Ld^\infty(\R^2)^{2\times 2}$ is given by
\begin{equation*}
A(x)\,:=\,\int_{\R^2}K(x-x_*)\otimes\Big((\Id\otimes\pi_0)\big[(x_1,x_2)\mapsto K(x_1-x_2)\big]\Big)(x,x_*)\,\mu_\beta(x_*)\,dx_*.
\end{equation*}
The projection $\pi_0$ onto the subspace of radial functions can be explicitly computed and we recover formula~\eqref{eq:def-coeff-A-proj} for $A$.
Combining this convergence with the estimate of Lemma~\ref{lem:approx-dt2g1-gaus}, and using~\eqref{eq:embeddingHs} in form of $\|h\|_{H^{-s}(\R^2)}\lesssim\|\frac1{\mu_\beta} h\|_{H^{-s}_\beta(\R^2)}$,
the conclusion follows.
\end{proof}

\section{Special degenerate case: uniform equilibrium}\label{sec:unif}

This section is devoted to the special degenerate case of a tagged particle in a uniformly distributed background on the torus $\T^d$, say in arbitrary spatial dimension $d\ge1$.
More precisely, given a smooth force kernel~$K\in C^\infty(\T^d)^d$ that satisfies the incompressibility condition $\Div(K)=0$ and the action-reaction condition \mbox{$K(-x)=-K(x)$}, we now consider the associated point-vortex dynamics
\[\partial_tx_i=\tfrac1N\sum_{j=1}^NK(x_i-x_j),\qquad x_i|_{t=0}=x_i^\circ,\qquad 1\le i\le N.\]
The Liouville equation for the $N$-point density $f_N\in\Pc((\T^d)^N)$ then reads
\begin{equation}\label{eq:Liouville}
\partial_tf_N+\tfrac1N\sum_{i,j=1}^NK(x_i-x_j)\cdot\nabla_if_N=0,\qquad f_N|_{t=0}=f_N^\circ.
\end{equation}
Consider a tagged particle (labeled `1') in a uniform equilibrium background: in other words, we assume that initially the $N$-point density $f_N|_{t=0}=f_N^\circ$ takes the form
\begin{equation}\label{eq:initial-unif}
f_N^{\circ}(x_1,\ldots,x_N)\,=\,f^{\circ}(x_1),
\end{equation}
for some $f^{\circ}\in \Pc\cap C^\infty(\T^d)$.
At later times, the tagged particle density is given by the first marginal
\[f_N^1(t,x_1)\,:=\,\int_{(\R^d)^{N-1}}f_N(t,x_1,x_2,\ldots,x_N)\,dx_2\ldots dx_N.\]
As linearized mean-field operators actually vanish at uniform equilibrium, Conjecture~\ref{conj} leads us to expect a nontrivial slow conservative dynamics for the tagged particle on the timescale $t=O(N^{1/2})$, displaying no thermalization in the strict sense.
In the spirit of~\eqref{eq:compact-wave}, we start by showing that the tagged particle density satisfies a linear wave equation to leading order for relatively short times $t\ll N^{1/2}$ (that is, $\tau\ll1$). This is analogous to Theorem~\ref{th:main-Gauss} in the Gaussian degenerate setting.
The proof is displayed in Section~\ref{sec:proof-unif-wave}.

\begin{prop}\label{prop:unif-wave}
The critically-rescaled tagged particle density
$\bar f_N^1(\tau):=f_N^1(N^{\frac12}\tau)$
satisfies
\[\|\partial_\tau^2\bar f_N^{1}(\tau)-\Div(A\nabla f^{\circ})\|_{\Ld^2(\T^d)}\,\lesssim_{f^\circ}\,\tau^2+\tfrac1N,\]
with constant diffusion matrix
$A:=\int_{\T^d}K\otimes K$.
\end{prop}

For a description of the dynamics on the critical timescale $t=O(N^{1/2})$, we need to consider the tagged particle density together with the whole family of its correlation functions with respect to background particles. We first recall some standard definitions.
As in~\eqref{eq:def-fNm}, we denote by~$\{f_N^m\}_{1\le m\le N}$ the marginals of $f_N$,
\begin{equation}\label{eq:def-fNm-re}
f_N^m(t,x_1,\ldots,x_m)\,:=\,\int_{(\T^d)^{N-m}}f_N(t,x_1,\ldots,x_m,x_{m+1},\ldots,x_N)\,dx_{m+1}\ldots dx_N.
\end{equation}
As background particles with labels $2,\ldots,N$ are exchangeable initially, cf.~\eqref{eq:initial-unif}, they remain so over time, hence $f_N^m$ is symmetric in its last $m-1$ variables.
The correlation functions $\{g_N^m\}_{1\le m\le N}$ for the tagged particle with respect to the initially uniform background are defined so as to satisfy the following cluster expansions,
\begin{equation}\label{eq:correl-def-clust}
f_N^m(t,x_1,\ldots,x_m)\,=\,\sum_{n=1}^m\sum_{\sigma\in P_{n-1}^{m-1}}g_N^{n}(t,x_1,x_\sigma),\qquad 1\le m\le N.
\end{equation}
For all $m$, the correlation function $g_N^m$ is uniquely chosen to be symmetric in its last $m-1$ variables and to satisfy $\int_{\T^d}g_N^m(x_1,\ldots,x_m)\,dx_j=0$ for all $2\le j\le m$.
More explicitly, the above relations can be inverted and correlation functions are given by
\begin{equation}\label{eq:correl-def}
g_N^m(t,x_1,\ldots,x_m):=\sum_{n=1}^m(-1)^{m-n}\sum_{\sigma\in P_{n-1}^{m-1}}f_N^{n}(t,x_1,x_\sigma).
\end{equation}
In these terms, we can now state the following result. In accordance with~\eqref{eq:compact-lim-hier}, the tagged particle density does not satisfy a closed equation on the critical timescale $t=O(N^{1/2})$, but its dynamics takes the form of a (formally) unitary evolution for an infinite hierarchy of coupled equations for the limits of all rescaled correlation functions.

\begin{theor}\label{th:unif-wave-2}
The critically-rescaled correlation functions
\begin{equation}\label{eq:rescaled-correl-1}
\bar g_N^m(\tau):=N^\frac{m-1}2g_N^m(N^\frac12\tau),\qquad1\le m\le N,
\end{equation}
converge weakly-* in $W^{1,\infty}(\R^+;\Ld^2((\T^d)^m))$ as $N\uparrow\infty$,
\[\bar g_N^m\overset*\cvf\bar g^m,\qquad m\ge1,\]
where the limit $\{\bar g^m\}_m$ is the unique weak solution of the limit hierarchy
\begin{equation}\label{eq:lim-hier-unif}
\left\{\begin{array}{lll}
\partial_\tau\bar g^m=iS^{m,+}\bar g^{m+1}+iS^{m,-}\bar g^{m-1}&:&m\ge1,\\
\bar g^1|_{\tau=0}=f^{\circ},\\
\bar g^m|_{\tau=0}=0&:&m>1,
\end{array}\right.
\end{equation}
such that
\begin{equation}\label{eq:estim-energy-dt}
\sup_{\tau\ge0}\sum_{m=1}^\infty\tfrac1{(m-1)!}\|\partial_\tau^k\bar g^m\|^2_{\Ld^2((\T^d)^m)}\,<\,\infty,\qquad\text{for all $k\ge0$},
\end{equation}
and where the operators $S^{m,\pm}$ are explicitly defined as follows,
\begin{eqnarray*}
iS^{m,+}\bar g^{m+1}&:=&-\sum_{i=1}^m\int_{\T^d}K(x_i-x_*)\cdot\nabla_i\bar g^{m+1}(x_{[m]},x_*)\,dx_*,\\
iS^{m,-}\bar g^{m-1}&:=&-\sum_{i=1}^m\sum_{j=2}^mK(x_i-x_j)\cdot\nabla_i\bar g^{m-1}(x_{[m]\setminus\{j\}}).
\end{eqnarray*}
In addition, the limit hierarchy~\eqref{eq:lim-hier-unif} has the following (formal) unitary structure:
defining~$\Hc^m$ as the Hilbert space of functions $h^m\in\Ld^2((\T^d)^m)$ that are symmetric in their last $m-1$ entries, and endowing this space with the norm
\[\|h^m\|_{\Hc^m}^2:=\tfrac1{(m-1)!}\|h^m\|_{\Ld^2((\T^d)^m)}^2,\]
the operators $S^{m,+}:\Hc^{m+1}\to\Hc^m$ and $S^{m+1,-}:\Hc^{m}\to\Hc^{m+1}$ satisfy
\begin{equation}\label{eq:sym-S+S-}
(S^{m,+})^*=S^{m+1,-},\qquad (S^{m+1,-})^*=S^{m,+}.
\end{equation}
\end{theor}

Let us further describe the structure of the limit hierarchy~\eqref{eq:lim-hier-unif}--\eqref{eq:estim-energy-dt} and investigate its actual unitary structure.
For that purpose, consider the Hilbert space
\begin{equation}\label{eq:HilbertHc}
\Hc\,:=\,\overline{\textstyle\bigoplus_{m\ge1}\Hc^m},
\end{equation}
that is, the Hilbert closure of the algebraic direct sum $\bigoplus_{m\ge1}\Hc^m$ with respect to the norm
\[\|h\|_{\Hc}^2\,:=\,\sum_{m=1}^\infty\|h^m\|_{\Hc^m}^2\,=\,\sum_{m=1}^\infty\tfrac1{(m-1)!}\|h^m\|_{\Ld^2((\T^d)^m)}^2.\]
In other words, this means $\Hc=\Ld^2(\T^d)\otimes\mathcal F_+(\Ld^2(\T^d))$,
where the bosonic Fock space $\mathcal F_+(\Ld^2(\T^d))$ is viewed as the state space for background correlations.
In this setting, we consider the operator $S:=S^++S^-$ on $\Hc$ given by
\begin{equation}\label{eq:def-S-oplim}
(Sh)^m\,:=\,S^{m,+}h^{m+1}+S^{m,-}h^{m-1},\qquad m\ge1,
\end{equation}
which is well-defined for all $h$ in the dense subset
\begin{equation}\label{eq:core-def}
\textstyle\Cc\,:=\,\bigoplus_{m\ge1}C^\infty_b((\T^d)^m)\,\subset\,\Hc.
\end{equation}
(Note that the direct sum is understood here in the algebraic sense.)
The symmetry relations~\eqref{eq:sym-S+S-}
precisely mean that this densely-defined operator $S$ is symmetric on~$\Cc$.
In these terms, the fact that the limit $\bar g=\{\bar g^m\}_m$ is a weak solution of the limit hierarchy~\eqref{eq:lim-hier-unif} and that it satisfies the a priori estimates~\eqref{eq:estim-energy-dt} is equivalent to the following: the limit $\bar g=\{\bar g^m\}_m$ belongs to $C^\infty_b(\R^+;\Hc)$
and satisfies
\begin{equation*}
\langle h,\partial_\tau\bar g\rangle_\Hc=-\langle iSh,\bar g\rangle_\Hc,\qquad\text{for all $h\in\Cc$,}
\end{equation*}
with initial condition $\bar g|_{\tau=0}=\bar g^\circ\in\Cc$ given by
\begin{equation}\label{eq:def-barg0}
(\bar g^\circ)^m\,:=\,\left\{\begin{array}{ll}
f^\circ,&\text{for $m=1$},\\
0,&\text{for $m>1$}.
\end{array}\right.
\end{equation}
This actually means that $\bar g$ is a strong solution of
\begin{equation}\label{eq:lim-hier-unif-rep}
\partial_\tau\bar g=iS^*\bar g,\qquad\bar g|_{\tau=0}=\bar g^\circ,
\end{equation}
in terms of the adjoint $S^*$ of $S$ on $\Hc$.
The uniqueness of the solution to this limit equation and the unitarity of the so-defined semigroup would amount to the essential self-adjointness of the operator~$S$ on~$\Hc$, or equivalently to the symmetry of its adjoint~$S^*$.
However, we do not expect~$S$ to be essentially self-adjoint:
this operator can indeed be viewed as a cubic expression in terms of canonical creation and annihilation operators on the Fock space, and in dimension $1$ nontrivial symmetric polynomials of order~$3$ in creation and annihilation operators are known to be non-self-adjoint~\cite{Rabsztyn-89,Otte-03}.
This expected lack of self-adjointness, leading to a lack of unitarity, is related to statistical closure problems in link with turbulence~\cite{Krommes-02}.
In fact, it is not even clear whether $S^*$ generates a semigroup.

Leaving these delicate issues aside, we can at least show that the limit equation~\eqref{eq:lim-hier-unif-rep} is well-posed in some sense. This explains the uniqueness of the limit in Theorem~\ref{th:unif-wave-2}.
The proof is postponed to Section~\ref{sec:unique-pr}. Note that the proof of uniqueness only relies on the symmetry of $S$ on its core $\Cc$ and on the observation that $S\Cc\subset\Cc$.

\begin{prop}[Well-posedness of limiting hierarchy]\label{prop:unique}
For all $g^\circ\in\Cc$, there is a unique strong solution $g\in C^\infty_b(\R^+;\Hc)$ of the effective equation
\begin{equation}\label{eq:limit-eqn-wellposed}
\partial_\tau g=iS^*g,\qquad g|_{\tau=0}=g^\circ.
\end{equation}
Moreover, it satisfies the following properties:
\begin{enumerate}[(i)]
\item \emph{Contraction:} for all $\tau\ge0$ and $k\ge0$,
\[\|\partial_\tau^kg(\tau)\|_\Hc\,\le\,\|S^kg^\circ\|_\Hc.\]
\item \emph{Approximate isometricity up to $O(\tau^\infty)$:} for all $k\ge0$,
\[-\tfrac{1}{k!}\tau^k\sum_{j=0}^k\binom{k}{j}\|S^{j}g^\circ\|_\Hc^2~\le~\|g(\tau)\|_\Hc^2-\|g^\circ\|_\Hc^2~\le~0.\]
\end{enumerate}
\end{prop}

Finally, we establish the following RAGE theorem, {which shows that the effective dynamics describes the weak relaxation of the system to equilibrium for $\tau\gg1$ (i.e., $t\gg N^{1/2}$) in the absence of periodic solutions.}
The proof is postponed to Section~\ref{sec:RAGE}.

\begin{prop}[RAGE theorem]\label{th:RAGE}
Let $\{\lambda_k\}_k$ be the set of real eigenvalues of $S^*$ (if any). There exists a family of positive contractions $\{P_k\}_k$ on $\Hc$ such that for all~$k$ the image $\operatorname{ran}(P_k)$ is a subset of the eigenspace of $S^*$ associated with $\lambda_k$, such that the orthogonality condition $\operatorname{ran}(P_k)\bot\operatorname{ran}(P_l)$ holds for $|\lambda_k|\ne|\lambda_l|$, and such that the following RAGE theorem holds: for all $g^\circ\in\Cc$, the unique strong solution $g$ of the effective equation~\eqref{eq:limit-eqn-wellposed} as given by Proposition~\ref{prop:unique} can be decomposed as
\[g(\tau)\,=\,\sum_ke^{i\tau\lambda_k}P_kg^\circ+R(\tau),\]
where the series converges in the weak operator topology and where the remainder $R$ satisfies for all $h\in\Hc$,
\[\lim_{T\uparrow\infty}\frac1T\int_0^T|\langle h,R(\tau)\rangle_\Hc|^2d\tau\,=\,0.\]
{In particular, if the only real eigenvalue of $S^*$ is $0$, then this implies $g(\tau)\cvf P_0g^\circ$ weakly in $\Hc$ as $\tau\uparrow\infty$ in Ces\`aro mean.}
\end{prop}

\subsection{Proof of Proposition~\ref{prop:unif-wave}}\label{sec:proof-unif-wave}
Taking time derivatives in the Liouville equation~\eqref{eq:Liouville}, and using that $\Ld^2$ norms are conserved, we find for all $k\ge0$,
\begin{equation}\label{eq:conserv-dt-k}
\|\partial_t^kf_N\|_{\Ld^2((\T^d)^N)}\,=\,\|\partial_t^kf_N|_{t=0}\|_{\Ld^2((\T^d)^N)}.
\end{equation}
From the Liouville equation~\eqref{eq:Liouville} and the choice~\eqref{eq:initial-unif} of initial data, we find
\begin{equation}\label{eq:dtkfN-0}
\partial_t^kf_N|_{t=0}\,=\,\bigg[-\tfrac1N\sum_{i,j=1}^NK(x_i-x_j)\cdot\nabla_i\bigg]^kf^{\circ}(x_1).
\end{equation}
By integration by parts, the norm of this quantity can be written as
\begin{multline*}
\|\partial_t^kf_N|_{t=0}\|_{\Ld^2((\T^d)^N)}^2
\,=\,(-1)^{k}N^{-2k}\sum_{i_1,j_1,\ldots,i_{2k},j_{2k}=1}^N\int_{(\T^d)^N} f^{\circ}(x_1)\,K_{\alpha_1}(x_{i_1}-x_{j_1})(\nabla_{i_1})_{\alpha_1}\\
\ldots K_{\alpha_{2k}}(x_{i_{2k}}-x_{j_{2k}})(\nabla_{i_{2k}})_{\alpha_{2k}}f^{\circ}(x_1)\,dx_1\ldots dx_N.
\end{multline*}
Taking advantage of cancellations when computing the derivatives, and recalling that \mbox{$\Div(K)=0$}, we find that the only non-vanishing contributions in this sum are those such that for all $1\le l\le 2k$ the index $i_l$ belongs both to $\{1,j_1,\ldots,j_{l-1}\}$ and to $\{j_{l+1},\ldots,j_{2k},1\}$.
Moreover, recalling that $\int_{\T^d}K=0$, we find that each value in \mbox{$\{i_1,j_2,\ldots,i_{2k},j_{2k}\}\setminus\{1\}$} must be taken by at least two different indices. From these two restrictions in the above summation, we can immediately conclude that
\begin{equation*}
\|\partial_t^kf_N|_{t=0}\|_{\Ld^2((\T^d)^N)}\,\lesssim_{k,f^\circ}\,N^{-\frac k2}.
\end{equation*}
Combining this with~\eqref{eq:conserv-dt-k}, we deduce for all $k\ge0$,
\begin{equation}\label{eq:estim-dt-k}
\|\partial_t^kf_N\|_{\Ld^2((\T^d)^N)}\,\lesssim_{k,f^\circ}\,N^{-\frac k2}.
\end{equation}
A second-order Taylor expansion then yields in particular,
\begin{multline*}
{\big\|\partial_\tau^2(f_N(N^{\frac12}\tau))-N(\partial_t^2f_N|_{t=0})-N^\frac32\tau(\partial_t^3f_N|_{t=0})\big\|_{\Ld^2((\T^d)^N)}}\\
\,=\,\Big\|N\int_0^{N^\frac12\tau}(N^\frac12\tau-s)\,(\partial_t^4f_N)(s)\,ds\Big\|_{\Ld^2((\T^d)^N)}
\,\lesssim_{f^\circ}\,\tau^2.
\end{multline*}
Averaging over variables $x_2,\ldots,x_N$, this entails
\begin{equation}\label{eq:Taylor-dt2fN}
\big\|\partial_\tau^2(f_N^1(N^{\frac12}\tau))-N(\partial_t^2f_N^1|_{t=0})-N^\frac32\tau(\partial_t^3f_N^1|_{t=0})\big\|_{\Ld^2((\T^d)^N)}\,\lesssim_{f^\circ}\,\tau^2.
\end{equation}
Starting from~\eqref{eq:dtkfN-0}, direct computations yield
\begin{eqnarray*}
\partial_t^2f_N^1|_{t=0}
&=&\int_{(\T^d)^{N-1}}\bigg[-\tfrac1N\sum_{i,j=1}^NK(x_i-x_j)\cdot\nabla_i\bigg]^2f^{\circ}(x_1)\,dx_2\ldots dx_N\\
&=&\tfrac1{N^2}\sum_{i,j,k=1}^N\int_{(\T^d)^{N-1}}K(x_i-x_j)\cdot\nabla_iK(x_1-x_k)\cdot\nabla_1f^{\circ}(x_1)\,dx_2\ldots dx_N\\
&=&\tfrac{N-1}{N^2}\Big(\int_{\T^d}K\otimes K\Big):\nabla^2f^{\circ},
\end{eqnarray*}
and
\[\partial_t^3f_N^1|_{t=0}\,=\,0,\]
which concludes the proof of Proposition~\ref{prop:unif-wave}.\qed

\begin{rem}
The above proof can be pursued to capture higher-order corrections in a similar perturbative way up to $O(\tau^\infty+\frac1N)$. For instance, the next-order Taylor expansion yields, instead of~\eqref{eq:Taylor-dt2fN},
\begin{multline*}
\Big\|\partial_\tau^2(f_N^1(N^{\frac12}\tau))-N(\partial_t^2f_N|_{t=0})-\tau N^\frac32 (\partial_t^3f_N^1|_{t=0})\\
-\tfrac12\tau^2N^2(\partial_t^4f_N|_{t=0})-\tfrac16\tau^3N^\frac52(\partial_t^5f_N|_{t=0})\Big\|_{\Ld^2(\T^d)}\,\lesssim\,\tau^4.
\end{multline*}
Further computing $\partial_t^4f_N|_{t=0}$, starting from~\eqref{eq:dtkfN-0}, and noting that $\partial_t^5f_N|_{t=0}=0$, we are led to the following next-order version of Proposition~\ref{prop:unif-wave},
\[\Big\|\partial_\tau^2(f_N^1(N^{\frac12}\tau))-\Div(A\nabla f^{\circ})-\tfrac32\tau^2(\Div A\nabla)^2f^{\circ}
-\tfrac12\tau^2\Div(B\nabla f^{\circ})\Big\|_{\Ld^2(\T^d)}\,\lesssim\,\tau^4+\tfrac1N,\]
where the matrix $B$ is explicitly given by
\begin{align*}
B_{\alpha\beta}&=\Big(\int_{\T^d}(\nabla_\gamma K_\delta) (\nabla_\delta K_\gamma) (K_\alpha\ast K_\beta)\Big)
-2\Big(\int_{\T^d} (\nabla_\delta K_\alpha)(\nabla_{\gamma}K_\beta)\Big)\Big(\int_{\T^d} K_\delta K_\gamma\Big)\\
&\qquad-2\Big(\int_{\T^d}(\nabla_\delta K_\alpha)(\nabla_\gamma K_\beta)(K_\delta\ast K_\gamma)\Big)
-2\Big(\int_{\T^d}(\nabla_\delta K_\alpha) (\nabla_\gamma K_\delta) (K_\gamma\ast K_\beta)\Big),
\end{align*}
where we implicitly sum over repeated indices $\delta,\gamma$.
\end{rem}

\subsection{Rigorous BBGKY analysis}
We start by reformulating the BBGKY hierarchy of Lemma~\ref{lem:eqn-gNm-21} in the present uniform setting, noting that it gets drastically simplified and that in particular linearized mean-field operators vanish.

\begingroup\allowdisplaybreaks
\begin{lem}[BBGKY hierarchy for cumulants]\label{lem:eqn-gNm}
For all $1\le m\le N$,
\begin{equation}\label{eq:eqn-gNm}
\partial_tg_N^m
\,=\,iS_N^{m,+}g_N^{m+1}+\frac1N\Big(iS_N^{m,\circ}g_N^m+iS_N^{m,-}g_N^{m-1}\Big),
\end{equation}
where we have set for notational  convenience $g^r_N=0$ for $r<1$ or $r>N$, and where we have defined the operators
\begin{eqnarray*}
iS_N^{m,+}g_N^{m+1}&:=&-\tfrac{N-m}N\sum_{i=1}^m\int_{\T^d} K(x_i-x_*)\cdot\nabla_ig_N^{m+1}(x_{[m]},x_*)\,dx_*,\\
iS_N^{m,\circ}g_N^m&:=&-\sum_{i,j=1}^mK(x_i-x_j)\cdot\nabla_ig_N^m\\
&&+\sum_{i=1}^m\sum_{2\le j\le m\atop i\ne j}\int_{\T^d}K(x_i-x_*)\cdot\nabla_ig_N^{m}(x_{[m]\setminus\{j\}},x_*)\,dx_*,\\
iS_N^{m,-}g_N^{m-1}&:=&-\sum_{i=1}^m\sum_{j=2}^mK(x_i-x_j)\cdot\nabla_ig_N^{m-1}(x_{[m]\setminus\{j\}}),
\end{eqnarray*}
where we recall the short-hand notation $[m]=\{1,\ldots,m\}$.
\end{lem}

\begin{proof}
Upon partial integration, the Liouville equation~\eqref{eq:Liouville} yields the following BBGKY hierarchy of equations for the marginals,
\begin{multline}\label{eq:BBGKY}
\partial_tf_N^m+\tfrac1N\sum_{i,j=1}^mK(x_i-x_j)\cdot\nabla_if_N^m\\
+\tfrac{N-m}N\sum_{i=1}^m\int_{\T^d} K(x_i-x_*)\cdot\nabla_if_N^{m+1}(x_1,\ldots,x_m,x_*)\,dx_*=0.
\end{multline}
By definition~\eqref{eq:correl-def} of correlation functions, we deduce
\begin{multline}\label{eq:eqn-gNm-0}
\partial_tg_N^m
\,=\,-\tfrac1N\sum_{j=2}^mK(x_1-x_j)\cdot(\nabla_1-\nabla_j)\sum_{n=1}^m(-1)^{m-n}\sum_{\sigma\in P_{n-1}^{m-1}}\mathds1_{j\in\sigma}f_N^n(x_1,x_{\sigma})\\
-\tfrac1N\sum_{i,j=2}^mK(x_i-x_j)\cdot\nabla_i\sum_{n=1}^m(-1)^{m-n}\sum_{\sigma\in P_{n-1}^{m-1}}\mathds1_{i,j\in\sigma}f_N^n(x_1,x_\sigma)\\
-\sum_{n=1}^m(-1)^{m-n}\tfrac{N-n}N\sum_{\sigma\in P_{n-1}^{m-1}}\int_{\T^d} K(x_1-x_*)\cdot\nabla_1f_N^{n+1}(x_1,x_\sigma,x_*)\,dx_*\\
-\sum_{i=2}^m\sum_{n=1}^m(-1)^{m-n}\tfrac{N-n}N\sum_{\sigma\in P_{n-1}^{m-1}}\mathds1_{i\in\sigma}\int_{\T^d} K(x_i-x_*)\cdot\nabla_if_N^{n+1}(x_1,x_\sigma,x_*)\,dx_*.
\end{multline}
Replacing marginals in terms of cumulants, cf.~\eqref{eq:correl-def-clust}, and arguing as in~\eqref{eq:marg-cumul-combi}, we get for the first right-hand side term, for all $j\in[m]\setminus\{1\}$,
\begin{equation*}
\sum_{n=1}^m(-1)^{m-n}\sum_{\sigma\in P_{n-1}^{m-1}}\mathds1_{j\in\sigma}f_N^n(x_1,x_{\sigma})
\,=\,g_N^m(x_{[m]})+g_N^{m-1}(x_{[m]\setminus\{j\}}).
\end{equation*}
Similarly, for the second right-hand side term in~\eqref{eq:eqn-gNm-0}, we find for all $i,j\in[m]\setminus\{1\}$,
\begin{multline*}
\sum_{n=1}^m(-1)^{m-n}\sum_{\sigma\in P_{n-1}^{m-1}}\mathds1_{i,j\in\sigma}f_N^n(x_1,x_\sigma)\\
\,=\,g_N^m(x_{[m]})
+g_N^{m-1}(x_{[m]\setminus\{i\}})+g_N^{m-1}(x_{[m]\setminus\{j\}})
+g_N^{m-2}(x_{[m]\setminus\{i,j\}}).
\end{multline*}
For the third right-hand side term in~\eqref{eq:eqn-gNm-0}, arguing as in~\eqref{eq:integr-marg-cum-ag}, replacing again marginals by cumulants and using that $\int_{\T^d} K=0$, we find that
\begin{eqnarray}
\lefteqn{\sum_{n=1}^m(-1)^{m-n}\tfrac{N-n}N\sum_{\sigma\in P_{n-1}^{m-1}}\int_{\T^d}K(x_1-x_*)\cdot\nabla_1f_N^{n+1}(x_1,x_\sigma,x_*)\,dx_*}\nonumber\\
&=&\sum_{n=1}^m(-1)^{m-n}\tfrac{N-n}N\sum_{\sigma\in P_{n-1}^{m-1}}\sum_{r=1}^n\sum_{\tau\in P_{r-1}^{\sigma}}\int_{\T^d}K(x_1-x_*)\cdot\nabla_1g_N^{r+1}(x_1,x_\tau,x_*)\,dx_*\nonumber\\
&=&\tfrac{N-m}N\int_{\T^d}K(x_1-x_*)\cdot\nabla_1g_N^{m+1}(x_{[m]},x_*)\,dx_*\nonumber\\
&&\hspace{3cm}-\tfrac{1}N\sum_{j=2}^{m}\int_{\T^d}K(x_1-x_*)\cdot\nabla_1g_N^{m}(x_{[m]\setminus\{j\}},x_*)\,dx_*.
\end{eqnarray}
For the last right-hand side term in~\eqref{eq:eqn-gNm-0}, we similarly get
\begin{multline*}
\sum_{i=2}^m\sum_{n=1}^m(-1)^{m-n}\tfrac{N-n}N\sum_{\sigma\in P_{n-1}^{m-1}}\mathds1_{i\in\sigma}\int_{\T^d} K(x_i-x_*)\cdot\nabla_if_N^{n+1}(x_1,x_\sigma,x_*)\,dx_*\\
\,=\,\tfrac{N-m}N\sum_{i=2}^m\int_{\T^d} K(x_i-x_*)\cdot\nabla_ig_N^{m+1}(x_{[m]},x_*)\,dx_*\\
-\tfrac{1}N\sum_{2\le i,j\le m}^{\ne}\int_{\T^d} K(x_i-x_*)\cdot\nabla_ig_N^{m}(x_{[m]\setminus\{j\}},x_*)\,dx_*.
\end{multline*}
Combining the above identities, the conclusion follows.
\end{proof}
\endgroup

Next, we prove uniform-in-time propagation-of-chaos estimates for the particle system in form of a priori estimates on correlation functions.
This is deduced from a straightforward symmetry argument inspired by the work of Bodineau, Gallagher, and Saint-Raymond~\cite{BGSR-16}. In contrast with the more delicate case of Lemma~\ref{lem:est-cum-2}, no large deviation theory is needed here.
Note that in the present situation the $N$-scaling $g_N^{m+1}=O(N^{-m/2})$ is actually optimal on long timescales, cf.~Theorem~\ref{th:unif-wave-2}, whereas the scaling $g_N^{m+1}=O(N^{-m})$ that can be heuristically guessed from the BBGKY hierarchy~\eqref{eq:eqn-gNm} is only valid on short times~\mbox{$t=O(1)$}. For later purposes, we also include estimates on time derivatives of correlation functions.

\begin{lem}[A priori cumulant estimates]\label{lem:est-cum}
For all $0\le m<N$ and $k\ge0$, we have
\[\|\partial_t^kg_N^{m+1}\|_{\Ld^2((\T^d)^{m+1})}\,\lesssim_{m,k,f^\circ}\,N^{-\frac{m+k}2}.\]
\end{lem}

\begin{proof}
Recall that correlation functions satisfy $\int_{\T^d}g_N^m(x_{[m]})\,dx_j=0$ for \mbox{$2\le j\le m$.} 
Computing the $\Ld^2$ norm of the $N$-point density $f_N$ and inserting the cluster expansion~\eqref{eq:correl-def-clust} in terms of correlation functions, we then get
\begin{equation}\label{eq:norm-cluster}
\int_{(\T^d)^N}|f_N|^2\,=\,\sum_{m=1}^N\binom{N-1}{m-1}\int_{(\T^d)^{m}}|g_N^{m}|^2,
\end{equation}
hence in particular, for all $0\le m< N$,
\[\|g_N^{m+1}\|_{\Ld^2((\T^d)^{m+1})}\,\lesssim_m\,N^{-\frac m2}\|f_N\|_{\Ld^2((\T^d)^N)}.\]
By linearity, we similarly find for all $k\ge0$,
\begin{equation}\label{eq:cluster-L2fN}
\int_{(\T^d)^N}|\partial_t^kf_N|^2\,=\,\sum_{m=1}^N\binom{N-1}{m-1}\int_{(\T^d)^{m}}|\partial_t^kg_N^{m}|^2,
\end{equation}
hence
\[\|\partial_t^kg_N^{m+1}\|_{\Ld^2((\T^d)^{m+1})}\,\lesssim_m\,N^{-\frac m2}\|\partial_t^kf_N\|_{\Ld^2((\T^d)^N)}.\]
Combining this with the energy estimate~\eqref{eq:estim-dt-k} for the Liouville equation, the conclusion follows.
\end{proof}

\subsection{Proof of Theorem~\ref{th:unif-wave-2}}
For all $m\ge1$, we deduce from~Lemma~\ref{lem:est-cum} that the rescaled correlation function $\bar g_N^m$ defined in~\eqref{eq:rescaled-correl-1} is bounded in $W^{k,\infty}(\R^+;\Ld^2((\T^d)^m))$, for any $k\ge0$, as~$N\uparrow\infty$.
By weak-* compactness, up to a subsequence, we deduce that there exists a limiting family $\bar g=\{\bar g^m\}_{m\ge1}$ such that for all $m,k$ we have
\[\bar g^m_N\cvf\bar g^m,\qquad\text{weakly-* in $W^{k,\infty}(\R^+;\Ld^2((\T^d)^m))$}.\]
As the choice~\eqref{eq:initial-unif} of initial data yields $\bar g_N^1|_{\tau=0}=f^{\circ}$ and $\bar g_N^m|_{\tau=0}=0$ for $m>1$, the convergence implies in particular
\begin{equation}\label{eq:initial-barg}
\bar g^1|_{\tau=0}=f^{\circ}\qquad\text{and}\qquad\bar g^m|_{\tau=0}=0~~\text{for $m>1$.}
\end{equation}
Next, we note that, after rescaling, the BBGKY hierarchy for cumulants in Lemma~\ref{lem:eqn-gNm} takes the form
\[\partial_\tau\bar g_N^m=iS_N^{m,+}\bar g_N^{m+1}+iS_N^{m,-}\bar g_N^{m-1}+N^{-\frac12}iS_N^{m,\circ}\bar g_N^m,\qquad1\le m\le N.\]
Passing to the limit in the weak formulation of those equations, we deduce that the limit~$\bar g$ satisfies the following hierarchy of equations in the weak sense,
\begin{equation}\label{eq:eqn-barg}
\partial_\tau\bar g^m=iS^{m,+}\bar g^{m+1}+iS^{m,-}\bar g^{m-1},\qquad m\ge1,
\end{equation}
where the limiting operators $iS^{m,+},iS^{m,-}$ are defined in the statement.
Moreover, reformulating~\eqref{eq:cluster-L2fN} in terms of rescaled correlations, truncating the sum, and combining with the energy estimate~\eqref{eq:estim-dt-k}, we find for any fixed~$M,k$,
\[\sum_{m=1}^{M\wedge N}N^{1-m}\binom{N-1}{m-1}\|\partial_\tau^k\bar g_N^m\|_{\Ld^2((\T^d)^m)}^2\,\le \,N^{k}\|\partial_t^kf_N\|_{\Ld^2((\T^d)^N)}^2\,\lesssim_{k,f^\circ}\,1.\]
Passing to the limit $N\uparrow\infty$, with $N^{1-m}\binom{N-1}{m-1}\to\tfrac1{(m-1)!}$, we deduce by weak lower semicontinuity,
\[\sum_{m=1}^{M}\tfrac1{(m-1)!}\|\partial_\tau^k\bar g^m\|_{\Ld^2((\T^d)^m)}^2\,\lesssim_{k,f^\circ}\,1,\]
and thus, letting $M\uparrow\infty$, we infer for all $k\ge1$,
\begin{equation}\label{eq:bnd-barg}
\sum_{m=1}^{\infty}\tfrac1{(m-1)!}\|\partial_\tau^k\bar g^m\|_{\Ld^2((\T^d)^m)}^2\,\lesssim_{k,f^\circ}\,1.
\end{equation}
Finally, in order to get rid of the extraction of a subsequence in the above argument, it remains to check that there is at most one $\bar g=\{\bar g^m\}_m$ that satisfies the limit hierarchy~\eqref{eq:eqn-barg} in the weak sense and such that the a priori estimates~\eqref{eq:bnd-barg} hold.
As explained, this is equivalent to proving that there is at most one strong solution $\bar g\in C^\infty_b(\R^+;\Hc)$ of the following equation,
\[\partial_\tau\bar g=iS^*\bar g,\qquad\bar g|_{\tau=0}=\bar g^\circ,\]
where the Hilbert space $\Hc$ and the densely-defined symmetric operator $S$ are defined in~\eqref{eq:HilbertHc}--\eqref{eq:def-S-oplim}, and where the initial condition $\bar g^\circ$ is given by~\eqref{eq:def-barg0}.
Noting that~$\bar g^\circ$ belongs to the core $\Cc$ of $S$ and that the latter satisfies $S\Cc\subset\Cc$, the desired uniqueness is a consequence of Lemma~\ref{lem:unique-S} below.
This ends the proof of Theorem~\ref{th:unif-wave-2}.
\qed

\subsection{Proof of Proposition~\ref{prop:unique}}\label{sec:unique-pr}
We start by proving the uniqueness of a solution of the limit hierarchy in $C^2_b(\R^+;\Hc)$. We state it separately in form of the following general abstract result.

\begin{lem}\label{lem:unique-S}
Let $\Hc$ be a Hilbert space, let $S$ be a densely-defined symmetric operator, defined on a dense subset $\Cc\subset\Hc$, and assume that $S\Cc\subset\Cc$. Then, for all $g^\circ\in\Hc$, there is at most one solution $g\in C^2_b(\R^+;\Hc)$ of the following equation,
\begin{equation}\label{eq:g-S*-gento}
\partial_\tau g=iS^*g,\qquad g|_{\tau=0}=g^\circ.
\end{equation}
\end{lem}
\begin{proof}
As $S\Cc\subset\Cc$, we note that the squared operator $S^2$ is well-defined, symmetric, and positive on $\Cc$. By Friedrichs' theorem, it admits a canonical self-adjoint extension $\Lc_0$. Given a solution $g\in C^2(\R^+;\Hc)$ of equation~\eqref{eq:g-S*-gento}, we may then compute for all $h\in\Cc$,
\[\langle h,\partial_\tau^2g(\tau)\rangle_\Hc=-\langle S^2h,g(\tau)\rangle_\Hc=-\langle \Lc_0 h,g(\tau)\rangle_\Hc.\]
As $g$ belongs to $C^2(\R^+;\Hc)$, the left-hand side is bounded by $\|\partial_\tau^2g(\tau)\|_\Hc\|h\|_\Hc$, hence so is the right-hand side, which entails that $g(\tau)$ belongs to the domain of $\Lc_0$ for all $\tau$. This implies that $g$ is a strong solution the following equation,
\[\partial_\tau^2g+\Lc_0 g=0,\qquad g|_{\tau=0}=g^\circ,\qquad \partial_\tau g|_{\tau=0}=iS^*g^\circ.\]
By self-adjointness and non-negativity of $\Lc_0$, the strong solution of this equation is unique.
(Note that equation~\eqref{eq:g-S*-gento} and the condition $g\in C^2(\R^+;\Hc)$ ensure $g^\circ\in\Dc(S^*)$, so $iS^*g^\circ$ is indeed well-defined.)
\end{proof}

\begin{rem}
The proof of Lemma~\ref{lem:unique-S} is easily adapted to establish the following result: given a Hilbert space $\Hc$ and a densely-defined symmetric operator $S$, defined on a dense subset $\Cc\subset\Hc$, if~$S$ admits a self-adjoint extension $S_0$, then for all $g^\circ\in\Dc(S_0)$ there is a unique strong solution $g\in C^1_b(\R^+;\Hc)$ of the equation
\begin{equation*}
\partial_\tau g=iS^*g,\qquad g|_{\tau=0}=g^\circ,
\end{equation*}
and it coincides with the unitary group $g(\tau)=e^{i\tau S_0}g^\circ$.
It is however not clear to us whether this result applies to our situation as we do not know how to prove the existence of a self-adjoint extension for the operator $S$ defined in~\eqref{eq:def-S-oplim}.
\end{rem}

With the above general uniqueness result at hand, we can now conclude the proof of Proposition~\ref{prop:unique}, establishing the well-posedness, contraction, and approximate isometricity of solutions of the limit hierarchy.

\begin{proof}[Proof of Proposition~\ref{prop:unique}]
Let the Hilbert space $\Hc$, the densely-defined symmetric operator~$S$, and its core~$\Cc\subset\Hc$ be defined as in~\eqref{eq:HilbertHc}--\eqref{eq:core-def}. We split the proof into two steps.

\medskip
\step1 Proof that there exists a unique contraction-valued strongly-continuous map \mbox{$U:\R^+\to L(\Hc)$}, which might not be a semigroup, such that for all $g^\circ\in\Cc$ the evolution $g(\tau):=U(\tau)g^\circ$ is the unique strong solution in $C^\infty_b(\R^+;\Hc)$ of the equation
\begin{equation}\label{eq:limS*eqnpr}
\partial_\tau g=iS^*g,\qquad g|_{\tau=0}=g^\circ.
\end{equation}
Moreover, we shall show that it satisfies the stability property
\begin{equation}\label{eq:g-dtCinfty-apriori}
\|\partial_\tau^kU(\tau)g^\circ\|_\Hc\,\le\,\|S^kg^\circ\|_\Hc,\qquad\text{for all $k\ge0$ and $g^\circ\in\Cc$}.
\end{equation}
To prove this result, we start by defining self-adjoint truncations of $S$:
for all~$N\ge1$, we consider the closed subspace $\Hc_{\le N}:=\bigoplus_{1\le m\le N}\Hc^m\subset\Hc$, we let $\pi_{\le N}:\Hc\to\Hc_{\le N}$ be the associated orthogonal projection, and we define the truncated operator $S_N:=\pi_{\le N}S\pi_{\le N}$ on $\Cc$.
Note that this truncation could be replaced by the original hierarchy~\eqref{eq:eqn-gNm} for fixed~$N$: it does not change the argument, but the present truncation is simpler to handle.
By definition of $S$, using the symmetry relations~\eqref{eq:sym-S+S-}, we easily check that this truncated operator~$S_N$ is essentially self-adjoint on~$\Cc$ (self-adjointness poses no difficulty here thanks to truncations).
We may then consider the unitary semigroup $U_N:\R^+\to L(\Hc)$ given by
\[U_N(\tau)\,:=\,e^{i\tau S_N}.\]
Up to extraction of a subsequence, as $N\uparrow\infty$,
the semigroup $U_N$ converges pointwise in the weak operator topology to some strongly-continuous map $U:\R^+\to L(\Hc)$ with $U(0)=\Id$ and $\|U(\tau)\|\le1$ for all $\tau\ge0$. Note that this weak limit map $U$ might no longer be a semigroup nor take its values among unitary operators.

Given $g^\circ\in\Cc$, let us examine the properties of the limit evolution $\tau\mapsto U(\tau)g^\circ$.
For all~$N\ge1$, as $i S_N$ is the generator of $U_N$ and as $S_N\Cc\subset\Cc$, the flow $\tau\mapsto U_N(\tau)g^\circ$ belongs to $C^\infty_b(\R^+;\Hc)$ and satisfies
\begin{equation}\label{eq:stability-SNUN}
\|\partial_\tau^kU_N(\tau)g^\circ\|_\Hc\,=\,
\|S_N^kg^\circ\|_\Hc,\qquad\text{for all $k\ge0$,}
\end{equation}
and
\[\langle h,\partial_\tau U_N(\tau)g^\circ\rangle\,=\,\langle S_Nh,iU_N(\tau)g^\circ\rangle,\qquad\text{for all $h\in\Cc$}.\]
Note that the tridiagonal structure of $S$ yields $S_N^kh=S^kh$ for all $h\in\pi_{\le N-k}\Cc$ and $k\ge1$, and therefore $S_N^kh\to S^kh$ strongly as $N\uparrow\infty$ for all $h\in\Cc$ and $k\ge1$.
We may then pass to the limit along the extracted subsequence in the above properties of $U_N(\tau)g^\circ$: we deduce that the limit evolution $\tau\mapsto U(\tau)g^\circ$ also belongs to $C^\infty_b(\R^+;\Hc)$ and satisfies
\begin{equation}\label{eq:stability-SU}
\|\partial_\tau^kU(\tau)g^\circ\|_\Hc\,\le\,\|S^kg^\circ\|_\Hc,\qquad\text{for all $k\ge0$},
\end{equation}
and
\[\langle h,\partial_\tau U(\tau)g^\circ\rangle\,=\,\langle Sh,iU(\tau)g^\circ\rangle,\qquad\text{for all $h\in\Cc$}.\]
The former is the desired contraction property~\eqref{eq:g-dtCinfty-apriori} and the latter precisely means that the limit evolution $g(\tau):=U(\tau)g^\circ$ satisfies equation~\eqref{eq:limS*eqnpr} in the strong sense.

Now, by Lemma~\ref{lem:unique-S}, the solution of equation~\eqref{eq:limS*eqnpr} is necessarily unique in $C^2_b(\R^+;\Hc)$.
In particular, the limit evolution $\tau\mapsto U(\tau)g^\circ$ is uniquely determined for all $g^\circ\in\Cc$. As~$U$ is contraction-valued and as~$\Cc$ is dense in $\Hc$, this entails that $U$ is itself uniquely determined as a contraction-valued strongly-continuous map $\R^+\to L(\Hc)$.

\medskip
\step2 Approximate unitarity: proof that for all $k\ge0$ and $g^\circ\in\Cc$,
\begin{equation}\label{eq:approx-unit-pr}
-\tfrac{1}{k!}\tau^k\sum_{j=0}^k\binom{k}{j}\|S^{j}g^\circ\|_\Hc^2\,\le\,\|U(\tau)g^\circ\|_\Hc^2-\|g^\circ\|_\Hc^2\,\le\,0.
\end{equation}
The upper bound follows from~\eqref{eq:stability-SU} and it remains to prove the lower bound.
Let $g^\circ\in\Cc$ be fixed.
For~$M,N\ge1$, consider
\[E_{M,N}(\tau)\,:=\,\|\pi_{\le M}U_N(\tau)g^\circ\|^2_\Hc\,=\,\sum_{m=1}^M\|U_N(\tau)g^\circ\|_{\Hc^m}^2,\]
and appeal to Taylor's expansion
\begin{equation*}
\Big|E_{M,N}(\tau)-\sum_{j=0}^{k-1}\tfrac1{j!}\tau^j\partial_\tau^jE_{M,N}(0)\Big|\,\le\,\tfrac1{k!}\tau^k\|\partial_\tau^kE_{M,N}\|_{\Ld^{\infty}(\R^+)}.
\end{equation*}
Using~\eqref{eq:stability-SNUN}, time derivatives of $E_{M,N}$ can be estimated as follows, for all $k\ge1$,
\begin{eqnarray*}
|\partial_\tau^k E_{M,N}(\tau)|
&\le&\sum_{j=0}^k\binom{k}{j}\|\pi_{\le M}\partial_\tau^jU_N(\tau)g^\circ\|_\Hc\|\pi_{\le M}\partial_\tau^{k-j}U_N(\tau)g^\circ\|_\Hc\\
&\le&\sum_{j=0}^k\binom{k}{j}\|S_N^jg^\circ\|_\Hc\|S_N^{k-j}g^\circ\|_\Hc.
\end{eqnarray*}
By the symmetry in the sum and the trivial bound $2ab\leq a^2+b^2$, the above becomes
\begin{equation}\label{eq:EM-Taylor}
\Big|E_{M,N}(\tau)-\sum_{j=0}^{k-1}\tfrac1{j!}\tau^j\partial_\tau^jE_{M,N}(0)\Big|\,\le\,\tfrac1{k!}\tau^k\sum_{j=0}^k\binom{k}{j}\|S_N^jg^\circ\|_\Hc^2.
\end{equation}
Now for $M< N$, recalling $S_N=\pi_{\le N}S\pi_{\le N}$ and the definition of $S$, we can compute the first time derivative of $E_{M,N}$ as follows: setting $g_N(\tau):=U_N(\tau)g^\circ$ and~$g_N=\{g_N^m\}_{m\ge1}$,
\begin{eqnarray*}
\partial_\tau E_{M,N}(\tau)&=&2\Re\sum_{m=1}^M\langle g_N^m(\tau),\partial_\tau g_N^m(\tau)\rangle_{\Hc^m}\\
&=&2\Re\sum_{m=1}^M \big\langle g^m_N(\tau),iS^{m,+}\bar g^{m+1}_N(\tau)+iS^{m,-} g^{m-1}_N(\tau)\big\rangle_{\Hc^m},
\end{eqnarray*}
and thus, using $(S^{m,-})^*=S^{m-1,+}$ and recognizing a telescoping sum,
\begin{eqnarray*}
\partial_\tau E_{M,N}(\tau)
&=&2\Re\sum_{m=1}^M \Big(\langle g^m_N(\tau),iS^{m,+} g^{m+1}_N(\tau)\rangle_{\Hc_m}-\langle iS^{m-1,+} g^m_N(\tau),g^{m-1}_N(\tau)\rangle_{\Hc^{m-1}}\Big)\\
&=&2\Re\langle g_N^M(\tau),iS^{M,+} g_N^{M+1}(\tau)\rangle_{\Hc_M}.
\end{eqnarray*}
Evaluating at $\tau=0$, this means
\begin{equation*}
\partial_\tau E_{M,N}(0)
\,=\,2\Re\langle (g^\circ)^M,iS^{M,+}(g^\circ)^{M+1}\rangle_{\Hc_M}.
\end{equation*}
For $g^\circ\in\Cc$, there is $M_0(g^\circ)<\infty$ such that $(g^\circ)^M=0$ for all $M>M_0(g^\circ)$. Hence, we deduce $\partial_\tau E_{M,N}(0)=0$ for $M\ge M_0(g^\circ)$.
Taking additional time derivatives of the above expression and using the tridiagonal structure of $S$, we find $\partial_\tau^jE_{M,N}(0)=0$ for all $1\le j\le M+1-M_0(g^\circ)$. By definition of $E_{M,N}$, the estimate~\eqref{eq:EM-Taylor} then takes on the following guise: for all $M<N$ and $k\le M+2-M_0(g^\circ)$,
\begin{equation*}
\Big|\|\pi_{\le M}U_N(\tau)g^\circ\|^2_\Hc-\|\pi_{\le M}g^\circ\|^2_\Hc\Big|\,\le\,\tfrac1{k!}\tau^k\sum_{j=0}^k\binom{k}{j}\|S_N^jg^\circ\|_\Hc^2.
\end{equation*}
Letting $N\uparrow\infty$ and then $M\uparrow\infty$, using that $U_N(\tau)g^\circ$ converges weakly to $U(\tau)g^\circ$ in $\Hc$, and
recalling that $S_N^jg^\circ$ converges strongly to $S^jg^\circ$ in $\Hc$ for all $j\ge0$,
the claimed lower bound~\eqref{eq:approx-unit-pr} follows.
\end{proof}

\subsection{Proof of Proposition~\ref{th:RAGE}}\label{sec:RAGE}
We show that the RAGE theorem still applies to the weak limit of a sequence of unitary groups, although the limit might no longer be a semigroup nor take its values among unitary operators. Note however that we loose (part of) the usual orthogonality property for periodic solutions. Proposition~\ref{th:RAGE} is a direct consequence of the following general abstract result.

\begin{lem}\label{lem:RAGE-abs}
Let $(U_N)_N$ be a sequence of unitary semigroups $U_N:\R^+\to L(\Hc)$ on a Hilbert space $\Hc$, and 
assume that their generators~$(S_N)_N$ are essentially self-adjoint on a common (dense) core $\Cc\subset\Hc$. As $N\uparrow\infty$, assume that $S_N$ converges in the strong operator topology to some operator $S$ on $\Cc$ (necessarily symmetric on $\Cc$, but possibly not essentially self-adjoint), and that~$U_N$ converges pointwise in the weak operator topology to a map $U:\R^+\to L(\Hc)$ (necessarily strongly-continuous and contraction-valued, but possibly not unitary and not a semigroup).
Then the following RAGE theorem holds for~$U$: denoting by~$\{\lambda_k\}_k$ the set of real eigenvalues of~$S^*$, there exists a family of positive contractions~$\{P_k\}_k$ on~$\Hc$ such that for all $k$ the image $\operatorname{ran}(P_k)$ is a subset of the eigenspace of~$S^*$ associated with~$\lambda_k$, and such that for all $g^\circ\in\Hc$ we can decompose the limit evolution as
\begin{equation*}
U(\tau)g^\circ\,=\,\sum_ke^{i\tau\lambda_k}P_k g^\circ+R(\tau)g^\circ,
\end{equation*}
where the remainder $R(\tau)g^\circ$ satisfies for all $h\in\Hc$,
\[\lim_{T\uparrow\infty}\frac1T\int_0^T|\langle h,R(\tau)g^\circ\rangle_\Hc|^2d\tau\,=\,0.\]
If in addition we have $S_N\Cc\subset\Cc$, $S\Cc\subset\Cc$, and if the squared operator~$S_N^2$ also converges in the strong operator topology to $S^2$ on $\Cc$, then we have the following partial orthogonality property: $\operatorname{ran}(P_k)\bot \operatorname{ran}(P_l)$ for $|\lambda_k|\ne|\lambda_l|$.
\end{lem}

\begin{proof}
For all $N$, we can consider the spectral measure $E_N$ of the essentially self-adjoint generator $S_N$ and represent the unitary group $U_N$ as
\[U_N(\tau)\,=\,e^{iS_N\tau}\,=\,\int_{\R}e^{i\lambda\tau} dE_N(\lambda).\]
Under the considered assumptions, the spectral measure $E_N$ converges weakly as $N\uparrow\infty$ to some positive operator-valued measure $E$ such that
\[U(\tau)\,=\,\int_\R e^{i\lambda\tau}dE(\lambda),\]
and in addition for all $g^\circ\in\Cc$ we find that the limit evolution $\tau\mapsto U(\tau)g^\circ$ belongs to $C^1_b(\R^+;\Hc)$ and is a strong solution of
\begin{equation}\label{eq:Utau-eqnflowlim}
\partial_\tau U(\tau)g^\circ\,=\,iS^*U(\tau)g^\circ\,=\,iU(\tau)Sg^\circ,\qquad  U(\tau)g^\circ|_{\tau=0}=g^\circ.
\end{equation}
Note that the limit measure $E$ is a priori not projection-valued, hence is not a spectral measure, in link with the fact that the limit operator $S$ might not be essentially self-adjoint and might not even generate a semigroup.
By Naimark's dilation theorem, see e.g.~\cite[Theorem~4.6]{Paulsen-02}, there exists an extended Hilbert space $\hat \Hc$, a bounded linear map $V:\hat\Hc\to\Hc$ with $VV^*=\Id$, and a spectral measure~$\hat E$ on~$\R$ such that $dE(\lambda)=Vd\hat E(\lambda)V^*$. In terms of the self-adjoint operator $\hat S:=\int_\R\lambda d\hat E(\lambda)$ on~$\hat\Hc$, we then get
\[U(\tau)\,=\,V e^{i\tau \hat S}V^*\,=\,V\int_\R e^{i\tau\lambda}d\hat E(\lambda)V^*.\]
We may now appeal to the standard form of the RAGE theorem for~$\hat S$, see e.g.~\cite[Section~5.4]{Cycon-Froese-Kirsch-Simon}: denoting by $\{\lambda_k\}_k\subset\R$ the set of eigenvalues of $\hat S$ and denoting by $\hat P_k:\hat\Hc\to\hat\Hc$ the orthogonal projection onto the eigenspace of $\hat S$ associated with $\lambda_k$,
we can decompose for all $g^\circ\in\Hc$,
\begin{equation}\label{eq:decomp-rage-0}
U(\tau)g^\circ\,=\,\sum_ke^{i\tau\lambda_k}V\hat P_kV^* g^\circ+R(\tau)g^\circ,
\end{equation}
where the remainder $R(\tau)g^\circ$ satisfies for all $h\in\Hc$,
\begin{equation}\label{eq:conv-Gbarg}
\lim_{T\uparrow\infty}\frac1T\int_0^T|\langle h,R(\tau)g^\circ\rangle_{\Hc}|^2d\tau\,=\,0.
\end{equation}
Note that by definition the remainder can be written as
\begin{equation}\label{eq:form-Rg-re}
R(\tau)g^\circ\,=\,V\Big(e^{i\tau\hat S}-\sum_ke^{i\tau\lambda_k}\hat P_k\Big)V^*g^\circ\,=\,Ve^{i\tau\hat S}\hat\pi_cV^*g^\circ,
\end{equation}
where $\hat\pi_c:=1-\sum_k\hat P_k$ is the orthogonal projection onto the continuous subspace of~$\hat S$.

Next, we show that for all~$k$ the value~$\lambda_k$ is actually also an eigenvalue of the adjoint~$S^*$ on~$\Hc$ and that the image of $P_k:=V\hat P_kV^*$ is a subset of the associated eigenspace in~$\Hc$.
For that purpose, applying the operator $\partial_\tau -iS^*$ to~\eqref{eq:decomp-rage-0} and using the equation~\eqref{eq:Utau-eqnflowlim} satisfied by the limit evolution, we find in the weak sense
\[(\partial_\tau-iS^*) R(\tau)g^\circ\,=\,-i\sum_k e^{i\tau\lambda_k}\,(\lambda_k-S^*)\,V\hat P_kV^* g^\circ.\]
Given some $k_0$, multiplying both sides of this identity by $e^{-i\tau\lambda_{k_0}}$ and integrating over some time interval $[0,T]$, with $T>0$, we deduce in the weak sense
\begin{multline*}
\frac1{iT}\Big(e^{-iT\lambda_{k_0}} R(T)g^\circ-R(0)g^\circ\Big)
+\frac1T\int_0^Te^{-i\tau\lambda_{k_0}}(\lambda_{k_0}-S^*) R(\tau)g^\circ\,d\tau\\
\,=\,
-(\lambda_{k_0}-S^*)\,V\hat P_{k_0}V^*\bar g^\circ
-\sum_{k:k\ne k_0}\frac{e^{iT(\lambda_k-\lambda_{k_0})}-1}{iT(\lambda_k-\lambda_{k_0})}\,(\lambda_k-S^*)\,V\hat P_kV^* g^\circ.
\end{multline*}
Now testing this identity with some $h\in\Cc$, and singling out the first right-hand side term, we deduce
\begin{multline}\label{eq:bound-eigen-er}
\big|\big\langle (\lambda_{k_0}-S)h,V\hat P_{k_0}V^* g^\circ\big\rangle_\Hc\big|\\
\,\le\,
\frac1{T}\|h\|_\Hc\Big(\|R(T)g^\circ\|_\Hc+\|R(0)g^\circ\|_\Hc\Big)
+\frac1T\int_0^T\big|\big\langle (\lambda_{k_0}-S)h, R(\tau)g^\circ\big\rangle_\Hc\big|\\
+\sum_{k:k\ne k_0}\bigg|\frac{e^{iT(\lambda_k-\lambda_{k_0})}-1}{iT(\lambda_k-\lambda_{k_0})}\bigg|\|\hat P_kV^*(\lambda_k-S)h\|_\Hc\|\hat P_kV^* g^\circ\|_\Hc.
\end{multline}
By~\eqref{eq:form-Rg-re}, we have $\|R(\tau)g^\circ\|_{\Hc}\le\|g^\circ\|_{\Hc}$,
and the first right-hand side term in~\eqref{eq:bound-eigen-er} thus converges to $0$ as $T\uparrow0$. By~\eqref{eq:conv-Gbarg} and Jensen's inequality, the second right-hand side term also converges to $0$. For the last right-hand side term, using $|e^{ix}-1|\le2\wedge|x|$, and decomposing $(\lambda_k-S)h=(\lambda_k-\lambda_{k_0})h+(\lambda_{k_0}-S)h$, we find
\begin{eqnarray*}
\lefteqn{\sum_{k:k\ne k_0}\bigg|\frac{e^{iT(\lambda_k-\lambda_{k_0})}-1}{iT(\lambda_k-\lambda_{k_0})}\bigg|\|\hat P_kV^*(\lambda_k-S)h\|_\Hc\|\hat P_kV^* g^\circ\|_\Hc}\\
&\le&
\tfrac2T\sum_{k}\|\hat P_kV^*h\|_\Hc\|\hat P_kV^* g^\circ\|_\Hc\\
&&+\sum_{k:k\ne k_0}\big(1\wedge\tfrac2{T|\lambda_k-\lambda_{k_0}|}\big)\|\hat P_kV^*(\lambda_{k_0}-S)h\|_\Hc\|\hat P_kV^* g^\circ\|_\Hc,
\end{eqnarray*}
which converges to $0$ as $T\uparrow\infty$ by dominated convergence. Going back to~\eqref{eq:bound-eigen-er}, we conclude for all $h\in\Cc$,
\begin{equation}\label{eq:comput-Pk-eigen}
\big\langle (\lambda_{k_0}-S)h,V\hat P_{k_0}V^* g^\circ\big\rangle_\Hc\,=\,0,
\end{equation}
which precisely means that $V\hat P_{k_0}V^* g^\circ$ is an eigenvector of $S^*$ with eigenvalue $\lambda_{k_0}$. 

Finally, let us further assume $S_N\Cc\subset\Cc$, $S\Cc\subset\Cc$, and that the squared operator~$S_N^2$ also converges in the strong operator topology to $S^2$ on $\Cc$. This easily implies that for $g^\circ\in\Cc$ the limit evolution $\tau\mapsto U(\tau)g^\circ$ belongs to $C^2_b(\R^+;\Hc)$.
By Lemma~\ref{lem:unique-S}, we then learn that this limit evolution is the unique strong solution of equation~\eqref{eq:Utau-eqnflowlim} in $C^2_b(\R^+;\Hc)$.
Moreover, in terms of the canonical self-adjoint extension $\Lc_0$ of the squared operator $S^2$ on $\Cc$ as given by Friedrichs' theorem, the proof of Lemma~\ref{lem:unique-S} shows that for $g^\circ\in\Cc$ the limit evolution $\tau\mapsto U(\tau)g^\circ$ coincides with the unique strong solution of
\[\partial_\tau^2g+\Lc_0g=0,\qquad g|_{\tau=0}=g^\circ,\qquad \partial_\tau g|_{\tau=0}=iS^*g^\circ.\]
For all $k$, a straightforward adaptation of the proof of~\eqref{eq:comput-Pk-eigen} above then leads us to conclude that $V\hat P_kV^*g^\circ$ is also an eigenvector of $\Lc_0$ with eigenvalue $\lambda_k^2$. As $\Lc_0$ is self-adjoint, this entails that $V\hat P_kV^*g^\circ$ is orthogonal to $V\hat P_lV^*g^\circ$ whenever $\lambda_k^2\ne\lambda_l^2$. This ends the proof of Proposition~\ref{th:RAGE}.
\end{proof}

\section*{Acknowledgements}
The authors thank Jean-Baptiste Fouvry and Pierre-Henri Chavanis for motivating discussions and for pointing out relevant physics references.
M. Duerinckx acknowledges financial support from the F.R.S.-FNRS, as well as from the European Union (ERC, PASTIS, Grant Agreement n$^\circ$101075879).\footnote{{Views and opinions expressed are however those of the authors only and do not necessarily reflect those of the European Union or the European Research Council Executive Agency. Neither the European Union nor the granting authority can be held responsible for them.}} P.-E. Jabin was partially supported by NSF DMS Grants 2205694 and 2219297.

\def\cprime{$'$}\def\cprime{$'$} \def\cprime{$'$}

\end{document}